\documentclass{article}
\RequirePackage[OT1]{fontenc}
\RequirePackage{amsthm,amsmath}
\usepackage{hyperref}
\usepackage{amsmath}
\usepackage{graphicx}
\usepackage{multirow}
\usepackage{subfigure}
\usepackage{subcaption}
\usepackage{lipsum}
\usepackage{setspace}
\usepackage{geometry}
\usepackage[style = authoryear, sorting = nty, isbn = false, doi = false, url = false, eprint = false, maxbibnames=99, citestyle=authoryear-comp, maxcitenames = 2]{biblatex}
\renewbibmacro{in:}{}
\newtheoremstyle{break}
  {\topsep}{\topsep}%
  {\itshape}{}%
  {\bfseries}{}%
  {\newline}{}%
\theoremstyle{break}

\newtheorem{theorem}{Theorem}
\newtheorem{remark}{Remark}
\newtheorem{definition}{Definition}

\newtheorem{lemma}{Lemma}
\newtheorem{example}{Example}

\newtheorem{algorithm}{Algorithm}

\title{Statistical inference of high-dimensional vector autoregressive time series with non-i.i.d. innovations}
\author{Yunyi Zhang}
\date{\begin{small}
School of Data Science, The Chinese University of Hong Kong, Shenzhen, Guangdong, 518172, China\\
zhangyunyi@cuhk.edu.cn
\end{small}
}

\DeclareMathOperator*{\argmin}{\arg\!\min}
\begin{document}
\maketitle
\abstract{
Independent or i.i.d. innovations is an essential assumption in the literature for analyzing a  vector time series. However, this assumption is either too restrictive for a real-life time series to satisfy or is hard to  verify through a hypothesis test. This paper performs statistical inference on a sparse high-dimensional vector autoregressive time series, allowing its white noise innovations to be dependent, even non-stationary. To achieve this goal, it adopts a post-selection estimator to fit the vector autoregressive model and derives the asymptotic distribution of the post-selection estimator. The innovations in the autoregressive time series are not assumed to be independent, thus making the covariance matrices of the autoregressive coefficient estimators complex and difficult to estimate. Our work develops a bootstrap algorithm to facilitate practitioners in performing statistical inference without having to engage in sophisticated calculations. Simulations and real-life data experiments reveal the validity of the proposed methods and theoretical results.

Real-life data is rarely considered to exactly satisfy an autoregressive model with independent or i.i.d. innovations, so our work should better reflect the reality compared to the literature that assumes i.i.d. innovations.

\textbf{Key words: } Vector autoregressive time series, Non-stationary time series, High-dimensional data, Bootstrap algorithm, Post-selection estimation.
}
\section{Introduction}
The vector autoregressive time series with order $p$(VAR(p)) of the form
\begin{equation}
\mathbf{x}^{(t)} = \sum_{j = 1}^p\mathbf{A}^{(j)}\mathbf{x}^{(t - j)} + \boldsymbol\epsilon^{(t)},\  \mathbf{x}^{(t)},\boldsymbol\epsilon^{(t)}\in\mathbf{R}^d,\ \mathbf{A}^{(j)}\in\mathbf{R}^{d\times d},\ t\in\mathbf{Z},
\label{eq.AR_model}
\end{equation}
where the innovations $\boldsymbol\epsilon^{(t)}$ are white noises, is one of the fundamental and most important models for
revealing  spatial and temporal dependence in data. VAR model is always considered to be a direct generalization of the
scalar autoregressive process introduced in \cite{DPbook}. It finds wide applications in various fields, such as engineering, finance and economics, biological and environmental
studies, among others. To mention a few, \cite{MESSNER20191485} and  \cite{HANNAFORD2023107659} adopted the VAR model for forecasting wind power as well as  analyzing microbial dynamics in a wastewater treatment plant, respectively. \cite{RePEc:tpr:restat:v:68:y:1986:i:4:p:628-37}, \cite{10.2307/2328190}, and \cite{10.1093/icc/dtq018} used the VAR model to explain some economy phenomena such as fluctuations in the U.S. -- Canadian exchange rate; while \cite{Varforcast} and \cite{HE2021105171} employed the vector autoregressive model to study temperature changes.

The VAR model enjoys many favourable properties. Compared to the more sophisticated models such as vector autoregressive
and moving average(VARMA) time series models, state space models introduced in \cite{MR1238940}, and nonlinear time series models like the recurrent neural network presented in \cite{10.1162/neco.1997.9.8.1735}, the VAR model has relatively few parameters to estimate. Moreover, it is easy to fit using linear regression methods or the Yule--Walker equation, as demonstrated in the work of \cite{MR3357870}, \cite{MR3450535}, \cite{liu2021highdimensional}, \cite{MR4561044}, and the reference therein. Furthermore, according to
\cite{DIAS201875} and \cite{MR4480716}, the VAR model does not introduce identification issues and is easy to interpret. Therefore, despite advances in time series modelling, research on VAR models continues to play an important roles in statistics literature.

Statisticians always need to impose certain technical assumptions when conducting statistical inference on vector autoregressive time series.
Two major assumptions in the literature, as discussed in \cite{MR1093459, MR3526245} and
\cite{MR1238940}, are the ``low-dimensional'' assumption, which requires that the number of parameters needed to be estimated is small compared to the sample size, and the assumption that the innovations ($\boldsymbol\epsilon^{(t)}$ in eq.\eqref{eq.AR_model}) are independent and identically distributed(i.i.d.). Especially, the assumption of i.i.d. innovations ensures that the VAR time series is strictly stationary.  However, these two assumptions often fail in the modern era, where the observed time series may have many dimensions, and their distributions may not be the same. To illustrate why the low-dimensional assumption fails, suppose the observed data $\mathbf{x}^{(t)}\in\mathbf{R}^{50}$. If statisticians fit the data with a VAR(3) model, then they need to estimate a  total of $3\times 50\times 50 = 7500$ parameters, which sometimes can be even greater than the sample size.

Another issue that needs careful reconsideration is the assumption of i.i.d. innovations, which implies strict stationarity in VAR time series data(see \cite{MR1093459} for an illustration). When the observed time series spans a short time period, assuming i.i.d. innovations, which renders the VAR time series strict stationary, is reasonable. However, as demonstrated by \cite{MR4270034}, if a time series covers a long time span, assuming that the distributions of the innovations do not evolve with time is often unrealistic.

Our work focuses on statistical inference under the setting that eq.\eqref{eq.AR_model} holds true with the white noise innovations $\boldsymbol\epsilon^{(t)}$, i.e.,
$$
\mathbf{E}\boldsymbol\epsilon^{(t)} = 0\ \text{and }
\mathbf{E}\boldsymbol\epsilon^{(t_1)}\boldsymbol\epsilon^{(t_2)\top} = \begin{cases}
0\ \text{if } t_1\neq t_2\\
\boldsymbol \Sigma_{\boldsymbol\epsilon}\ \text{if } t_1 = t_2
\end{cases}.
$$
However, $\boldsymbol\epsilon^{(t)}$ can be dependent, and the distributions of $\boldsymbol\epsilon^{(t)}$ may vary with respect to the time index $t$. Example \ref{example.ar_noniid} provides an illustrative example. By adopting this setting, the second--order structures of the time series, including its covariance matrix, cross -- covariance matrices, autoregressive parameters, and spectral density, remain unchanged. In other words, the observed time series $\mathbf{x}^{(t)}$ is weakly stationary but not necessarily strictly stationary. Especially, as demonstrated in Example \ref{example.ar_noniid}, the distributions of the corresponding estimators can be quite different from the i.i.d. innovation situations. \cite{zhang2023simultaneous} and \cite{MR4025739} considered this problem for scalar time series, but to the best of our knowledge, we have not yet seen many discussions on vector, especially high--dimensional VAR models with non-i.i.d. innovations.

The advantages of the aforementioned setting can be viewed from two perspectives. On the one hand, according to the Wold representation presented in \cite{MR1238940} and similar to the work of \cite{MR2893863}, a wide range of nonlinear stationary time series possesses an infinite order VAR representation  with white noise (but not necessarily i.i.d.) innovations, and eq.\eqref{eq.AR_model} can be considered as a truncated approximation of this representation. Therefore, by allowing non-i.i.d. innovations, the
applicability of eq.\eqref{eq.AR_model} extends beyond linear time series. On the other hand, from a practical point of view, while statisticians can test whether the innovations in a time series are white noises using tests like the Portmanteau test introduces in \cite{MR3235390} and \cite{ESCANCIANO2009140}; testing for the i.i.d. nature of innovations is generally hard, even for scalar time series. Furthermore, as illustrated in example \ref{example.ar_noniid}, the variance of the autoregressive coefficient matrix estimator may vary significantly even when time series exhibit the same second-order structures. Therefore, the assumption of i.i.d. innovations has a substantial impact on the distribution of the estimated autoregressive parameters. By allowing non-i.i.d. innovations in eq.\eqref{eq.AR_model}, statisticians only need to test the white noise innovation assumption before fitting the VAR model, resulting in reliable statistical inference.

\begin{example}
Suppose the data $\mathbf{x}^{(t)}\in\mathbf{R}^5, t\in\mathbf{Z}$ satisfy the VAR(1) model $\mathbf{x}^{(t)} = \mathbf{A}\mathbf{x}^{(t - 1)} + \boldsymbol{\Theta\gamma}^{(t)}$, here
$$
\mathbf{A} =
\left[
\begin{matrix}
0.6 & 0.4 & 0 & -0.4 & 0.4\\
-0.4 & 0.6 & 0.4 & -0.4 & 0.4\\
0    & -0.4 & 0.6 & -0.4 & 0.4\\
0    & 0    & -0.4 &-0.4 & 0.4\\
\end{matrix}
\right]\ \text{and } \boldsymbol\Theta =
\left[
\begin{matrix}
1.0 & 0.5 & 0 & 0 & 0\\
-0.5 & 1.0 & 0.5 & 0 & 0\\
0    & -0.5 & 1.0 & 0.5 & 0\\
0    & 0    & -0.5 &1.0 & 0.5\\
\end{matrix}
\right]
$$
Consider three types of $\boldsymbol\gamma^{(t)}$, i.e., for $t\in\mathbf{Z}$ and $i = 1,\cdots, 5$,

\textit{Independent: } $\boldsymbol\gamma^{(t)}_i = \mathbf{e}^{(t)}_i$.

\textit{Product normal: } $\boldsymbol\gamma^{(t)}_i = \mathbf{e}^{(t)}_i\times \mathbf{e}^{(t - 1)}_i\times \mathbf{e}_i^{(t - 2)}$.

\textit{Non-stationary: } $\boldsymbol\gamma_i^{(t)} = \mathbf{e}_i^{(t)}$ for $t = 2k + 1$ and $\boldsymbol\gamma_i^{(t)} = \mathbf{e}_i^{(t)}\mathbf{e}_i^{(t - 1)}$ for $t = 2k$.

Here $\mathbf{e}_i^{(t)}$ are i.i.d. normal random variables with mean $0$ and variance $1$. In this setup, all $\boldsymbol \gamma_i^{(t)}$ are white noises, meaning that $\mathbf{E}\boldsymbol\gamma^{(t)}_i = 0$; $\mathbf{E}\boldsymbol\gamma^{(t_1)}_i\boldsymbol\gamma^{(t_2)}_j = 0$ if $t_1\neq t_2$ or $i\neq j$; and $\mathbf{E}\boldsymbol\gamma^{(t)2}_i = 1$. However, the \textit{product normal} innovations are dependent, and the \textit{non-stationary } innovations are non-stationary.
Table \ref{example.ar_noniid} demonstrates the variance of the estimator for $\mathbf{A}_{12}$. The estimator remains consistent for different kinds of innovations. However, the variance of the estimator is affected when non-i.i.d. innovations are introduced.
\begin{table}[htbp]
\caption{Estimator $\widehat{\mathbf{A}}_{12}$ for $\mathbf{A}_{12}$ and the variance of $\widehat{\mathbf{A}}_{12}$. We use Monte-Carlo method to derive the variance of $\widehat{\mathbf{A}}_{12}$. The sample size is $10000$ and the simulation times is $5000$.}
\centering
\begin{tabular}{l l l}
\hline\hline
Innovation type  &  Estimate & The variance $Var(\sqrt{T}(\widehat{\mathbf{A}}_{12} - \mathbf{A}_{12}))$\\
\hline
Independent      &  0.40     & 0.20\\
Product normal   &  0.41     & 0.36\\
Non-stationary   &  0.40     & 0.21\\
\hline\hline
\end{tabular}
\end{table}
\label{example.ar_noniid}
\end{example}

\textbf{Related literature: }  There have been several attempts to relax the ``low-dimensional'' assumption, which involves performing statistical inference on a VAR model without assuming a small number of parameters to estimate. One approach is to impose structural assumptions on the autoregressive coefficient matrix $\mathbf{A}^{(j)}$ to account for a large number of parameters. For example, statisticians may assume that $\mathbf{A}^{(j)}$ is a sparse matrix, where only a few parameters are not zero. The work of \cite{MR4325665} made use of this setting to derive consistent estimators and performed valid statistical inference. Once the sparsity assumption is adopted, statisticians can leverage the power of Lasso, as introduced by \cite{MR1379242} and \cite{MR2274449}, to recover the underlying sparsity pattern of $\mathbf{A}^{(j)}$ and obtain a satisfactory estimator. Readers can refer to \cite{MR3357870}, \cite{MR3343790}, \cite{MR4102690}, \cite{MR4278792}, and \cite{MR4446423}
for a detailed introduction to Lasso estimation. Apart from Lasso, \cite{MR2847973} proposed
a method to estimate the precision matrix of high-dimensional data, and \cite{MR3450535} utilized this method along with the penalized Yule-Walker equation to estimate $\mathbf{A}^{(j)}$. \cite{liu2021highdimensional} developed a debiased estimator based on a convex loss function. \cite{MR4278792} derived the asymptotic distribution of the desparsified Lasso estimator and developed a bootstrap algorithm to assist statistical inference, allowing statisticians to conduct statistical inference without resorting to complex calculations. We also refer the readers to \cite{MR2732601}, \cite{MR3346699}, and  \cite{MR4278792} regarding bootstrapping a high-dimensional time series.

In addition to the general sparsity assumption, statisticians have explored various specific structural assumptions when applying the VAR model to high-dimensional time series data
in the literature. To mention a few examples, \cite{MR4480716} assumed that the coefficient matrices $\mathbf{A}^{(j)}$ have low rank and performed reduced-rank regression in estimation; \cite{MR3620446}, \cite{MR4259144}, \cite{MR3725456}, and \cite{MR3662449}, respectively imposed the banded structure,  linear restrictions, block structure, and network structure on the autoregressive coefficient matrices. \cite{MR4561044} relaxed the sparsity assumption to the so-called ``approximately sparse matrices'' in statistical inference; also see \cite{MR2485008} for a further introduction.

In addition to relaxing the low-dimensional assumption, another research direction aims to remove the assumption of i.i.d. innovations, thus allowing for time series to be weakly stationary or non-stationary. \cite{MR2893863} and \cite{MR3796524} extended the validity of bootstrap-based inference to general time series. \cite{MR3331856}, \cite{MR4270034}, and \cite{MR4325666}, respectively discussed prediction and autoregressive spectral estimation for general time series. \cite{MR2190207}, \cite{MR4025739}, and \cite{zhang2023simultaneous} performed statistical inference on a scalar autoregressive model without assuming that the innovations are i.i.d. random variables.
\cite{ding2023autoregressive} proved that a certain type of non-stationary time series can be approximated by an autoregressive model with independent innovations. When the innovations in the autoregressive model \eqref{eq.AR_model} are not i.i.d., the observations $\mathbf{x}^{(t)}$ become non-stationary, meaning the marginal distributions of $\mathbf{x}^{(t)}$ are not identical. Statisticians may refer to the non-stationary time series literature, such as \cite{MR2827528}, \cite{MR3299408}, \cite{MR3920364}, \cite{MR4134802}, \cite{MR4091095}, \cite{MR4270034}, and \cite{zhang2023simultaneous}, to analyze the time series $\mathbf{x}^{(t)}$ under this circumstance.

One recent trend in the literature tries to simultaneously address the high-dimensionality of parameters and the non-stationarity present in vector time series. Examples of such research include \cite{MR3779697}, \cite{chang2023central}, and \cite{mies2022sequential} for inference on sample mean; \cite{MR4206676} for covariance and spectral density estimation; \cite{doi:10.1080/01621459.2022.2161386} for inference on functional-coefficient autoregressive model. However, to the best of our knowledge, most of the research focuses on the inference of first-order statistics, such as the sample mean, but does not consider statistical inference on second-order statistics, like the sample autocovariances and autoregressive coefficients. In particular, there have been few works conducted on statistical inference of autoregressive time series with white noise innovations(which are not necessarily i.i.d.). This paper aims to address this gap in the literature.

\textbf{Our contribution: } Our work assumes that the white noise innovations $\boldsymbol{\epsilon}^{(t)}$ in eq.\eqref{eq.AR_model} satisfy the $(m,\alpha)-$short range dependent condition proposed in \cite{zhangPolitis} (also refer to \cite{MR3779697}), and the coefficient matrices $\mathbf{A}^{(j)}$ are sparse matrices. Under these assumptions, it employs the post-selection ``Lasso + OLS'' estimator introduced by \cite{MR3151764} to estimate $\mathbf{A}^{(j)}$ and establishes the (model-selection) consistency for the estimator. Furthermore, the paper derives the Gaussian approximation theorem for the two-stage estimator.

In the literature, such as \cite{MR1391173}, the innovations in eq.\eqref{eq.AR_model} are assumed to be i.i.d., so the asymptotic covariance matrix of the Yule-Walker estimator depends solely on the second-order moments of the data and has a closed form. However, \cite{zhang2023simultaneous} shows that this is not true when the innovations are dependent. Moreover, in the case of dependent innovations, the covariance matrix of the estimator can be too complex to have a closed-form expression, hindering statisticians from determining the width of the estimator's confidence interval. To solve this issue, this paper generalizes the second-order wild bootstrap introduced in \cite{zhang2023simultaneous} to vector time series. By employing the bootstrap algorithm, the width of the estimator's confidence interval can be derived without the need for manual calculations.

\textbf{The structure of the remaining parts: }The remainder of the paper is organized as follows: Section \ref{section.Setting} introduces the post-selection Lasso+OLS estimator for the coefficient matrices $\mathbf{A}^{(j)}$ along with the corresponding second-order wild bootstrap algorithm. Section \ref{section.m_alpha_dependent} reviews the properties of $(m,\alpha)-$short range dependent random variables.  Section \ref{section.Theoretical} presents the consistency and Gaussian approximation results for the estimator. Section \ref{section.numerical}
conducts numerical experiments to verify the finite-sample performance of the estimators.  Additionally, it explores researches the mutual relations of electricity generation in ASEAN countries using the proposed methods. Section \ref{section.conclusion} makes conclusions.  All proofs, technical calculations, as well as additional numerical simulation results, are postponed to the online Supplementary
Material.

\textbf{Notation: } This paper uses the standard order notation $O(\cdot), o(\cdot), O_p(\cdot)$, and $o_p(\cdot)$: For two numerical sequences $a_n, b_n, n = 1,2,\cdots$, we say $a_n = O(b_n)$ if there exists a constant $C>0$ such that $\vert a_n\vert\leq C\vert b_n\vert$ for $\forall n\in\mathbf{N}$; and $a_n = o(b_n)$ if $\lim_{n\to\infty} a_n/b_n = 0$. For two random variable sequences $X_n, Y_n$, we say $X_n = O_p(Y_n)$ if for any given $0 < \varepsilon < 1$, there exists a constant $C_\varepsilon>0$ such that $Prob\left(\vert X_n\vert\leq C_\varepsilon\vert Y_n\vert\right)\geq 1 - \varepsilon$ for any $n$; and $X_n = o_p(Y_n)$ if $X_n / Y_n\to_p 0$ where the latter denotes convergence in probability. Readers can refer to \cite{MR2002723} for a detailed introduction. For a random variable $X\in\mathbf{R}$, define its $m$ norm $\Vert X\Vert_m = \mathbf{E}(\vert X\vert^m)^{1/m}$, here $m\geq 1$. We use the notation $\wedge$ and $\vee$ to represent `minimum' and `maximum' respectively, i.e., $a\wedge b = \min(a,b)$ and $a\vee b = \max(a,b)$. We use $C,C_0,C_1,\cdots$, to represent general constants. However, their value may change from line to line.

In the remaining parts of this paper, bold lowercase (Greek) letters, such as $\mathbf{a}$ or $\boldsymbol \omega$, represent vectors; uppercase (Greek) letters, such as $\mathbf{A}$ or $\boldsymbol \Sigma$, represent matrices. Subscripts are used to represent elements in a vector or a matrix. For example, a vector $\mathbf{a}\in\mathbf{R}^n = (\mathbf{a}_1,\cdots, \mathbf{a}_n)^T$ and a matrix $\mathbf{A}\in\mathbf{R}^{m\times n} = \{\mathbf{A}_{ij}\}_{i = 1,\cdots, m, j = 1,\cdots, n}$. In addition, for a matrix $\mathbf{A}\in\mathbf{R}^{m\times n}$, define its $i$th row $\mathbf{A}_{i\cdot} = (\mathbf{A}_{i1},\cdots, \mathbf{A}_{in})$ and $j$th column $\mathbf{A}_{\cdot j} = (\mathbf{A}_{1j},\cdots, \mathbf{A}_{mj})^T$. We use the symbol ``$\top$'' to represent matrix transpose.

For a vector $\mathbf{a}\in\mathbf{R}^n$, define its $p$ norm ($1 \leq p< \infty$) to be $\vert\mathbf{a}\vert_p = \left(\sum_{i = 1}^n\vert\mathbf{a}_i\vert^p\right)^{1/p}$ and $\vert\mathbf{a}\vert_\infty = \max_{i = 1,\cdots, n}\vert\mathbf{a}_i\vert$. We also define $\vert\mathbf{a}\vert_0 = \sum_{i = 1}^n\mathbf{1}_{\mathbf{a}_i\neq 0}$, i.e., the number of non-zero elements in $\mathbf{a}$. However, $\vert\cdot\vert_0$ is not a norm since $\vert c\mathbf{a}\vert_\infty\neq c\vert\mathbf{a}\vert_\infty$ for a number $c\neq 0$. For a matrix $\mathbf{A}\in\mathbf{R}^{n\times n}$, define its spectral norm $\vert\mathbf{A}\vert_2 = \max_{\vert\mathbf{a}\vert_2 = 1}\vert\mathbf{A}\mathbf{a}\vert_2$. For a matrix $\mathbf{A}\in\mathbf{R}^{n\times m}$, define its maximum row sum norm $\vert\mathbf{A}\vert_{L_1} = \max_{i = 1,\cdots, n}\sum_{j = 1}^m\vert\mathbf{A}_{ij}\vert$ (see \cite{MR2978290}) and its infinity norm $\vert\mathbf{A}\vert_\infty = \max_{i = 1,\cdots, n, j = 1,\cdots, m}\vert\mathbf{A}_{ij}\vert$. For a matrix $\mathbf{A}\in\mathbf{R}^{n\times n}$ and a set $\mathcal{K} = \{1\leq k_1 < k_2<\cdots < k_s\leq n\}$, define the submatrix
\begin{align*}
\mathbf{A}_{\mathcal{K}} = \left[
\begin{matrix}
\mathbf{A}_{k_1k_1} & \mathbf{A}_{k_1k_2} & \cdots & \mathbf{A}_{k_1k_v}\\
\mathbf{A}_{k_2k_1} & \mathbf{A}_{k_2k_2} & \cdots & \mathbf{A}_{k_2k_v}\\
\vdots              & \vdots              & \cdots & \vdots\\
\mathbf{A}_{k_vk_1} & \mathbf{A}_{k_vk_2} & \cdots & \mathbf{A}_{k_vk_v}\\
\end{matrix}
\right],
\end{align*}
i.e., the matrix formed by selecting the $k_v$th rows and columns in matrix $\mathbf{A}$ where $k_v\in\mathcal{K}$.

\section{Setting and methodology}
\label{section.Setting}
Suppose the observations $\mathbf{x}^{(t)}\in\mathbf{R}^d, t = 1,\cdots, T$ are stemmed from the autoregressive process \eqref{eq.AR_model}.
Following the time series literature, we call the random vector $\boldsymbol{\epsilon}^{(t)}$ ``the innovations'',
and suppose that the innovations are white noises, meaning that
$\mathbf{E}\boldsymbol\epsilon^{(t)} = 0$, $\mathbf{E}\boldsymbol\epsilon^{(t)}\boldsymbol\epsilon^{(t)\top} = \boldsymbol\Sigma_{\boldsymbol\epsilon}$, and
$\mathbf{E}\boldsymbol\epsilon^{(t_1)}\boldsymbol\epsilon^{(t_2)\top} = 0$ for any $t_1\neq t_2$. However, we do not assume independence among the innovations.
This paper aims to establish a simultaneous confidence interval for the coefficient matrix $\mathbf{A}^{(j)}, j = 1,\cdots, p$, or to perform the exact test
$$
H_0:\ \mathbf{A}^{(j)} = \mathbf{A}^{(j)\dagger}\ \text{for } j = 1,\cdots, p\ \text{versus } H_1: \exists j\in\{1,\cdots, p\}\ \text{such that } \mathbf{A}^{(j)} \neq \mathbf{A}^{(j)\dagger}.
$$
Here $\mathbf{A}^{(j)\dagger}$ are given matrices, and there exists $j_0\in\{1,\cdots,p\}$ such that $\mathbf{A}^{(j_0)\dagger}\neq 0$. Define the matrices
\begin{align*}
\mathbf{W} = \left[
\begin{matrix}
\mathbf{x}^{(p)\top} & \mathbf{x}^{(p - 1)\top} & \cdots & \mathbf{x}^{(1)\top}\\
\mathbf{x}^{(p + 1)\top} & \mathbf{x}^{(p)\top} & \cdots & \mathbf{x}^{(2)\top}\\
\vdots                & \vdots            & \cdots & \vdots           \\
\mathbf{x}^{(T - 1)\top} & \mathbf{x}^{(T - 2)\top} & \cdots & \mathbf{x}^{(T - p)\top}\\
\end{matrix}
\right],\ \mathbf{S} = \left[
\begin{matrix}
\mathbf{A}^{(1)\top}\\
\mathbf{A}^{(2)\top}\\
\vdots\\
\mathbf{A}^{(p)\top}\\
\end{matrix}
\right]
\end{align*}
and the vectors
\begin{align*}
\mathbf{y}^{(l)} = (\mathbf{x}^{(p + 1)}_{l}, \mathbf{x}^{(p + 2)}_{l},\cdots, \mathbf{x}^{(T)}_l)^\top,\ \boldsymbol\eta^{(l)} = (\boldsymbol\epsilon^{(p + 1)}_l,\cdots, \boldsymbol\epsilon_l^{(T)})^\top\ \text{for } l = 1,\cdots, d,
\end{align*}
then eq.\eqref{eq.AR_model} implies
\begin{equation}
\mathbf{y}^{(l)} = \mathbf{W}\mathbf{S}_{\cdot l} + \boldsymbol\eta^{(l)},\ \text{equivalently } \mathbf{y}^{(l)}_{k} = \sum_{j = 1}^{pd}\mathbf{W}_{kj}\mathbf{S}_{jl} + \boldsymbol\eta^{(l)}_k\ \text{for } k = 1,\cdots, T - p,
\label{eq.linear_regression_form}
\end{equation}
which satisfies a linear model. With the help of this representation, statisticians can leverage the linear regression literature, such as the works by \cite{MR1946581}, \cite{MR3102549}, and \cite{MR3153940}, to estimate $\mathbf{S}$. However, many linear regression methods, along with the debiased techniques, rely on inverting the (adjusted) hat matrix $\mathbf{W}^\top\mathbf{W}$, and require that the (inversion of the) hat matrix follows certain sparse constraints, which is unrealistic and hard for a VAR model to satisfy.

Our work employs a  post-selection Lasso + OLS (referred to as post-selection) estimator to avoid directly inverting the singular hat matrix.
Following the method introduced by \cite{MR3151764},
the post-selection estimator uses Lasso to select the indices ordinate corresponding to non-zero elements in each column $\mathbf{S}_{\cdot l}$; then it fits  $\mathbf{y}^{(l)} $ on the selected indices through least-square estimation. The sparsity assumption ensures that only a few indices $\widetilde{\mathcal{B}}_l$ in each column  $\mathbf{S}_{\cdot l}$ are selected. After selection, statisticians can invert the sub-matrix $(\frac{1}{T}\mathbf{W}^\top\mathbf{W})_{\widehat{\mathcal{B}}_l}$ to calculate the least-square estimator. Since the number of non-zero indices in $\mathbf{S}_{\cdot l}$ is very small compared to the sample size $T$, the sub-matrix $(\frac{1}{T}\mathbf{W}^\top\mathbf{W})_{\widehat{\mathcal{B}}_l}$ is always invertible, ensuring that the post-selection estimator is well-defined.

The post-selection estimation follows a two-step procedure. In the first step, we define the Lasso estimator $\widetilde{\mathbf{S}}$ whose columns satisfy
\begin{equation}
\widetilde{\mathbf{S}}_{\cdot l} = \argmin_{\mathbf{s}\in\mathbf{R}^{pd}} \frac{1}{2T}\vert \mathbf{y}^{(l)} - \mathbf{W}\mathbf{s}\vert_2^2 + \lambda \vert\mathbf{s}\vert_1.
\label{eq.use_lasso}
\end{equation}
Our work uses the same $\lambda$ for different $l$ in eq.\eqref{eq.use_lasso} for simplicity and to avoid tuning too many parameters. However, this constraint is unnecessary. Statisticians may choose different $\lambda$ for different $l$ when performing Lasso.

Calculating the Lasso estimator relies on numerical optimization algorithms, and the computation errors inherent in these algorithms
always leave $\widetilde{\mathbf{S}}_{\cdot l}$ to have few exactly $0$ elements
in spite of the sparsity assumption. We demonstrate this in Table \ref{table.numerical_res}.
To maintain sparsity in the estimator, we introduce a threshold $b_T$ and select the elements in $\widetilde{\mathbf{S}}_{\cdot l}$ whose absolute
values are greater than $b_T$, forming the following sets
\begin{equation}
\widetilde{\mathcal{B}}_l = \left\{j = 1,\cdots, pd: \vert \widetilde{\mathbf{S}}_{j l}\vert > b_T\right\}.
\label{eq.def_Bl}
\end{equation}
Before moving to the second step, we present the partial inverse operator $\mathcal{F}_{\cdot}(\cdot)$, whose mathematical properties will be explored in the online supplement.

\begin{definition}[Partial inverse operator $\mathcal{F}_{\cdot}(\cdot)$]
Suppose a matrix $\mathbf{A}\in\mathbf{R}^{d \times d}$ and a set $\mathcal{B} = \{1\leq k_1 < k_2 < \cdots < k_v\leq d\}$, do the following operations:

i. Choose the $k_s, s = 1,\cdots, v$th row and column of $\mathbf{A}$ and form a new matrix $\mathbf{A}_{\mathcal{B}}$.

ii. Calculate the Moore - Penrose pseudo-inverse matrix of $\mathbf{A}_{\mathcal{B}}$, and denote the matrix by $\mathbf{Z}$.

iii. Define the matrix $\mathcal{F}_{\mathcal{B}}(\mathbf{A})$ as follows:
\begin{equation}
\mathcal{F}_{\mathcal{B}}(\mathbf{A}) = \mathbf{C}\in\mathbf{R}^{d\times d},\ \text{here } \mathbf{C}_{ij} =
\begin{cases}
\mathbf{Z}_{st}\ \text{if } i = k_s\ \text{and } j = k_t\\
0\ \text{otherwise}
\end{cases}.
\label{eq.matrix_C}
\end{equation}
\label{definition.partial_inverse}
\end{definition}
As demonstrated in the online supplement, the operator $\mathcal{F}_{\cdot}(\cdot)$ satisfies
$\mathcal{F}_{\mathcal{B}}(\mathbf{A})\mathbf{A}\boldsymbol\beta = \boldsymbol\beta$ if $\boldsymbol\beta_i = 0$ for all $i\not\in\mathcal{B}$.
This property explains the reason for calling $\mathcal{F}_{\cdot}(\cdot)$ ``the partial inverse operator'', as
$\mathcal{F}_{\mathcal{B}}(\mathbf{A})$ behaves the same as an inverse matrix when acting on the sparse vector $\boldsymbol\beta$.

With the help of the partial inverse operator, we can present the post-selection inference estimator $\widehat{\mathbf{S}}$, which is demonstrated as follows:
\begin{equation}
\widehat{\mathbf{S}}_{\cdot l}  = \mathcal{F}_{\widetilde{\mathcal{B}}_{l}}(\frac{1}{T}\mathbf{W}^\top\mathbf{W})\frac{1}{T}\mathbf{W}^\top\mathbf{y}^{(l)},\ \text{here } l = 1,\cdots, d.
\label{eq.post_selection}
\end{equation}
Clearly, by replacing $\mathcal{F}_{\widetilde{\mathcal{B}}_{l}}(\frac{1}{T}\mathbf{W}^\top\mathbf{W})$ with the inverse matrix
$(\frac{1}{T}\mathbf{W}^\top\mathbf{W})^{-1}$,
eq.\eqref{eq.post_selection} becomes the usual least-square estimator.
However, the number of parameters in eq.\eqref{eq.linear_regression_form} may exceed the sample size,
causing the matrix $\mathbf{W}^\top\mathbf{W}$ to be non-invertible. Therefore, we need the partial inverse operator to make the estimator well-defined. From the definition of $\mathbf{S}$, we have the estimator for the coefficient matrix
\begin{equation}
  \left[
  \begin{matrix}
    \widehat{\mathbf{A}}^{(1)} & \widehat{\mathbf{A}}^{(2)} &\cdots & \widehat{\mathbf{A}}^{(p)} \\
  \end{matrix}
  \right] = \widehat{\mathbf{S}}^{\top},\ \text{specifically } \widehat{\mathbf{A}}^{(k)}_{ij} = \widehat{\mathbf{S}}^{\top}_{i((k-1)\times d + j)}
  = \widehat{\mathbf{S}}_{((k-1)\times d + j)i}.
  \label{eq.def_A_hat}
\end{equation}

\textbf{Bootstrap-assisted confidence interval \& hypothesis testing: } Originating from \cite{MR0515681}, the bootstrap algorithm has become a prevalent tool
for dealing with
estimators with hard-to-calculate variances. The seminal work by \cite{MR1197789, MR1310224}, among others, demonstrated the effectiveness of bootstrap algorithms
in analyzing time series. Statisticians have developed variants of bootstrap algorithms to analyze different kinds of time series and estimators. Examples include \cite{MR1466304}
and \cite{MR3346699} for linear process; \cite{MR4134800} for spectral statistics; \cite{MR3299408} for locally stationary time series. \cite{MR2656050} introduced
the dependent wild bootstrap for stationary  time series, which was later found to be useful in dealing with non-stationary time series as well, see the work of
\cite{zhangPolitis} and \cite{zhang2023simultaneous} for a further illustration.

A key challenge in the analysis of VAR time series is the presence of complicated spatial and temporal correlations. Those correlations are conveyed to the estimators of
the time series, which lead to complex variance and covariance structures in the estimators. As demonstrated by \cite{MR4561044}, even when the
innovations are i.i.d., the autoregressive coefficient estimator still exhibits complex covariance structures.
Moreover, our work does not require the innovations
to be independent and strictly stationary. Consequently,
the covariance matrix  of the post-selection estimator not only relies on the innovations' covariance matrix $\boldsymbol\Sigma_{\boldsymbol\epsilon}$, but also
depends on their fourth-order cumulants; and the cumulants of $\boldsymbol\epsilon^{(t)}$ may vary significantly for different $t$. Confronted with this issue,
our work leverages the power of bootstrap algorithm and develops a bootstrap algorithm, called ``second-order wild bootstrap'', that facilitates statisticians
in analyzing the
VAR time series. The detailed procedure of the second - order wild bootstrap is presented in Algorithm \ref{algorithm.bootstrap}.

\begin{algorithm}[Second-order wild bootstrap]
\label{algorithm.bootstrap}
  \textbf{Input: } Observations $\mathbf{x}^{(t)}\in\mathbf{R}^d, t = 1,\cdots, T$ stemmed from eq.\eqref{eq.AR_model}; the Lasso parameter $\lambda$; the threshold $b_T$; the kernel bandwidth $k_T$; and a kernel function $K(\cdot)$; the nominal coverage probability $ 1 - \alpha$; the number of bootstrap replicates $B$.

  \textbf{Additional input for exact test: } the expected parameter matrices $\mathbf{A}^{(j)\dagger}, j  = 1,\cdots, p$.

  1. Derive the Lasso estimator $\widetilde{\mathbf{S}}$ as in eq.\eqref{eq.use_lasso}, the set $\widetilde{\mathcal{B}}_l$, the post-selection estimator $\widehat{\mathbf{S}}$ as in eq.\eqref{eq.post_selection}, and the estimator $\widehat{\mathbf{A}}_{ij}^{(k)}$ as in eq.\eqref{eq.def_A_hat}. After that, concatenate $\mathbf{x}^{(t)}$ into
  \begin{equation}
  \mathbf{z}^{(t)} = (\mathbf{x}^{(t)\top}, \mathbf{x}^{(t - 1)\top}, \cdots, \mathbf{x}^{(t - p + 1)\top})^\top\ \text{for } t = p, p+1,\cdots, T
  \label{eq.def_Z}
  \end{equation}

  2. Define the matrix $\widehat{\boldsymbol\Sigma}^{(1)}$ and $\widehat{\mathbf{S}}^\dagger$ as follows
  \begin{align*}
  \widehat{\boldsymbol\Sigma}^{(1)} = \frac{1}{T}\mathbf{W}^\top
  \left[
      \begin{matrix}
      \mathbf{z}^{(p+1)\top}\\
      \mathbf{z}^{(p + 2)\top}\\
      \vdots\\
      \mathbf{z}^{(T)\top}\\
\end{matrix}
\right],\ \widehat{\mathbf{S}}^\dagger = \left[
\begin{matrix}
\widehat{\mathbf{S}} & \mathbf{I}_d & 0             & 0             & \cdots & 0\\
                     & 0            & \mathbf{I}_d  & 0             & \cdots & 0\\
                     & 0            & 0             & \mathbf{I}_d  & \cdots & 0\\
                     & \vdots       & \vdots        & \vdots        & \cdots & \vdots\\
                     & 0            & 0             & 0             & \cdots & \mathbf{I}_d\\
                     & 0            & 0             & 0             & \cdots & 0\\
\end{matrix}
\right],
  \end{align*}
here $\mathbf{I}_d$ are $d\times d$ identity matrix. After that, calculate the lag 1 second-order residuals
\begin{align*}
  \widehat{\boldsymbol\Theta}^{(t)} = \mathbf{z}^{(t)}\mathbf{z}^{(t + 1)\top} - \mathbf{z}^{(t)}\mathbf{z}^{(t)\top}
  \widehat{\mathbf{S}}^\dagger\ \text{for } t = p,\cdots, T - 1.
\end{align*}
  3. Generate joint normal random variables $e^{(t)}, t = 1,\cdots, T$ such that $\mathbf{E}e^{(t)} = 0$ and $\mathbf{E}e^{(t_1)}e^{(t_2)} = K\left(\frac{t_1 - t_2}{k_T}\right)$. After that, calculate the resampled covariance matrix
  \begin{align*}
  \widehat{\boldsymbol \Sigma}^{(1)*} = \widehat{\boldsymbol\Sigma}^{(1)} + \frac{1}{T}\sum_{t = p}^{T - 1}\widehat{\boldsymbol\Theta}^{(t)}\times e^{(t)}.
  \end{align*}
  4. Derive the bootstrapped estimator $\widehat{\mathbf{S}}^{*}$ and the bootstrapped estimation root as follows:
  \begin{align*}
    \widehat{\mathbf{S}}^*_{\cdot l} = \mathcal{F}_{\widetilde{\mathcal{B}}_{l}}(\frac{1}{T}\mathbf{W}^\top\mathbf{W})\widehat{\boldsymbol \Sigma}^{(1)*}_{\cdot l}
    \ \text{and } \boldsymbol\Delta^*_{il} = \sqrt{T}(\widehat{\mathbf{S}}^*_{i l} - \widehat{\mathbf{S}}_{il}),
  \end{align*}
  here $i\in \widetilde{\mathcal{B}}_{l}$, $l = 1,\cdots, d$.
  Then Calculate
  $$
  \psi^*_b = \max_{l = 1,\cdots, d, i = 1,\cdots, pd}\vert\boldsymbol\Delta^*_{il}\vert.
  $$
  5. Repeat step 3 - 4 for $b = 1,\cdots, B$. After that, calculate the $1 - \alpha$ sample quantile, i.e., sort $\psi^*_b$ into $\psi^*_{(1)}\leq \psi^*_{(2)}\leq \cdots \leq \psi^*_{(B)}$, then
  $$
  C^*_{1 - \alpha} = \psi^*_{(k)}\ \text{such that } k = \min\{s = 1,\cdots, B: \frac{s}{B}\geq 1 - \alpha\}.
  $$
  6.a (Constructing simultaneous confidence interval) The $1 - \alpha$ simultaneous confidence interval for $\mathbf{A}^{(l)}, l = 1,\cdots, p$ is given by the inequality
  $$
  \vert\mathbf{A}^{(l)}_{ij} - \widehat{\mathbf{A}}_{ij}^{(l)}\vert\leq \frac{C^*_{1 - \alpha}}{\sqrt{T}}
  $$
  6.b (Performing exact test) Reject the null hypothesis if
  $$
  \max_{l = 1,\cdots, p, i,j = 1,\cdots, d}\sqrt{T}\vert\widehat{\mathbf{A}}_{ij}^{(l)} - \mathbf{A}_{ij}^{(l)\dagger}\vert > C^*_{1-\alpha}
  $$
\end{algorithm}

\begin{remark}[Why we do not use studentization]
A common way for constructing simultaneous confidence intervals, as demonstrated by \cite{MR3779709}, is through ``studentization'', which incorporates the
estimated standard deviations in the confidence interval construction. However, the studentization
is not well-suited for the post-selection estimator. To illustrate this, notice that the post-selection estimator only selects a few indices for least-square
estimation and estimates the other parameters to be $0$, which results in a 0 standard deviation and consequently a studentized confidence interval with width 0.
However, the model selector, which in our work is chosen to be the Lasso, cannot guarantee the identification of all non-zero parameters in the
underlying model for a given data set.  Furthermore, if the
model selector considers a parameter $\mathbf{A}_{ij}^{(l)}$ to be 0 and truncates it while it is, in fact, not $0$, then the real $\mathbf{A}_{ij}^{(l)}$
cannot be encompassed by the studentized confidence interval, which is unsatisfactory.

On the contrary, the confidence interval in algorithm \ref{algorithm.bootstrap} has positive widths for all parameters $\mathbf{A}_{ij}^{(l)}$.
Despite the fact that the model selector may not be able to identify all non-zero parameters,
the positive widths in the confidence interval still allows it to encompass the non-zero parameters
that were incorrectedly truncated by the model selector.
Therefore, the confidence interval in algorithm \ref{algorithm.bootstrap} exhibits greater robustness to model misspecification compared to the studentized confidence intervals.
\end{remark}

To illustrate how the second-order wild bootstrap algorithm generates correct simultaneous confidence intervals,
we define the conditional expectation $\mathbf{E}^*\cdot = \mathbf{E}\cdot|\mathbf{x}^{(t)}, t = 1,\cdots, T$,
often referred to as ``the expectation in the bootstrap world''. Suppose $\widetilde{\mathcal{B}}_l = \mathcal{B}_l$, then the conditional covariances of $\boldsymbol\Delta_{jl}^*$ is
\begin{equation}
\begin{aligned}
\mathbf{E}^*\boldsymbol\Delta_{j_1l_1}^*\boldsymbol\Delta_{j_2l_2}^* = \frac{1}{T}\sum_{q_1 = 1}^{pd}\sum_{q_2 =  1}^{pd}\sum_{t_1 = p}^{T - 1}\sum_{t_2 = p}^{T - 1}
\mathcal{F}_{\mathcal{B}_{l_1}}(\widehat{\boldsymbol\Sigma}^{(0)})_{j_1q_1}\mathcal{F}_{\mathcal{B}_{l_2}}(\widehat{\boldsymbol\Sigma}^{(0)})_{j_2q_2}\\
\times
\widehat{\boldsymbol\Theta}^{(t_1)}_{q_1l_1}\widehat{\boldsymbol\Theta}^{(t_2)}_{q_2l_2}K\left(\frac{t_1 - t_2}{k_T}\right)
\end{aligned}
\end{equation}
if $j_1\in\mathcal{B}_{l_1}$ and $j_2\in\mathcal{B}_{l_2}$ and $0$ otherwise.
As demonstrated in the online supplement,
the above formula converges in probability to the covariances of the estimation root
$
\sqrt{T}\vert\widehat{\mathbf{A}}_{ij}^{(l)} - \mathbf{A}_{ij}^{(l)}\vert
$. Therefore, the bootstrap algorithm \ref{algorithm.bootstrap} implicitly assists statisticians in estimating the covariance structure of the estimation root.
In addition, Theorem \ref{theorem.Gaussian_approx_} in Section \ref{section.Theoretical} shows that the estimation root converges in distribution to the maximum
of joint normal random variables with specific covariances. However,
unlike the classical T-test or F-test, statisticians do not have a quantile table
for the maximum of joint normal random variables.
Confronted with such circumstance, the Monte-Carlo methods, including bootstrap algorithms, which calculate the quantiles through simulations,
always provide viable solutions.

The kernel function $K(\cdot)$ in Algorithm \ref{algorithm.bootstrap} needs to satisfy Definition \ref{definition.kernel}. A simple example of a satisfactory
kernel function is $\sqrt{2\pi}$ times the standard normal density, i.e.,
$
K(x) = \exp(-\frac{x^2}{2}).
$

\section{$(m,\alpha)-$short range dependent random variables}
\label{section.m_alpha_dependent}
This section sets up constraints on the form of dependent and non-stationary time series for discussion in our work. Suppose $e_t, t\in \mathbf{Z}$ are independent (but not necessarily identically distributed) random variables, and random vectors
$\boldsymbol \epsilon^{(t)} = (\boldsymbol \epsilon^{(t)}_1, \cdots, \boldsymbol \epsilon^{(t)}_d)^\top\in\mathbf{R}^d$ satisfy
\begin{equation}
\boldsymbol \epsilon^{(t)}_i = g^{(t, T)}_i(\cdots, e_{t - 2}, e_{t - 1}, e_t),\ \text{here }\  i = 1,2,\cdots, d\ \text{and } t\in\mathbf{Z}.
\label{eq.epsilon_t}
\end{equation}
The function $g^{(t, T)}_i$ in eq.\eqref{eq.epsilon_t} is a measurable function for each $i$. The superscript $t$ and $T$ indicate that the function
$g^{(\cdot,\cdot)}$ can vary with respect to the time label $t$ and the number of observations $T$. Since the marginal and joint distribution
of $\boldsymbol \epsilon^{(t)}_i$ depend on the function $g^{(t, T)}_i$, the different choices of the function $g^{(t, T)}_i$ allows for the presence of
non-stationarity. \cite{zhangPolitis} introduced a form of non-stationary random variables named ``($m,\alpha$)-''short range dependent random variables.
Building on this concept, \cite{zhang2023simultaneous} further extended their results to perform statistical inference on second-order statistics of
scalar non-stationary time series. \cite{MR3779697}, \cite{MR2827528}, and \cite{MR4206676} utilized similar non-stationarity assumptions
in their work, but they did not focus on performing statistical inference on second-order statistics of non-stationary time series.
Our work also adopts the framework introduced by \cite{zhangPolitis} to model high-dimensional time series.

\begin{remark}
Discussion of  the form \eqref{eq.epsilon_t} dates back to Wiener's conjecture, which was later proven to be false by
\cite{MR2493017}. \cite{MR2172215} introduced the physical dependence measure for this form, which served as a useful
tool for analyzing nonlinear time series. Applications of this form include the works of \cite{MR2351105}, \cite{MR2915091}, and
\cite{MR3235390}. Recent research tends to extend the form \eqref{eq.epsilon_t} to more complex settings; readers can refer to the work of \cite{MR3310530}
and \cite{MR3779697} for examples.
\end{remark}

Define $e_t^\dagger$ as independent random variables such that $e_{t_1}^{\dagger}$ is independent of $e_{t_2}$ for any $t_1, t_2\in\mathbf{Z}$; and $e_{t}^\dagger$ has the same distribution as $e_t$. For any $j = 0,1,2,\cdots, $ define the random variables
\begin{equation}
\boldsymbol \epsilon^{(t, j)}_i =
\begin{cases}
g^{(t, T)}_i(\cdots,e_{t - j - 2},e_{t - j - 1},e_{t - j}^\dagger,e_{t - j + 1}\cdots, e_{t - 1}, e_t)\ \text{if } j\geq 0\\
\boldsymbol \epsilon^{(t)}_i\ \text{if } j < 0
\end{cases}
\label{eq.replace_epsilon}
\end{equation}
Then define the term
$$
\delta_i^{(t, j)} = \Vert \boldsymbol \epsilon^{(t)}_i - \boldsymbol \epsilon^{(t, j)}_i\Vert_m\ \text{and } \delta^{(j)} = \sup_{i = 1,\cdots, d, t\in\mathbf{Z}}\delta_i^{(t,j)}
$$
Clearly, $\delta_i^{(t, j)} = \delta^{(j)} = 0$ if $j<0$. With a slight abuse of notation, we omit $m$ here since it remains unchanged.
With these notations, we can define the $(m,\alpha)-$short range dependent random variables.
\begin{definition}[$(m,\alpha)-$short range dependent random variables]
Suppose $\boldsymbol \epsilon^{(t)}\in\mathbf{R}^d, t \in\mathbf{Z}$ are random vectors satisfying

1. $\mathbf{E}\boldsymbol\epsilon^{(t)} = 0$ for $t \in\mathbf{Z}$,

2. $\sup_{t\in\mathbf{Z}, i = 1,\cdots, d}\Vert \boldsymbol \epsilon^{(t)}_i\Vert_m = O(1)$,

3. $\sup_{k = 0,1,\cdots,}(1 + k)^\alpha\sum_{j = k}^\infty \delta^{(j)} = O(1)$.

\noindent Then we call  $\boldsymbol\epsilon^{(t)}$ the $(m,\alpha)-$short range dependent random variables. Here the constants $m,\alpha > 1$.
\label{definition.short_range}
\end{definition}

\begin{remark}
  In their work, \cite{MR3779697} considered similar settings, where $\delta^{(j)}$ was assumed to have exponential decay with respect to $j$. However,
  they only considered the sample mean of vector time series and did not estimate the covariance matrices of the estimators. Our work,
  on the other hand, focuses on
  estimating the autoregressive coefficient matrix, which is a second-order statistics with variances dependent on the fourth-order cumulants of the time series.
  Moreover, we provides statisticians with a consistent estimation procedure for the covariance matrices of the autoregressive coefficient matrix, which helps
  them establish simultaneous confidence intervals.
\end{remark}

Definition \ref{definition.short_range} is quite inclusive and covers various frequently-used vector time series models,
including non-stationary linear processes, random coefficient autoregressive process, among others.
Readers can access \cite{MR3779697} for more examples.

The first lemma involves arithmetic operations on $(m,\alpha)-$short range dependent random variables.
\begin{lemma}
Suppose $\boldsymbol\epsilon^{(t)}\in\mathbf{R}^d, t \in\mathbf{Z}$ are $(m,\alpha)-$short range dependent random variables with $m\geq 4$.

1. Suppose $p = O(1)$ is an integer. Define
$\boldsymbol \eta^{(t)} = (\boldsymbol\epsilon^{(t)\top}, \boldsymbol\epsilon^{(t - 1)\top}, \boldsymbol\epsilon^{(t - p + 1)\top})^\top\in\mathbf{R}^{pd}$, here $t\in\mathbf{Z}$. Then $\boldsymbol \eta^{(t)}$ are $(m,\alpha)-$short range dependent random variables.

2. Define $\boldsymbol \gamma^{(t)}, \boldsymbol \zeta^{(t)}\in\mathbf{R}^{d^2}$ such that
$$
\boldsymbol \gamma^{(t)}_{(i - 1)\times d + j} = \boldsymbol\epsilon^{(t)}_i\boldsymbol\epsilon^{(t)}_j - \mathbf{E}\boldsymbol\epsilon^{(t)}_i\boldsymbol\epsilon^{(t)}_j
\ \text{and } \boldsymbol \zeta^{(t)}_{(i - 1)\times d + j} = \boldsymbol\epsilon^{(t - 1)}_i \boldsymbol\epsilon^{(t)}_j - \mathbf{E}\boldsymbol\epsilon^{(t - 1)}_i \boldsymbol\epsilon^{(t)}_j
$$
then $\boldsymbol \gamma^{(t)}$ and $\boldsymbol \zeta^{(t)}$ are $(m/2,\alpha)-$short range dependent random variables.

3. Suppose the matrix $\mathbf{L}\in\mathbf{R}^{c\times d}$ satisfies $\vert\mathbf{L}\vert_{L_1} = O(1)$, then the vector $\mathbf{L}\boldsymbol\epsilon^{(t)}$ are $(m,\alpha)-$short range dependent random variables.

4. Suppose the matrix $\mathbf{L}\in\mathbf{R}^{d\times d}$ satisfies $\vert\mathbf{L}\vert_2\leq \rho < 1$ (which is the stability condition in \cite{MR1238940}) and $\rho^2\max_{i = 1,\cdots, d}\vert\mathbf{L}_{i\cdot}\vert_0 < 1$. In addition, suppose $\boldsymbol\epsilon^{(t)}$ are white noise, i.e.,
$$
\mathbf{E}\boldsymbol\epsilon^{(t)} = 0\ \text{and }  \mathbf{E}\boldsymbol\epsilon^{(t_1)}\boldsymbol\epsilon^{(t_2)T} =
\begin{cases}
0\ \text{if } t_1\neq t_2\\
\boldsymbol \Sigma_{\boldsymbol\epsilon}\ \text{if } t_1 = t_2
\end{cases}
$$
Define the weak stationary VAR(1) process as follows:
$$
\boldsymbol\beta^{(t)} = \mathbf{L}\boldsymbol\beta^{(t - 1)} + \boldsymbol\epsilon^{(t)}
$$
then $\boldsymbol\beta^{(t)}$ are $(m,\alpha)-$short range dependent random variables.
\label{lemma.operation}
\end{lemma}
\begin{remark}
The condition that $\vert\mathbf{L}\vert_2 \leq \rho < 1$ in lemma \ref{lemma.operation} is common in the multivariate time series literature , such as \cite{MR4206676}.
Especially, according to section 2.2 in \cite{MR1238940}, $\boldsymbol\beta^{(t)}$ has a moving average representation
$$
\boldsymbol\beta^{(t)}  = \boldsymbol\epsilon^{(t)} + \sum_{k = 1}^\infty\mathbf{L}^k\boldsymbol\epsilon^{(t - k)},
$$
so this condition is a must for $\boldsymbol{\beta}^{(t)}$ to have a stationary covariance structure.

\end{remark}
The remaining parts of this section involve deriving the consistency result,
the Gaussian approximation theorem, and the covariances estimation for the $(m,\alpha)-$short range dependent random
variables. Gaussian approximation, as presented in \cite{MR3161448}, \cite{MR3718156}, \cite{MR4500619}, \cite{chang2023central}, among others,
is a useful technique for analyzing high dimensional data.  Before presenting our work, we introduce the definition of a kernel function,
which will be used to construct spatial and temporal covariances estimators in Theorem \ref{theorem.MA}.
\begin{definition}[Kernel function]
Suppose function $K(\cdot): \mathbf{R}\to [0,\infty)$ is symmetric, continuously differentiable, $K(0) = 1$, $\int_{\mathbf{R}} K(x)dx < \infty$, and $K(x)$ is decreasing on $[0,\infty)$. Besides, define the Fourier transform $\mathcal{F}K(x) = \int_{\mathbf{R}} K(t)\exp(-2\pi\mathrm{i}\times tx)dt$, here $\mathrm{i} = \sqrt{-1}$. Assume $\mathcal{F}K(x)\geq 0$ for all $x\in\mathbf{R}$ and $\int_{\mathbf{R}} \mathcal{F}K(x)dx < \infty$. We call $K$ the kernel function.
\label{definition.kernel}
\end{definition}

\begin{theorem}
Suppose $\boldsymbol \epsilon^{(t)}\in\mathbf{R}^d, t\in\mathbf{Z}$ are $(m,\alpha)-$short range dependent random variables with $m>4,\alpha > 1$, and $d = O(T^{\alpha_d})$, here  $0\leq \alpha_d < \frac{m\alpha}{1 + 12\alpha}$.

1. Define the matrix
$$
\boldsymbol \Pi = \{\boldsymbol \epsilon^{(j)}_i\}_{i = 1,\cdots, d, j = 1,\cdots, T}\in\mathbf{R}^{d\times T}\ \text{and }\ \boldsymbol \Pi^{(s)} = \{\boldsymbol \epsilon^{(j)}_i - \mathbf{E}\boldsymbol \epsilon^{(j)}_i|\mathcal{F}^{(t, s)}\}_{i = 1,\cdots, d, j = 1,\cdots, T},
$$
then here exists a constant $C > 0$ such that for any vector $\mathbf{b}\in\mathbf{R}^T$,
\begin{equation}
\Vert\ \vert\boldsymbol \Pi\mathbf{b}\vert_\infty\ \Vert_m\leq Cd^{1/m}\times\vert\mathbf{b}\vert_2\ \text{and } \Vert\ \vert\boldsymbol \Pi^{(s)}\mathbf{b}\vert_\infty\ \Vert_m
\leq Cd^{1/m}\times \vert\mathbf{b}\vert_2\times (1 + s)^{-\alpha}.
\end{equation}
Here $\mathcal{F}^{(t, s)}$ is the $\sigma-$field generated by $e_t, e_{t - 1},\cdots, e_{t - s}$.

2. Define the sample mean $\overline{\boldsymbol\epsilon} = \frac{1}{T}\sum_{t = 1}^T\boldsymbol \epsilon^{(t)}$. Suppose $\exists $
a constant $c>0$ such that
$$
Var(\sqrt{T}\overline{\boldsymbol\epsilon}_i)\geq c\ \text{for } i = 1,\cdots, d,
$$
then we have
\begin{equation}
\sup_{x\in\mathbf{R}}\vert
Prob\left(\vert\sqrt{T}\overline{\boldsymbol\epsilon}\vert_\infty\leq x\right) - Prob\left(\vert\boldsymbol \xi\vert_\infty \leq x\right)
\vert = o(1),
\end{equation}
here $\boldsymbol\xi\in\mathbf{R}^d$ have joint normal distribution such that $\mathbf{E}\boldsymbol\xi = 0$ and $Cov(\boldsymbol \xi_i, \boldsymbol \xi_j) = T\times Cov(\overline{\boldsymbol\epsilon}_i, \overline{\boldsymbol\epsilon}_j)$.

3. Suppose $K(\cdot)$ is a kernel function as in definition \ref{definition.kernel}, and the bandwidth $k_T$ satisfies $k_T\to\infty$ and $T^{\frac{4\alpha_d}{m} - \frac{1}{2}}\times k_T\to 0$ as $T\to\infty$. Then
\begin{equation}
\begin{aligned}
\max_{i, j = 1,\cdots, d}\vert
\frac{1}{T}\sum_{t_1 = 1}^T\sum_{t_2 = 1}^T K\left(\frac{t_1 - t_2}{k_T}\right)\boldsymbol\epsilon^{(t_1)}_i\boldsymbol\epsilon^{(t_2)}_j - T\times Cov(\overline{\boldsymbol\epsilon}_i, \overline{\boldsymbol\epsilon}_j)
\vert\\
 = O_p\left(v_T + T^{\frac{4\alpha_d}{m} - \frac{1}{2}}\times k_T\right),
\end{aligned}
\end{equation}
here
\begin{align*}
v_T =
\begin{cases}
k_T^{1 - \alpha},\ \text{if } 1 < \alpha < 2\\
\log(k_T) / k_T,\ \text{if } \alpha = 2\\
1 / k_T,\ \text{if } \alpha > 2
\end{cases}.
\end{align*}
\label{theorem.MA}
\end{theorem}
\begin{remark}
In particular, $\sqrt{Var(\overline{\boldsymbol\epsilon}_i)} = \Vert\overline{\boldsymbol\epsilon}_i\Vert_2\leq \Vert\overline{\boldsymbol\epsilon}_i\Vert_m\leq \frac{C}{\sqrt{T}}$. From theorem \ref{theorem.MA}, suppose $t_1 > t_2$, then
\begin{equation}
\begin{aligned}
\vert Cov(\boldsymbol\epsilon^{(t_1)}_i, \boldsymbol\epsilon^{(t_2)}_j)\vert = \vert\mathbf{E} \boldsymbol\epsilon^{(t_2)}_j\times(\boldsymbol\epsilon^{(t_1)}_i - \mathbf{E}\boldsymbol\epsilon^{(t_1)}_i|\mathcal{F}^{(t_1, t_1 - t_2 - 1)})\vert\\
\leq C\Vert\boldsymbol\epsilon^{(t_1)}_i - \mathbf{E}\boldsymbol\epsilon^{(t_1)}_i|\mathcal{F}^{(t_1, t_1 - t_2 - 1)}\Vert_m\\
\leq C_1(1 + \vert t_1 - t_2\vert)^{-\alpha},
\end{aligned}
\label{eq.covS}
\end{equation}
so the covariances of random variables in Definition \ref{definition.short_range} should exhibit a polynomial decay with
respect to the time lag between two random variables. This observation partially explains why we call definition \ref{definition.short_range} as
``short range dependence,'' meaning that covariances between $\boldsymbol\epsilon^{(t)}$ and $\boldsymbol\epsilon^{(t - k)}$ cannot be too large for large $k$.
\end{remark}

\section{Asymptotic results}
\label{section.Theoretical}
This section adopts the notion of $(m,\alpha)-$short range dependence to analyze the VAR model with non-i.i.d. innovations, which allows for the absence of
strict stationarity in the time series.
In the literature, i.i.d. innovation is always a necessary condition for the validity of the estimation and statistical inference of VAR time series
due to the fact that there are few satisfactory tools to describe the non-stationarity in time series. With the help of the $(m,\alpha)-$short range dependence notion,
statisticians are now able to derive the asymptotic distribution of the estimator $\widehat{\mathbf{A}}^{(l)}$ and prove the validity of the bootstrap
algorithm \ref{algorithm.bootstrap} even when the underlying process $\mathbf{x}^{(t)}$ is only weakly stationary but not strictly stationary.
Readers can refer to \cite{MR1093459} for the difference between weak and strict stationarity. We first present some technical assumptions:

\textbf{Assumptions:}

1. Observations $\mathbf{x}^{(t)}\in\mathbf{R}^d, t = 1,\cdots, T$ are stemmed from a vector autoregressive process, i.e., eq.\eqref{eq.AR_model}. Besides, $\mathbf{x}^{(t)}, t\in\mathbf{Z}$ are $(m,\alpha)-$short range dependent random variables with $m > 8$, $\alpha > 1$ are two constants. Suppose $p = O(1)$, $d = O(T^{\alpha_d})$, and
$
0\leq \alpha_d < \frac{m\alpha}{2 + 24\alpha}
$
is a constant.

2. Suppose the Lasso parameter $\lambda = c\times T^{-\alpha_\lambda}$, here $c, \alpha_\lambda$ are two constants with $0 < \alpha_\lambda < \frac{1}{2} - \frac{4\alpha_d}{m}$ and $c > 0$. The threshold $b_T$ satisfies $b_T\to 0$ and $\frac{b_T}{\lambda}\to\infty$ as $T\to\infty$. Define the set
$$
\mathcal{B}_l = \left\{j = 1,\cdots, pd: \mathcal{S}_{jl} \neq 0\right\}.
$$
Assume that $\max_{l = 1,\cdots, d}\vert\mathcal{B}_l\vert = O(1)$ and $\max_{l = 1,\cdots, d}\vert\mathcal{B}_l\vert > 0$. In addition assume
$\min_{l = 1,\cdots, d, j\in\mathcal{B}_l}\vert\mathbf{S}_{jl}\vert > 2b_T$. Define the $pd\times pd$ matrix
$$
\mathbf{S}^{\dagger} = \left[
\begin{matrix}
  \mathbf{S}^\top \\
  \mathbf{I}_d & 0 & \cdots & 0 & 0 \\
  0            & \mathbf{I}_d & \cdots & 0 & 0\\
  \vdots       & \vdots       & \cdots & \vdots & \vdots \\
  0            & 0            & \cdots &\mathbf{I}_d & 0\\
\end{matrix}
\right]^\top
$$
and assume all eigenvalues of $\mathbf{S}^{\dagger}$ is less than $1$ in absolute value. Define the matrix
$\boldsymbol\Sigma^{(0)} = \mathbf{E}\mathbf{z}^{(p)}\mathbf{z}^{(p)\top}$, where $\mathbf{z}^{(t)}$ is defined in eq.\eqref{eq.def_Z}. Assume that the minimum eigenvalue of $\boldsymbol\Sigma^{(0)}$ is greater than a constant $c > 0$.

3. Define the matrix
\begin{equation}
\boldsymbol\Lambda^{(l_1, l_2)} = \mathcal{F}_{\mathcal{B}_{l_1}}(\boldsymbol\Sigma^{(0)})\left(\frac{1}{T}\mathbf{E}\mathbf{W}\boldsymbol{\eta}^{(l_1)}\boldsymbol{\eta}^{(l_2)\top}\mathbf{W}^\top\right)\mathcal{F}_{\mathcal{B}_{l_2}}(\boldsymbol\Sigma^{(0)})
\text{, here  } l_1,l_2 = 1,\cdots, d.
\label{eq.cov_matrixX}
\end{equation}
Assume that there exists a constant $c>0$ such that for sufficiently large $T$, $\boldsymbol\Lambda^{(l,l)}_{jj} > c$ for all $l = 1,\cdots, d$
and $j\in\mathcal{B}_l$.

\begin{remark}
  The constraint on the eigenvalues of $\mathbf{S}^{\dagger}$ is essential to maintain the weak stationarity of the autoregressive
  process eq.\eqref{eq.AR_model}, as demonstrated in \cite{MR1238940}. Assumption 3 actually requires that the variances of
  the post-selection estimators do not degenerate to $0$; otherwise, the estimator will shrink to a constant.
  In such cases, statisticians may need to model the decay form of the variances before performing statistical inference.
  The work of \cite{MR4102694} provides readers with a good example of dealing with decaying variances.
  Finally, $(m,\alpha)-$short range dependence and the autoregressive form are not mutually exclusive.
  If there are only a few non-zero elements in $\mathbf{S}$,
   and the innovations are $(m,\alpha)$-short range dependent, then the observations' $(m,\alpha)-$short range dependency can be
   assured, as demonstrated in lemma \ref{lemma.operation}. Even when $\mathbf{S}$ contains moderate non-zero elements, numerical simulations
   show that the estimator and the proposed bootstrap algorithm still maintain
   the desired coverage probability, suggesting that the underlying assumption is not violated.
\end{remark}

With these technical assumptions, we can derive the desired theoretical guarantees of the proposed estimators and algorithms.
The first result involves the (model--selection) consistency of the Lasso estimator, which is presented in theorem \ref{theorem.Lasso_consistency}.

\begin{theorem}
  Suppose assumption 1 to 3. Then
  \begin{equation}
  \label{eq.Lasso_consistency}
    \max_{l = 1,\cdots, d}\vert\widetilde{\mathbf{S}}_{\cdot l} - \mathbf{S}_{\cdot l}\vert_1 = O_p\left(\lambda\right)\ \text{and  }
    \max_{l = 1,\cdots, d}\vert\widetilde{\mathbf{S}}_{\cdot l} - \mathbf{S}_{\cdot l}\vert_2 = O_p\left(\lambda\right),
  \end{equation}
  here $\lambda$ is the Lasso parameter defined in eq.\eqref{eq.use_lasso}.
  In particular, we have
  \begin{equation}\label{eq.model_selection}
    Prob\left(\cup_{l = 1,\cdots, d}(\widetilde{\mathcal{B}}_l \neq \mathcal{B}_l)\right) = o(1).
  \end{equation}
  \label{theorem.Lasso_consistency}
\end{theorem}
The second result presents the convergence rate and the asymptotic distribution of the post-selection estimator,
which are summarized in the following theorem.
\begin{theorem}
Suppose assumption 1 to 3. Then the post-selection estimator satisfies
\begin{equation}
\begin{aligned}
\sup_{x\in\mathbf{R}}
\vert
Prob\left(
\max_{j = 1,\cdots, pd, l = 1,\cdots, d}\sqrt{T}\vert\widehat{\mathbf{S}}_{jl} - \mathbf{S}_{jl}\vert\leq x
\right)\\
 - Prob\left(\max_{l = 1,\cdots, d, j\in\mathcal{B}_l}\vert\boldsymbol\Gamma_{jl}\vert\leq x\right)
\vert = o(1).
\end{aligned}
\label{eq.prob_result}
\end{equation}
Here $\boldsymbol\Gamma_{jl}$ are joint normal random variables with $\mathbf{E}\boldsymbol\Gamma_{jl} = 0$ and
\begin{align*}
Cov\left(
\boldsymbol\Gamma_{j_1l_1}, \boldsymbol\Gamma_{j_2l_2}
\right) = \boldsymbol\Lambda^{(l_1, l_2)}_{j_1 j_2}.
\end{align*}
\label{theorem.Gaussian_approx_}
\end{theorem}
\begin{remark}
Theorem \ref{theorem.Gaussian_approx_} explains the necessity of the condition $\max_{l = 1,\cdots, d}\vert\mathcal{B}_l\vert > 0$ in Assumption 2.
Without this constraint,
$\mathcal{B}_l$ can be empty sets for all $l$, and $\boldsymbol\Gamma_{jl}$ may degenerate to the constant $0$. If this happens, then the discussion regarding
 the asymptotic distribution of the post--selection estimator $\widehat{\mathbf{S}}_{jl}$ becomes meaningless.
It is well-known that the joint normal random variables $\mathbf{\Gamma}_{jl}$ satisfy
$\max_{l = 1,\cdots, d, j\in\mathcal{B}_l}\vert\boldsymbol\Gamma_{jl}\vert = O_p\left(\sqrt{\log(T)}\right)$. Moreover, from
Theorem \ref{theorem.MA} in Section \ref{section.m_alpha_dependent}, for any $j = 1,\cdots, pd$,
\begin{align*}
\Vert\mathcal{F}_{\mathcal{B}_{l_1}}(\boldsymbol\Sigma^{(0)})_{j\cdot}\frac{1}{\sqrt{T}}\mathbf{W}\boldsymbol\eta^{(l)}\Vert_{m/2}\leq C
\end{align*}
for a fixed constant $C$. Therefore, theorem \ref{theorem.Gaussian_approx_} implies that the post-selection estimator has convergence rate
\begin{equation}
\max_{j = 1,\cdots, pd, l = 1,\cdots, d}\vert\widehat{\mathbf{S}}_{jl} - \mathbf{S}_{jl}\vert = O_p\left(\frac{\sqrt{\log(T)}}{\sqrt{T}}\right).
\label{eq.S_max_bound}
\end{equation}
\end{remark}

The last result involves the consistency of the (implicit) covariances estimator, as well as the validity of the bootstrap
algorithm \ref{algorithm.bootstrap}.  According to \cite{MR1707286}, the consistency of the bootstrap algorithm is
guaranteed as long as the distribution of the bootstrapped statistics in the bootstrap world
(i.e., conditional on the observed data) converges to the asymptotic distribution of the original estimator, that is,
\begin{equation}
\sup_{x\in\mathbf{R}}\vert
Prob^*\left(\max_{j = 1,\cdots, pd, l = 1,\cdots, d}\vert\boldsymbol\Delta^*_{jl}\vert\leq x\right) - Prob\left(\max_{l = 1,\cdots, d, j\in\mathcal{B}_l}\vert\boldsymbol\Gamma_{jl}\vert\leq x\right)
\vert = o_p(1),
\label{eq.boot_consistency}
\end{equation}
where $\boldsymbol\Delta^*_{jl}$ is introduced in algorithm \ref{algorithm.bootstrap}, $\boldsymbol\Gamma_{jl}$ is defined in theorem \ref{theorem.Gaussian_approx_}, and
$Prob^*(\cdot)$ is ``the probability in the bootstrap world, '' i.e., the conditional probability conditioning on the observations $\mathbf{x}^{(t)}, t = 1,\cdots, T$.
We summarize these results in Theorem \ref{theorem.bootstrap}.
\begin{theorem}
Suppose assumption 1 to 3. In addition, suppose the function $K(\cdot)$ in algorithm \ref{algorithm.bootstrap} is a kernel function (satisfying definition \ref{definition.kernel}) and $k_T$ satisfies $k_T\to\infty$, $v_T\times\log^9(T) = o(1)$, and
\begin{equation}
k_T\times T^{\frac{8\alpha_d}{m} - \frac{1}{2}}\times\log(T)^{\frac{19}{2}}\to 0 \text{ as } T\to\infty.
\end{equation}
 Then
\begin{equation}
\begin{aligned}
\max_{j = 1,\cdots, pd, l = 1,\cdots, d}\vert\mathbf{E}^*\boldsymbol\Delta_{j_1l_1}^*\boldsymbol\Delta_{j_2l_2}^* -\boldsymbol\Lambda^{(l_1, l_2)}_{j_1 j_2}\vert =
O_p\left(v_T + k_T\times T^{\frac{8\alpha_d}{m} - \frac{1}{2}}\times\sqrt{\log(T)}\right),
\end{aligned}
\label{eq.convergence_covariance}
\end{equation}
where $\boldsymbol\Lambda^{(l_1, l_2)}$ is defined in eq.\eqref{eq.cov_matrixX} and $k_T$ coincides with theorem \ref{theorem.MA}.
Moreover, eq.\eqref{eq.boot_consistency} holds true asymptotically. $\boldsymbol\Delta^*_{jl}$ in eq.\eqref{eq.convergence_covariance}
is defined in Algorithm \ref{algorithm.bootstrap}.
\label{theorem.bootstrap}
\end{theorem}

\section{Numerical experiments}
\label{section.numerical}
This section provides readers with a series of numerical experiments to verify the soundness of our work under different circumstances.
We undertake a performance comparison among our approach, the method of fitting column-wise Lasso, as introduced by
\cite{MR3357870}, \cite{MESSNER20191485} and \cite{MR4209452}; the method of
fitting column-wise Dantzig selector introduced by \cite{MR3450535}; and the method of fitting column-wise thresholded Lasso presented
in \cite{MR2847974}. The autoregressive sieve bootstrap, as demonstrated by \cite{MR3343007}, is recognized as a
common practice for analyzing a VAR time series. Nevertheless, the consistency of the autoregressive sieve bootstrap hinges
on the assumption that  the innovations $\boldsymbol\epsilon^{(t)}$ in eq.\eqref{eq.AR_model} are indeed independent and identically distributed;
while the validity of algorithm \ref{algorithm.bootstrap} does not rely on this assumption.
This section compares the bootstrap algorithm \ref{algorithm.bootstrap} and the autoregressive sieve bootstrap to elucidate their different domain of validity.

\textbf{Generating data: } Suppose $\mathbf{e}^{(t)}_j$ are independent normal random variables with mean $0$ and variance $1$,
here $t\in\mathbf{Z}$ and $j = 1,\cdots, d$. After that, consider three types of $\boldsymbol\eta^{(t)}\in\mathbf{R}^d$ as follows:

\textbf{Independent: } $\boldsymbol\eta^{(t)}_j = \mathbf{e}^{(t)}_j$;

\textbf{Product normal: } $\boldsymbol\eta^{(t)}_j = \mathbf{e}^{(t)}_j\mathbf{e}^{(t - 1)}_j$;

\textbf{Non-stationary: } $\boldsymbol\eta^{(t)}_j = \mathbf{e}^{(t)}_j$ for $t \leq \lfloor\frac{T}{2}\rfloor$ and $\boldsymbol\eta^{(t)}_j = \mathbf{e}^{(t)}_j\mathbf{e}^{(t - 1)}_j$ for $t > \lfloor\frac{T}{2}\rfloor$, here $T$ is the sample size.

We define the innovations $\boldsymbol\epsilon^{(t)}$ to be
\begin{align*}
\boldsymbol\epsilon^{(t)} =
\left[
\begin{matrix}
1.0 & 0.5 & 0 & 0 & \cdots & 0\\
-0.5 & 1.0 & 0.5 & 0& \cdots & 0\\
\vdots & \vdots & \vdots & \vdots & \cdots & \vdots\\
0      &  0     & 0      & 0      & \cdots & 1.0\\
\end{matrix}
\right]\boldsymbol\eta^{(t)}.
\end{align*}
With these constructions, all 3 types of $\boldsymbol\epsilon^{(t)}$ are white noises with the same covariance matrices. However, the ``product normal'' innovations are not independent, and the ``non-stationary'' innovations are not stationary.  After that, define the matrices
\begin{align*}
	\mathbf{A}^{(1)} = \{\mathbf{A}_{ij}^{(1)}\}_{i,j = 1,\cdots, d}\ \text{with }\mathbf{A}_{ii+1}^{(1)} =  \mathbf{A}_{i+1i}^{(1)} = 0.3\ \text{and } 0\ \text{otherwise},\\
	\mathbf{A}^{(2)} = \{\mathbf{A}_{ij}^{(2)}\}_{i,j = 1,\cdots, d}\ \text{with }\ \mathbf{A}_{ii+1}^{(2)} = -0.3\ \text{and } 0\ \text{otherwise},\\
	\mathbf{A}^{(3)} = \{\mathbf{A}_{ij}^{(3)}\}_{i,j = 1,\cdots, d}\ \text{with }\ \mathbf{A}_{i+1i}^{(3)} = -0.4,\ \text{and } 0\ \text{otherwise }.
\end{align*}
Here $i = 1,\cdots, d$. We generate the observations $\mathbf{x}^{(t)}$ based on three time series models in this section, that is the VAR(1) model $\mathbf{x}^{(t)} = \boldsymbol\epsilon^{(t)} + \mathbf{A}^{(1)}\mathbf{x}^{(t - 1)} $, the VAR(2) model
$\mathbf{x}^{(t)} = \boldsymbol\epsilon^{(t)} + \mathbf{A}^{(1)}\mathbf{x}^{(t - 1)} + \mathbf{A}^{(2)}\mathbf{x}^{(t - 2)}$, and the VAR(3) model
$
\mathbf{x}^{(t)} = \boldsymbol\epsilon^{(t)} + \mathbf{A}^{(1)}\mathbf{x}^{(t - 1)} + \mathbf{A}^{(2)}\mathbf{x}^{(t - 2)} +
\mathbf{A}^{(3)}\mathbf{x}^{(t - 3)}
$.

\textbf{Hyper-parameter selection: } The proposed methods require the selection of several hyper--parameters: the order of the VAR model $p$,
the Lasso parameter $\lambda$, the threshold $b_T$, and the kernel bandwidth
$k_T$. The order $p$ can be selected using the Akaike information criterion (AIC),
as introduced in the manuscript by \cite{MR2172368}. Furthermore, numerous packages have implemented AIC for the order selection
of vector autoregressive time series data. Our work utilizes the Python package ``statsmodels'' implemented by
\cite{seabold2010statsmodels} to facilitate the order selection.

When data exhibit dependent structures, \cite{MR1294896} illustrate that commonly used cross -- validation techniques may not
work well due to the dependency between training and test sets. To avoid violating the dependent structure of the vector time series,
we prefer using the train-test split rather than cross-validation to fine-tune $\lambda$ and $b_T$.
Specifically, we fix an integer $T_1 < T$, the sample size, and use $\mathbf{x}_1,\cdots, \mathbf{x}_{T_1}$ to
fit the parameters $\widehat{\mathbf{A}}^{(k)},\ k = 1,\cdots, p$. Afterward, we use the data
$\mathbf{x}_{T_1 + 1},\cdots, \mathbf{x}_T$ to calculate the test set mean-squared error
$$
\tau = \frac{1}{T - T_1}\sum_{t = T_1 + p}^T \vert\mathbf{x}^{(t)} - \sum_{k = 1}^p \widehat{\mathbf{A}}^{(k)}\mathbf{x}^{(t - k)}\vert^2_2,
$$
here the order $p$ is selected through AIC. We choose $\lambda$ and $b_T$ that minimize $\tau$. In this work, we choose $T_1 = \lfloor\frac{3}{4}T\rfloor$, and $\lfloor x\rfloor$ represents the largest integer that is smaller than or equal to $x$.

Selecting the kernel bandwidth $k_T$ based on data can be quite involved, but the literature has already
provided some ways to address this issue. \cite{MR2041534} derived an estimator of the optimal block size for
block bootstrap methods based on the notion of spectral estimation via the flat-top lag-windows, which was
implemented in the R package ``np'' . This method was later utilized by \cite{MR2656050} for
bandwidth selection in the dependent wild bootstrap. Other notable
contributions include \cite{MR2380557}, \cite{MR2374985},  and references therein. Our work prefers
the method introduced by \cite{MR2041534} for selecting $k_T$. However, this method was originally designed
for scalar time series, so we make certain modifications, as demonstrated in Algorithm \ref{algorithm.select}, to adapt the method to vector time series.
\begin{algorithm}
\textbf{Input: } Observations $\mathbf{x}^{(t)}\in\mathbf{R}^d, t = 1,\cdots, T$ stemmed from eq.\eqref{eq.AR_model}; the Lasso parameter $\lambda$; and the threshold $b_T$.

1. Follow step 1 and 2 in algorithm \ref{algorithm.bootstrap} to derive $\widehat{\boldsymbol\Theta}^{(t)}$ for $t = p,\cdots, T - 1$.

2. For $i = 1,\cdots, pd, j = 1,\cdots, d$, feed the scalar time series $\widehat{\boldsymbol\Theta}^{(t)}_{ij}, t= p,\cdots, T - 1$ to the algorithm derived by \cite{MR2041534}, and calculate the corresponding bandwidths $\widetilde{k}_{T, ij}$.

3. Finally, calculate the median $\widehat{k}_T$ of the bandwidths $\widetilde{k}_{T, ij}$. We consider the median $\widehat{k}_T$ to be a suitable bandwidth for the bootstrap algorithm \ref{algorithm.bootstrap}.
\label{algorithm.select}
\end{algorithm}
To illustrate Algorithm \ref{algorithm.select}, notice that
\begin{align*}
  \{\widehat{\boldsymbol\Theta}^{(t)}_{ij}\}_{i = 1,\cdots, pd, j = 1,\cdots, d} = \mathbf{z}^{(t)}\mathbf{x}^{(t+1)\mathcal{T}} - \mathbf{z}^{(t)}\mathbf{z}^{(t)\mathcal{T}}\widehat{\mathbf{S}}\approx \mathbf{z}^{(t)}\boldsymbol\epsilon^{(t + 1)\mathcal{T}}
\end{align*}
according to eq.\eqref{eq.S_max_bound}. Moreover, Algorithm \ref{algorithm.bootstrap} eventually constructs confidence
intervals for
the linear combinations of the sample mean $\frac{1}{T}\sum_{t = p}^{T - 1}\mathbf{z}^{(t)}\boldsymbol\epsilon^{(t + 1)\mathcal{T}}$.
This aligns with the setting described in \cite{MR2041534}. While we opt for the medium with the expectation that the bandwidth should
be appropriate for the majority of
$\widehat{\boldsymbol\Theta}^{(t)}_{ij}$, other alternatives, such as the sample mean of $\widetilde{k}_{T, ij}$, should also be reasonable choices.

\textbf{Results: } The numerical results are presented in Table \ref{table.experiment_parameter}, Table \ref{table.numerical_res},
Figure \ref{figure.heat}, and Figure \ref{figure.performance_sec}.  Additional simulation results, which are similar to what
we present here, will be postponed to the online
supplement.
We evaluate the performance of the estimator using the column-wise one (C1) norm $\vert\cdot\vert_{C1}$, the column-wise
two (C2) norm $\vert\cdot\vert_{C2}$, and the model misspecification $\kappa$, defined as follows:
\begin{align*}
\vert \widehat{\mathbf{S}}^\dagger - \mathbf{S}\vert_{C1} = \max_{l = 1,\cdots, d}\sum_{i = 1,\cdots, pd}\vert\widehat{\mathbf{S}}_{il}^\dagger - \mathbf{S}_{il}\vert,\
\vert\widehat{\mathbf{S}}^\dagger - \mathbf{S}\vert_{C2} = \max_{l = 1,\cdots, d}\sqrt{\sum_{i = 1}^{pd}\vert\widehat{\mathbf{S}}_{il}^\dagger - \mathbf{S}_{il}\vert^2},\\
\text{and } \kappa = \text{the number of elements in the set } \mathcal{K},\\
\text{where } \mathcal{K} = \{(i,j): \widehat{\mathbf{S}}_{ij}^\dagger = 0\text{ and } \mathbf{S}_{ij}\neq 0\}\cup \{(i,j): \widehat{\mathbf{S}}_{ij}^\dagger \neq 0\text{ and } \mathbf{S}_{ij} =  0\},
\end{align*}
here $\widehat{\mathbf{S}}^\dagger$ is an estimator for the matrix $\mathbf{S} = [\mathbf{A}^{(1)},\cdots, \mathbf{A}^{(p)}]^\top$.
The reason for using the column-wise norm is that many estimation procedures for the vector autoregressive model,
including those introduced in this section, are established through column-wise linear regression.
The column-wise norms not only reflect the general efficiency of these methods but also provide insights
into the overall performance of different linear regression methods when fitting a sparse linear model.
The model misspecification evaluates the difference between the estimated and actual positions
of non-zero parameters in matrix $\mathbf{S}$. Clearly, if the difference is small, then the proposed algorithm
is capable of recovering the underlying sparse pattern of  $\mathbf{S}$.

Figure \ref{figure.heat} plots the actual positions of non-zero parameters in $\mathbf{S}$ along with
those estimated by the proposed post-selection algorithm and the Lasso algorithm. For the Lasso algorithm, we consider the estimated parameter to be $0$ if its absolute value is smaller than 0.01. According to Figure \ref{figure.heat}, the post-selection algorithm is capable of recovering the underlying sparse pattern of the actual parameters. In contrast, the column-wise Lasso algorithm tends to identify many zero parameters as non-zero. Consequently, as demonstrated in Table \ref{table.numerical_res}, the failure to recover the sparse pattern of the parameters significantly increases the estimation error.

Figure \ref{figure.performance_sec} demonstrates the performance of various algorithms concerning
different choices of hyperparameters. The threshold Lasso exhibits  performance similar to the post-selection estimator,
but the post-selection estimator is less sensitive to fluctuations in hyperparameters compared to the threshold Lasso.
Therefore, even if the model selection algorithm does not provide optimal parameters, the performance of the post-selection algorithm
is not significantly compromised. The Lasso and the column-wise Dantzig selector have relatively worse performance compared to the
threshold Lasso and the post-selection estimator.

Table \ref{table.numerical_res} evaluates the overall performance of various estimation procedures, as well as bootstrap algorithms for statistical inference.
Among the procedures used in this section, the post-selection estimator exhibits the smallest estimation errors and demonstrates low average model
misspecification, indicating its ability to correctly recover the sparse pattern of the parameters. In contrast, both the column-wise Lasso and
the column-wise Dantiz selector fail to recover the sparse pattern.
The simultaneous confidence intervals generated by the second-order wild bootstrap algorithm achieve desired coverage probability,
even in cases where the innovations are dependent or non-stationary. On the other hand, classical statistical inference methods
in multivariate time series literature, such as the autoregressive sieve bootstrap,
work well when the innovations are indeed independent. However, in situations where the
innovations are dependent or non-stationary, the autoregressive sieve bootstrap either results
in under-coverage or generates unnecessarily wide confidence intervals. This comparison reveals
that the second-order wild bootstrap has a broader range of validity compared to the autoregressive sieve bootstrap.
Practically speaking, as demonstrated in section 3.3 of \cite{MR1238940},
modelling or approximating a vector time series by an autoregressive process  with white-noise innovations is reasonable.
However, further assuming that the innovations are i.i.d. can be too restrictive.
Therefore, statisticians may choose to adopt the proposed second-order wild bootstrap for statistical inference,
thus getting rid of the strong assumptions about the innovations present in the literature.

\begin{figure}[htbp]
  \centering
     \subfigure[Real parameter]{
         \centering
         \includegraphics[width = 10cm]{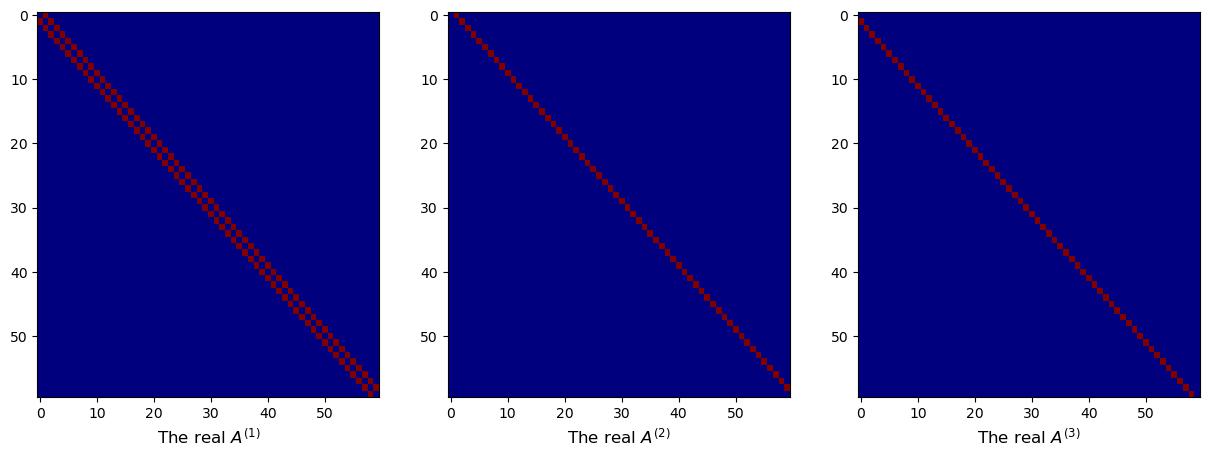}
         \label{figure.heat_1}
     }
     \subfigure[Post-selection estimator]{
         \centering
         \includegraphics[width = 10cm]{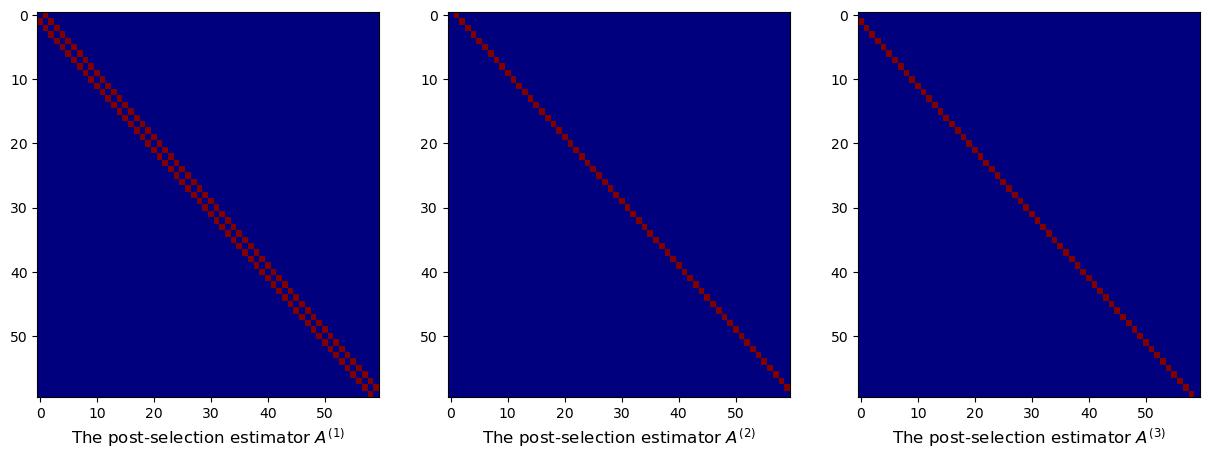}
         \label{figure.heat_2}
     }
     \subfigure[Column-wise Lasso estimator]{
         \centering
         \includegraphics[width = 10cm]{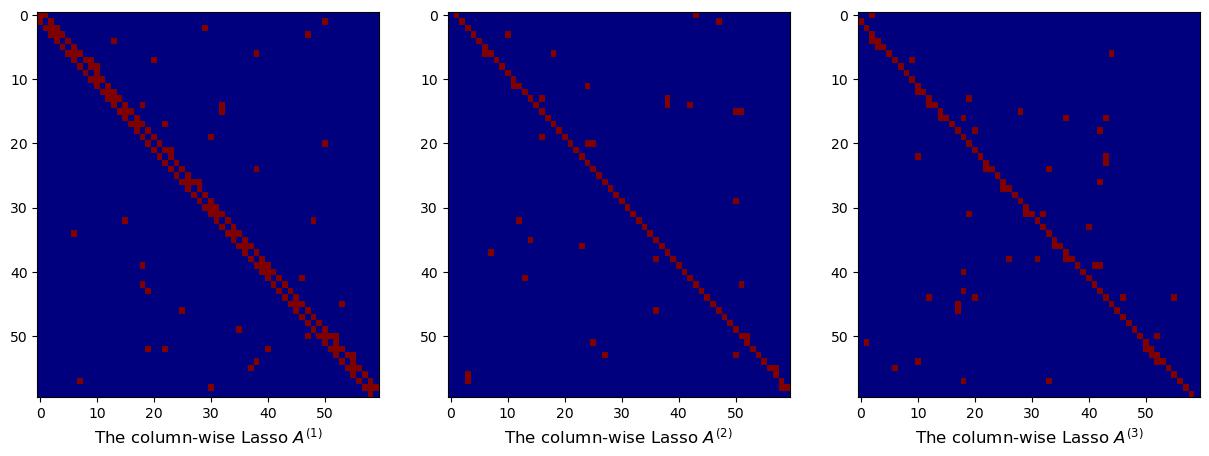}
         \label{figure.heat_3}
     }
     \caption{Figure \ref{figure.heat_1}, \ref{figure.heat_2}, and \ref{figure.heat_3} respectively depict
     the positions of the actual non-zero parameters and the corresponding estimated positions
     based on the post-selection estimator and the column-wise Lasso algorithm. These observations are generated using the settings of Experiment 9.}
     \label{figure.heat}
\end{figure}

\begin{figure}[htbp]
  \centering
  \subfigure[Independent]{
    \includegraphics[width = 9cm]{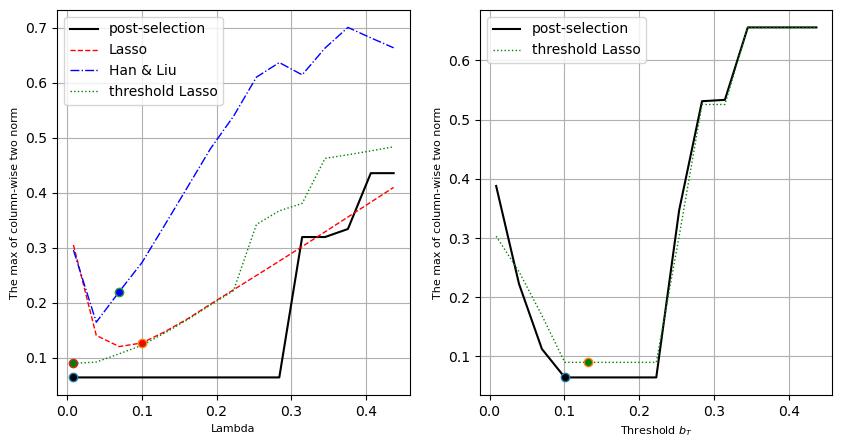}
    \label{figure.3Ind}
  }
  \subfigure[Product normal] {
    \includegraphics[width=9cm]{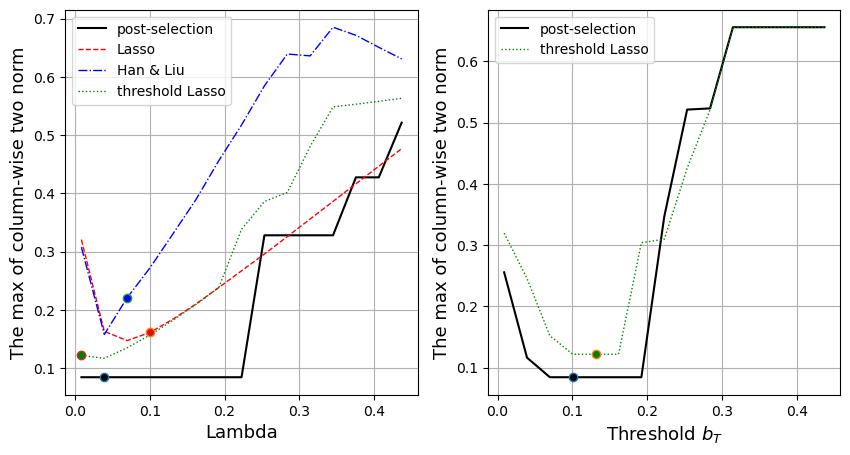}
    \label{figure.3Prd}
  }
  \subfigure[Non-stationary]{
    \includegraphics[width=9cm]{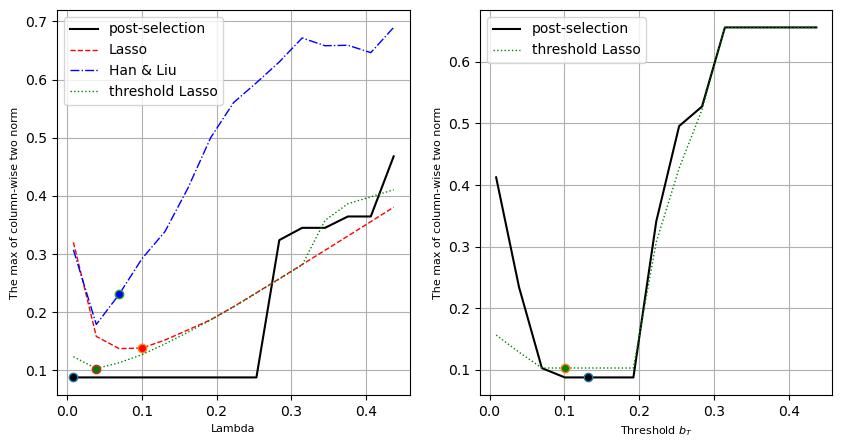}
    \label{figure.3Non}
  }
  \caption{The performance of various algorithms concerning different choices of hyperparameters.
  Figure \ref{figure.3Ind}, Figure \ref{figure.3Prd}, and Figure \ref{figure.3Non}
  respectively demonstrate the performance of algorithms when the innovations are independent, product normal,
  and non-stationary. The observations are generated using the AR(3) model, while the results for AR(1) and AR(2) situations are available
  in the online supplement.
  In these figures, the dots represent the optimal parameters selected by the method presented in Section \ref{section.numerical}.
  In the case of the post-selection algorithm and the threshold Lasso,
  the left figure plots their performance with $b_T$ as the optimal parameter, while the right figure plots their performance with $\lambda$ as the optimal parameter.}
  \label{figure.performance_sec}
\end{figure}

\begin{table}[htbp]
  \centering
  \caption{Experiment parameters selected based on methods presented in Section \ref{section.numerical}. The sample size is chosen to be 1500.}
  \scriptsize
  \begin{tabular}{c c c c c c c c c}
    \hline\hline
  Experiment\#       &   Model & Innovation     & Method         &  $d$           &  $p$   & $\lambda$ & $b_T$  &$k_T$\\
  \hline
  1            &  VAR(1)  & Independent    & Post-selection & 80                   & 1      & 0.009     & 0.131  & 1.638\\
               &         &                & Lasso          &                      &        & 0.070      \\
               &         &                & Han \& Liu     &                            &        & 0.070      \\
               &         &                & Threshold Lasso&                           &          & 0.131 & 0.070\\
               \hline
  2            & VAR(1)   & Product normal & Post-selection & 80                         &   1      & 0.009  & 0.162 & 2.054\\
               &         &                & Lasso          &                         &         & 0.070\\
               &          &                 & Han \& Liu  &                             &        & 0.070\\
               &          &               &  Threshold Lasso &                          &       & 0.162 & 0.070\\
               \hline
  3            & VAR(1)   & Nonstationary  & Post-selection &   80                    &  1       & 0.009 &0.131 & 1.895\\
               &          &               & Lasso          &                                 &           &
               0.070\\
               &          &               & Han \& Liu     &                                 &           & 0.070\\
               &          &               &  Threshold Lasso                                 &              &            & 0.009 & 0.131\\
               \hline
  4            & VAR(2)   & Independent   & Post-selection     & 70  & 2 & 0.009 & 0.131 & 1.621\\
               &          &               & Lasso             &         &   & 0.100\\
               &           &               & Han \& Liu        &          &   & 0.070\\
               &             &               & Threshold Lasso  &         &   & 0.009 & 0.131\\
  \hline
  5            &  VAR(2)  & Product normal  & Post-selection & 70                   & 2   & 0.039 & 0.131 & 2.121\\
               &         &                 & Lasso          &                         &    & 0.100 \\
               &         &                & Han \& Liu       &                           &   & 0.070\\
               &          &                 & Threshold Lasso &                             &   & 0.009 & 0.131\\
  \hline
  6            & VAR(2)   & Nonstationary   & Post-selection  & 70   & 2  & 0.009 & 0.131 & 1.924\\
               &         &                 & Lasso           &           & &  0.100\\
               &         &                 & Han \& Liu      &            &  & 0.070\\
               &        &                  & Threshold Lasso &             &  & 0.009 & 0.131\\
  \hline
  7            &  VAR(3)  & Independent    & Post-selection & 60                  & 3      & 0.009     & 0.100  & 1.596\\
               &         &                & Lasso          &                            &        & 0.100\\
               &         &                & Han \& Liu     &                            &        & 0.070\\
               &         &                & Threshold Lasso&                            &        & 0.009 & 0.131\\
               \hline
  8            &  VAR(3)  & Product normal & Post-selection & 60                   & 3            & 0.039     & 0.100 & 2.150\\
               &         &                & Lasso          &                            &         & 0.100     &       &     \\
               &         &                & Han \& Liu     &                            &              & 0.070     &       &     \\
               &         &                & Threshold Lasso &                           &              & 0.009     & 0.131 &     \\
               \hline
  9            & VAR(3)   & Nonstationary  & Post-selection  &  60                  & 3       & 0.009  &  0.131  & 1.878\\
               &         &                & Lasso           &                           &         & 0.100  \\
               &          &               & Han \& Liu      &                           &        & 0.070\\
               &        &                 & Threshold Lasso                               &              &      & 0.039 & 0.100\\
    \hline\hline
  \end{tabular}
  \label{table.experiment_parameter}
\end{table}

\begin{table}[t]
  \centering
  \caption{The performance of different estimation procedures and bootstrap algorithms. $C_1, C_2, \kappa$ respectively represent the column-wise one norm, the column-wise
  two norm, and the model misspecification. $C_1, C_2,$ and $\kappa$ are derived from the average of 300 simulations.}
  \scriptsize
  \begin{tabular}{c c c c c c c c c}
    \hline\hline
   Experiment\#        & Algorithm       & $C_1$  &   $C_2$ & $\kappa$ & \multicolumn{2}{c}{Algorithm \ref{algorithm.bootstrap}} & \multicolumn{2}{c}{AR bootstrap}\\
                    &                             &             &              &                        & Coverage & Length           & Coverage & Length\\
                    \hline
   1             &  Post-selection & 0.073     & 0.098     & 0.01          &                         $95\%$   & 0.167            & $95\%$  & 0.170\\
                 & Lasso           & 0.133     & 0.262     & 161.34\\
                 & Han \& Liu      & 0.134     & 0.263     & 6059.76\\
                 & Threshold Lasso & 0.093     & 0.123     & 0.01\\
                 \hline
   2             & Post-selection  & 0.088     & 0.114     & 0.03 & $98\%$ & 0.196            & 100\% & 0.402\\
                 & Lasso           & 0.153     & 0.304     & 161.44\\
                 & Han \& Liu      & 0.158     & 0.312     & 6058.44\\
                 & Threshold Lasso & 0.111     & 0.145     & 0.03           &                            \\
   \hline
   3             & Post-selection  & 0.080     & 0.105     & 0.0 & $94\%$ & 0.185 & $94\%$ & 0.177\\
                 & Lasso           & 0.146     & 0.287     & 161.18\\
                 & Han \& Liu      & 0.151     & 0.296     & 6057.09\\
                 & Threshold Lasso & 0.101     & 0.133     & 0.0\\
  \hline
   4             &  Post-selection  & 0.078   & 0.124     & 0.0 & $97\%$ & 0.162 & $89\%$ & 0.154\\
                 &  Lasso           & 0.148   & 0.321    & 210.58\\
                 & Han \& Liu       & 0.326   & 0.810  & 8983.81\\
                 & Threshold Lasso  & 0.106   & 0.165   & 0.0\\
  \hline
   5             &   Post-selection & 0.086   & 0.136   & 0.0 & 91\% & 0.181 & 82\% & 0.163\\
                 &   Lasso          & 0.163   & 0.353    & 210.46\\
                 &   Han \& Liu     & 0.329   & 0.804    & 8997.37\\
                 & Threshold Lasso  & 0.126   & 0.198    & 0.16\\
   \hline
   6             & Post-selection   & 0.086   & 0.136 & 0.11 & 97\% & 0.172 & 86\% & 0.159\\
                 & Lasso            & 0.154   & 0.340  & 210.18\\
                 & Han \& Liu       & 0.331   & 0.811  & 8986.21\\
                 & Threshold Lasso  & 0.117   & 0.183   & 0.11\\
   \hline

   7             &  Post-selection  & 0.076  & 0.135    & 0.39 & $92\%$ & 0.141 & $89\%$ & 0.138\\
                 &  Lasso           & 0.143  & 0.360       & 252.53\\
                 &  Han \& Liu      & 0.216  & 0.608       & 9758.25\\
                 &  Threshold Lasso & 0.102  & 0.180       & 0.0\\
  \hline
   8             &  Post-selection  & 0.079  & 0.140       & 0.06   & $93\%$   & 0.153 & $86\%$ & 0.137\\
                 &  Lasso           & 0.148  & 0.373       & 241.96&\\
                 &  Han \& Liu      & 0.222  & 0.619       & 9771.78\\
                 &  Threshold Lasso & 0.112  & 0.195       & 0.04\\
   \hline
   9             &  Post-selection  & 0.074  & 0.133       & 0.00   & $90\%$    & 0.148  & $80\%$ & 0.135\\
                 &  Lasso           & 0.145  & 0.372       & 242.88\\
                 & Han \& Liu       & 0.219  & 0.620       & 9770.23\\
                 & Threshold Lasso  & 0.106  & 0.190       & 0.02\\
    \hline\hline
  \end{tabular}
 \label{table.numerical_res}
\end{table}

\textbf{Real data example: } We analyze the annual global electricity generation data
presented in \url{https://www.kaggle.com/datasets/akhiljethwa/global-electricity-statistics},
which records the annual electricity generation of different countries from 1980 to 2021.
Our focus is on studying the relationships in electricity generation among ASEAN countries.
Since the original data is not stationary, we perform a $\log$ transformation and
remove its linear trend before conducting statistical inference.
Figure \ref{figure.detrended} presents the transformed data.
For each country, we employ the augmented Dickey-Fuller test to verify the stationarity
of the time series.

The estimation results are presented in Figure \ref{figure.AR_coef} and Table \ref{table.real_result}.
As shown in Figure \ref{figure.AR_coef}, the electricity generation in Cambodia and Laos strongly
depends on their respective previous year's electricity generations. Besides, the electricity generation in Thailand
negatively depends on the previous year's electricity generation in Laos. This phenomenon can be attributed to the
export of hydropower generated from dams on the Mekong River in Laos to fulfill the electricity demand in Thailand.

In table \ref{table.real_result}, the width of the confidence interval
generated by the dependent wild bootstrap is different from that generated by
the autoregressive sieve bootstrap, suggesting violations in the assumption of
independent and identically distributed innovations. Since the dependent wild bootstrap algorithm
is robust to errors' dependency or non-stationarity, adopting it provides the practitioners with a
confidence interval that closely reflects the real-world situations.

\begin{figure}
  \centering
  \subfigure[Detrended $\log$ electricity generation]{
      \centering
     \includegraphics[width = 8cm]{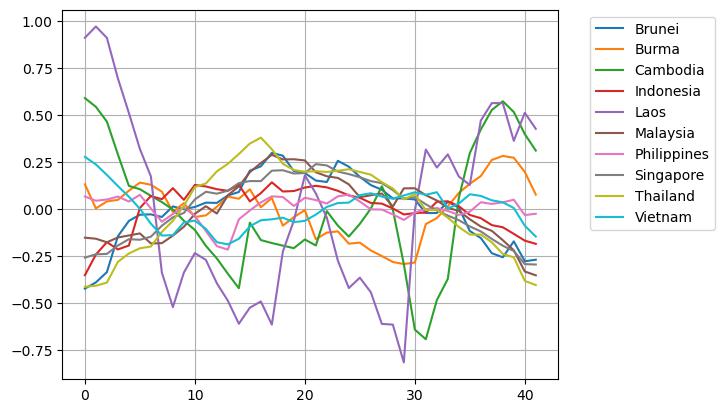}
     \label{figure.detrended}
  }
  \subfigure[Estimated autoregressive coefficients]{
     \centering
     \includegraphics[width = 6cm]{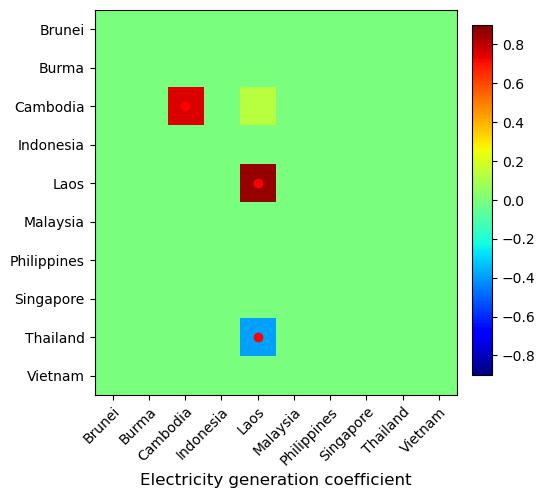}
     \label{figure.AR_coef}
  }
  \caption{The detrended $\log$ transform of the annual electricity generation data in ASEAN countries,
  and the estimated AR coefficients. Red dots indicate statistically significant coefficients, i.e., that $0$ is not contained within their confidence intervals.}
  \label{}
\end{figure}

\begin{table}
  \centering
  \caption{Description and results of the annual global electricity consumption data}
  \scriptsize
  \begin{tabular}{c c c c c c c c}
    \hline
    \hline
    Model & $d$ & Sample size & $\lambda$ & $b_T$  & $k_T$ & \multicolumn{2}{c}{Length}\\
          &     &             &           &        &       & Algorithm \ref{algorithm.bootstrap} & AR bootstrap\\
    VAR(1)& 10   & 42          & 0.063     & 0.063  & 3.775 & 0.271   & 1.724\\
    \hline
    \hline
  \end{tabular}
  \label{table.real_result}
\end{table}

\section{Conclusion}
\label{section.conclusion}
This paper focuses on the statistical inference of a sparse vector autoregressive time series
without assuming that its innovations are i.i.d. To achieve this goal, it adopts the notion
of $(m,\alpha)-$short range dependent random variables presented in \cite{zhangPolitis} and  \cite{MR3779697} to model the time series.
After that, it utilizes a post-selection estimator to fit the time series, which maintains sparsity in the autoregressive coefficient matrices.
To help statisticians getting rid of complex calculations, this paper modifies the dependent wild bootstrap method
introduced in \cite{MR2656050}. While the original dependent wild bootstrap was designed to analyze first-order statistics, such as the sample mean,
of a stationary time series; our work adapts it to second-order statistics,
such as the sample autoregressive coefficients, of a non-stationary time series.
This paper theoretically derives the asymptotic distribution of the proposed estimator
and establishes the validity of the second-order wild bootstrap algorithm (Algorithm \ref{algorithm.bootstrap}).
Numerical simulations and real-life data experiments confirm the correctness
of the asymptotic distribution of the estimator and the validity of the bootstrap algorithm.

One of the  essential conditions in the literature on (vector) autoregressive time series is
that the innovations in the time series are independent or i.i.d. However, this condition is often too
restrictive for real-life data to satisfy or too hard to verify through hypothesis tests.
Worse still, as demonstrated in Table \ref{table.numerical_res}, ignoring this issue can result
in excessively wide confidence intervals or significant under-coverage. By allowing for the presence of
dependent and non-stationary innovations in the autoregressive time series, the confidence intervals
derived in our work should better reflect real-world situations compared to
the methods in the literature that rely on the i.i.d. innovations.

\printbibliography

@book {MR2002723,
    AUTHOR = {Shao, Jun},
     TITLE = {Mathematical statistics},
    SERIES = {Springer Texts in Statistics},
   EDITION = {Second},
 PUBLISHER = {Springer-Verlag, New York},
      YEAR = {2003},
     PAGES = {xvi+591},
      ISBN = {0-387-95382-5},
   MRCLASS = {62-01},
  MRNUMBER = {2002723},
       DOI = {10.1007/b97553},
       URL = {https://doi.org/10.1007/b97553},
}

@book {MR2978290,
    AUTHOR = {Horn, Roger A. and Johnson, Charles R.},
     TITLE = {Matrix analysis},
   EDITION = {Second},
 PUBLISHER = {Cambridge University Press, Cambridge},
      YEAR = {2013},
     PAGES = {xviii+643},
      ISBN = {978-0-521-54823-6},
   MRCLASS = {15-01},
  MRNUMBER = {2978290},
MRREVIEWER = {Mohammad Sal Moslehian},
}

@article {MR3153940,
    AUTHOR = {Zhang, Cun-Hui and Zhang, Stephanie S.},
     TITLE = {Confidence intervals for low dimensional parameters in high
              dimensional linear models},
   JOURNAL = {J. R. Stat. Soc. Ser. B. Stat. Methodol.},
  FJOURNAL = {Journal of the Royal Statistical Society. Series B.
              Statistical Methodology},
    VOLUME = {76},
      YEAR = {2014},
    NUMBER = {1},
     PAGES = {217--242},
      ISSN = {1369-7412},
   MRCLASS = {62F25 (62J05)},
  MRNUMBER = {3153940},
       DOI = {10.1111/rssb.12026},
       URL = {https://doi.org/10.1111/rssb.12026},
}

@article {MR4325665,
    AUTHOR = {Krampe, Jonas and Paparoditis, Efstathios},
     TITLE = {Sparsity concepts and estimation procedures for
              high-dimensional vector autoregressive models},
   JOURNAL = {J. Time Series Anal.},
  FJOURNAL = {Journal of Time Series Analysis},
    VOLUME = {42},
      YEAR = {2021},
    NUMBER = {5-6},
     PAGES = {554--579},
      ISSN = {0143-9782},
   MRCLASS = {Expansion},
  MRNUMBER = {4325665},
       DOI = {10.1111/jtsa.12586},
       URL = {https://doi.org/10.1111/jtsa.12586},
}

@article {MR3357870,
    AUTHOR = {Basu, Sumanta and Michailidis, George},
     TITLE = {Regularized estimation in sparse high-dimensional time series
              models},
   JOURNAL = {Ann. Statist.},
  FJOURNAL = {The Annals of Statistics},
    VOLUME = {43},
      YEAR = {2015},
    NUMBER = {4},
     PAGES = {1535--1567},
      ISSN = {0090-5364},
   MRCLASS = {62M10 (62J07 62M15)},
  MRNUMBER = {3357870},
MRREVIEWER = {Yuzo Hosoya},
       DOI = {10.1214/15-AOS1315},
       URL = {https://doi.org/10.1214/15-AOS1315},
}

@article {MR3718156,
    AUTHOR = {Zhang, Danna and Wu, Wei Biao},
     TITLE = {Gaussian approximation for high dimensional time series},
   JOURNAL = {Ann. Statist.},
  FJOURNAL = {The Annals of Statistics},
    VOLUME = {45},
      YEAR = {2017},
    NUMBER = {5},
     PAGES = {1895--1919},
      ISSN = {0090-5364},
   MRCLASS = {62M10 (62E17)},
  MRNUMBER = {3718156},
MRREVIEWER = {Martin Wendler},
       DOI = {10.1214/16-AOS1512},
       URL = {https://doi.org/10.1214/16-AOS1512},
}

@article {MR3161448,
    AUTHOR = {Chernozhukov, Victor and Chetverikov, Denis and Kato, Kengo},
     TITLE = {Gaussian approximations and multiplier bootstrap for maxima of
              sums of high-dimensional random vectors},
   JOURNAL = {Ann. Statist.},
  FJOURNAL = {The Annals of Statistics},
    VOLUME = {41},
      YEAR = {2013},
    NUMBER = {6},
     PAGES = {2786--2819},
      ISSN = {0090-5364},
   MRCLASS = {62E17 (62F40)},
  MRNUMBER = {3161448},
MRREVIEWER = {Palaniappan Vellaisamy},
       DOI = {10.1214/13-AOS1161},
       URL = {https://doi.org/10.1214/13-AOS1161},
}

@article {MR4441125,
    AUTHOR = {Zhang, Yunyi and Politis, Dimitris N.},
     TITLE = {Ridge regression revisited: debiasing, thresholding and
              bootstrap},
   JOURNAL = {Ann. Statist.},
  FJOURNAL = {The Annals of Statistics},
    VOLUME = {50},
      YEAR = {2022},
    NUMBER = {3},
     PAGES = {1401--1422},
      ISSN = {0090-5364},
   MRCLASS = {62J05 (62F40)},
  MRNUMBER = {4441125},
       DOI = {10.1214/21-aos2156},
       URL = {https://doi.org/10.1214/21-aos2156},
}

@article {MR4206676,
    AUTHOR = {Zhang, Danna and Wu, Wei Biao},
     TITLE = {Convergence of covariance and spectral density estimates for
              high-dimensional locally stationary processes},
   JOURNAL = {Ann. Statist.},
  FJOURNAL = {The Annals of Statistics},
    VOLUME = {49},
      YEAR = {2021},
    NUMBER = {1},
     PAGES = {233--254},
      ISSN = {0090-5364},
   MRCLASS = {62M10 (62M15)},
  MRNUMBER = {4206676},
MRREVIEWER = {Yasumasa Matsuda},
       DOI = {10.1214/20-AOS1954},
       URL = {https://doi.org/10.1214/20-AOS1954},
}

@article{MESSNER20191485,
title = {Online adaptive lasso estimation in vector autoregressive models for high dimensional wind power forecasting},
journal = {International Journal of Forecasting},
volume = {35},
number = {4},
pages = {1485-1498},
year = {2019},
issn = {0169-2070},
doi = {https://doi.org/10.1016/j.ijforecast.2018.02.001},
url = {https://www.sciencedirect.com/science/article/pii/S0169207018300347},
author = {Jakob W. Messner and Pierre Pinson}
}

@article{HANNAFORD2023107659,
title = {A sparse Bayesian hierarchical vector autoregressive model for microbial dynamics in a wastewater treatment plant},
journal = {Computational Statistics \& Data Analysis},
volume = {179},
pages = {107659},
year = {2023},
issn = {0167-9473},
doi = {https://doi.org/10.1016/j.csda.2022.107659},
url = {https://www.sciencedirect.com/science/article/pii/S0167947322002390},
author = {Naomi E. Hannaford and Sarah E. Heaps and Tom M.W. Nye and Thomas P. Curtis and Ben Allen and Andrew Golightly and Darren J. Wilkinson}
}

@article{HE2021105171,
title = {Comparing long monthly Chinese and selected European temperature series using the Vector Seasonal Shifting Mean and Covariance Autoregressive model},
journal = {Energy Economics},
volume = {97},
pages = {105171},
year = {2021},
issn = {0140-9883},
doi = {https://doi.org/10.1016/j.eneco.2021.105171},
url = {https://www.sciencedirect.com/science/article/pii/S0140988321000761},
author = {Changli He and Jian Kang and Timo Teräsvirta and Shuhua Zhang}
}

@ARTICLE{RePEc:tpr:restat:v:68:y:1986:i:4:p:628-37,
title = {The Canadian-U.S. Exchange Rate: Evidence from a Vector Autoregression},
author = {Backus, David},
year = {1986},
journal = {The Review of Economics and Statistics},
volume = {68},
number = {4},
pages = {628-37},
url = {https://EconPapers.repec.org/RePEc:tpr:restat:v:68:y:1986:i:4:p:628-37}
}

@article{10.2307/2328190,
 ISSN = {00221082, 15406261},
 URL = {http://www.jstor.org/stable/2328190},
 author = {John Y. Campbell and Robert J. Shiller},
 journal = {The Journal of Finance},
 number = {3},
 pages = {661--676},
 publisher = {[American Finance Association, Wiley]},
 title = {Stock Prices, Earnings, and Expected Dividends},
 urldate = {2023-05-14},
 volume = {43},
 year = {1988}
}

@article{10.1093/icc/dtq018,
    author = {Coad, Alex},
    title = "{Exploring the processes of firm growth: evidence from a vector auto-regression}",
    journal = {Industrial and Corporate Change},
    volume = {19},
    number = {6},
    pages = {1677-1703},
    year = {2010},
    month = {04},
    issn = {0960-6491},
    doi = {10.1093/icc/dtq018},
    url = {https://doi.org/10.1093/icc/dtq018},
    eprint = {https://academic.oup.com/icc/article-pdf/19/6/1677/2156142/dtq018.pdf},
}

@article{Varforcast,
    author = {Liu, Yixian and Roberts, Matthew C. and Sioshansi Ramteen},
    title = "{A vector autoregression weather model for electricity supply
and demand modeling}",
    journal = {J. Mod. Power Syst. Clean Energy},
    volume = {6},
    number = {4},
    pages = {763–776},
    year = {2018},
    month = {04}
}

@article{DIAS201875,
title = {Estimation and forecasting in vector autoregressive moving average models for rich datasets},
journal = {Journal of Econometrics},
volume = {202},
number = {1},
pages = {75-91},
year = {2018},
issn = {0304-4076},
doi = {https://doi.org/10.1016/j.jeconom.2017.06.022},
url = {https://www.sciencedirect.com/science/article/pii/S0304407617301707},
author = {Gustavo Fruet Dias and George Kapetanios}
}

@article {MR4480716,
    AUTHOR = {Wang, Di and Zheng, Yao and Lian, Heng and Li, Guodong},
     TITLE = {High-dimensional vector autoregressive time series modeling
              via tensor decomposition},
   JOURNAL = {J. Amer. Statist. Assoc.},
  FJOURNAL = {Journal of the American Statistical Association},
    VOLUME = {117},
      YEAR = {2022},
    NUMBER = {539},
     PAGES = {1338--1356},
      ISSN = {0162-1459},
   MRCLASS = {Expansion},
  MRNUMBER = {4480716},
       DOI = {10.1080/01621459.2020.1855183},
       URL = {https://doi.org/10.1080/01621459.2020.1855183},
}

@book {MR1093459,
    AUTHOR = {Brockwell, Peter J. and Davis, Richard A.},
     TITLE = {Time series: theory and methods},
    SERIES = {Springer Series in Statistics},
   EDITION = {Second},
 PUBLISHER = {Springer-Verlag, New York},
      YEAR = {1991},
     PAGES = {xvi+577},
      ISBN = {0-387-97429-6},
   MRCLASS = {62-01 (62M10)},
  MRNUMBER = {1093459},
       DOI = {10.1007/978-1-4419-0320-4},
       URL = {https://doi.org/10.1007/978-1-4419-0320-4},
}

@book {MR3526245,
    AUTHOR = {Brockwell, Peter J. and Davis, Richard A.},
     TITLE = {Introduction to time series and forecasting},
    SERIES = {Springer Texts in Statistics},
   EDITION = {Third},
 PUBLISHER = {Springer, [Cham]},
      YEAR = {2016},
     PAGES = {xiv+425},
      ISBN = {978-3-319-29852-8; 978-3-319-29854-2},
   MRCLASS = {62-01 (60G10 62M10 62M20)},
  MRNUMBER = {3526245},
MRREVIEWER = {Wilfredo Palma},
       DOI = {10.1007/978-3-319-29854-2},
       URL = {https://doi.org/10.1007/978-3-319-29854-2},
}

@book {MR1238940,
    AUTHOR = {Reinsel, Gregory C.},
     TITLE = {Elements of multivariate time series analysis},
    SERIES = {Springer Series in Statistics},
 PUBLISHER = {Springer-Verlag, New York},
      YEAR = {1993},
     PAGES = {xiv+263},
      ISBN = {0-387-94063-4},
   MRCLASS = {62-01 (62M10)},
  MRNUMBER = {1238940},
MRREVIEWER = {B.\ G.\ Quinn},
       DOI = {10.1007/978-1-4684-0198-1},
       URL = {https://doi.org/10.1007/978-1-4684-0198-1},
}

@article {MR3235390,
    AUTHOR = {Xiao, Han and Wu, Wei Biao},
     TITLE = {Portmanteau test and simultaneous inference for serial
              covariances},
   JOURNAL = {Statist. Sinica},
  FJOURNAL = {Statistica Sinica},
    VOLUME = {24},
      YEAR = {2014},
    NUMBER = {2},
     PAGES = {577--599},
      ISSN = {1017-0405},
   MRCLASS = {60G10 (62G10 62G20 62G32 62H20)},
  MRNUMBER = {3235390},
}

@article{ESCANCIANO2009140,
title = {An automatic Portmanteau test for serial correlation},
journal = {Journal of Econometrics},
volume = {151},
number = {2},
pages = {140-149},
year = {2009},
note = {Recent Advances in Time Series Analysis: A Volume Honouring Peter M. Robinson},
issn = {0304-4076},
doi = {https://doi.org/10.1016/j.jeconom.2009.03.001},
url = {https://www.sciencedirect.com/science/article/pii/S0304407609000773},
author = {J. Carlos Escanciano and Ignacio N. Lobato}
}

@article{zhang2023simultaneous,
      title={Simultaneous Statistical Inference for Second Order Parameters of Time Series under Weak Conditions}, 
      author={Yunyi Zhang and Efstathios Paparoditis and Dimitris N. Politis},
      year={2023},
      journal = {Arxiv id: 2110.14067},
      archivePrefix={arXiv},
      primaryClass={math.ST}
}

@article {MR2274449,
    AUTHOR = {Zhao, Peng and Yu, Bin},
     TITLE = {On model selection consistency of {L}asso},
   JOURNAL = {J. Mach. Learn. Res.},
  FJOURNAL = {Journal of Machine Learning Research (JMLR)},
    VOLUME = {7},
      YEAR = {2006},
     PAGES = {2541--2563},
      ISSN = {1532-4435},
   MRCLASS = {62C05 (62J10)},
  MRNUMBER = {2274449},
}

@article {MR1379242,
    AUTHOR = {Tibshirani, Robert},
     TITLE = {Regression shrinkage and selection via the lasso},
   JOURNAL = {J. Roy. Statist. Soc. Ser. B},
  FJOURNAL = {Journal of the Royal Statistical Society. Series B.
              Methodological},
    VOLUME = {58},
      YEAR = {1996},
    NUMBER = {1},
     PAGES = {267--288},
      ISSN = {0035-9246},
   MRCLASS = {62J05 (62J07)},
  MRNUMBER = {1379242},
       URL =
              {http://links.jstor.org/sici?sici=0035-9246(1996)58:1<267:RSASVT>2.0.CO;2-G&origin=MSN},
}

@article {MR4446423,
    AUTHOR = {Masini, Ricardo P. and Medeiros, Marcelo C. and Mendes,
              Eduardo F.},
     TITLE = {Regularized estimation of high-dimensional vector
              autoregressions with weakly dependent innovations},
   JOURNAL = {J. Time Series Anal.},
  FJOURNAL = {Journal of Time Series Analysis},
    VOLUME = {43},
      YEAR = {2022},
    NUMBER = {4},
     PAGES = {532--557},
      ISSN = {0143-9782},
  MRNUMBER = {4446423},
       DOI = {10.1111/jtsa.12627},
       URL = {https://doi.org/10.1111/jtsa.12627},
}

@article {MR4102690,
    AUTHOR = {Wong, Kam Chung and Li, Zifan and Tewari, Ambuj},
     TITLE = {Lasso guarantees for {$\beta$}-mixing heavy-tailed time
              series},
   JOURNAL = {Ann. Statist.},
  FJOURNAL = {The Annals of Statistics},
    VOLUME = {48},
      YEAR = {2020},
    NUMBER = {2},
     PAGES = {1124--1142},
      ISSN = {0090-5364},
   MRCLASS = {62M10 (62J07)},
  MRNUMBER = {4102690},
MRREVIEWER = {Anoop Chaturvedi},
       DOI = {10.1214/19-AOS1840},
       URL = {https://doi.org/10.1214/19-AOS1840},
}

@article {MR3343790,
    AUTHOR = {Kock, Anders Bredahl and Callot, Laurent},
     TITLE = {Oracle inequalities for high dimensional vector
              autoregressions},
   JOURNAL = {J. Econometrics},
  FJOURNAL = {Journal of Econometrics},
    VOLUME = {186},
      YEAR = {2015},
    NUMBER = {2},
     PAGES = {325--344},
      ISSN = {0304-4076},
   MRCLASS = {62J02 (91G70)},
  MRNUMBER = {3343790},
       DOI = {10.1016/j.jeconom.2015.02.013},
       URL = {https://doi.org/10.1016/j.jeconom.2015.02.013},
}

@article {MR4278792,
    AUTHOR = {Krampe, Jonas and Kreiss, Jens-Peter and Paparoditis,
              Efstathios},
     TITLE = {Bootstrap based inference for sparse high-dimensional time
              series models},
   JOURNAL = {Bernoulli},
  FJOURNAL = {Bernoulli. Official Journal of the Bernoulli Society for
              Mathematical Statistics and Probability},
    VOLUME = {27},
      YEAR = {2021},
    NUMBER = {3},
     PAGES = {1441--1466},
      ISSN = {1350-7265},
   MRCLASS = {62M10 (62F03 62F12 62F40)},
  MRNUMBER = {4278792},
       DOI = {10.3150/20-bej1239},
       URL = {https://doi.org/10.3150/20-bej1239},
}

@article {MR3450535,
    AUTHOR = {Han, Fang and Lu, Huanran and Liu, Han},
     TITLE = {A direct estimation of high dimensional stationary vector
              autoregressions},
   JOURNAL = {J. Mach. Learn. Res.},
  FJOURNAL = {Journal of Machine Learning Research (JMLR)},
    VOLUME = {16},
      YEAR = {2015},
     PAGES = {3115--3150},
      ISSN = {1532-4435},
   MRCLASS = {62M10 (62H12 62J07)},
  MRNUMBER = {3450535},
MRREVIEWER = {Biljana \v{C}. Popovi\'{c}},
}

@article {MR2847973,
    AUTHOR = {Cai, Tony and Liu, Weidong and Luo, Xi},
     TITLE = {A constrained {$\ell_1$} minimization approach to sparse
              precision matrix estimation},
   JOURNAL = {J. Amer. Statist. Assoc.},
  FJOURNAL = {Journal of the American Statistical Association},
    VOLUME = {106},
      YEAR = {2011},
    NUMBER = {494},
     PAGES = {594--607},
      ISSN = {0162-1459},
   MRCLASS = {62H12 (62F12 62J07 62P10)},
  MRNUMBER = {2847973},
MRREVIEWER = {Goetz Trenkler},
       DOI = {10.1198/jasa.2011.tm10155},
       URL = {https://doi.org/10.1198/jasa.2011.tm10155},
}

@article {MR3620446,
    AUTHOR = {Guo, Shaojun and Wang, Yazhen and Yao, Qiwei},
     TITLE = {High-dimensional and banded vector autoregressions},
   JOURNAL = {Biometrika},
  FJOURNAL = {Biometrika},
    VOLUME = {103},
      YEAR = {2016},
    NUMBER = {4},
     PAGES = {889--903},
      ISSN = {0006-3444},
   MRCLASS = {62M10 (62F12 62H12)},
  MRNUMBER = {3620446},
       DOI = {10.1093/biomet/asw046},
       URL = {https://doi.org/10.1093/biomet/asw046},
}

@article {MR4259144,
    AUTHOR = {Zheng, Yao and Cheng, Guang},
     TITLE = {Finite-time analysis of vector autoregressive models under
              linear restrictions},
   JOURNAL = {Biometrika},
  FJOURNAL = {Biometrika},
    VOLUME = {108},
      YEAR = {2021},
    NUMBER = {2},
     PAGES = {469--489},
      ISSN = {0006-3444},
   MRCLASS = {62M10 (62F12)},
  MRNUMBER = {4259144},
MRREVIEWER = {Nazar\'{e} Mendes Lopes},
       DOI = {10.1093/biomet/asaa065},
       URL = {https://doi.org/10.1093/biomet/asaa065},
}

@article {MR3725456,
    AUTHOR = {Lin, Jiahe and Michailidis, George},
     TITLE = {Regularized estimation and testing for high-dimensional
              multi-block vector-autoregressive models},
   JOURNAL = {J. Mach. Learn. Res.},
  FJOURNAL = {Journal of Machine Learning Research (JMLR)},
    VOLUME = {18},
      YEAR = {2017},
     PAGES = {Paper No. 117, 49},
      ISSN = {1532-4435},
   MRCLASS = {62H12 (62H15 62M10)},
  MRNUMBER = {3725456},
MRREVIEWER = {Konrad Furma\'{n}czyk},
       DOI = {10.1631/jzus.a1500279},
       URL = {https://doi.org/10.1631/jzus.a1500279},
}

@article {MR4561044,
    AUTHOR = {Krampe, J. and Paparoditis, E. and Trenkler, C.},
     TITLE = {Structural inference in sparse high-dimensional vector
              autoregressions},
   JOURNAL = {J. Econometrics},
  FJOURNAL = {Journal of Econometrics},
    VOLUME = {234},
      YEAR = {2023},
    NUMBER = {1},
     PAGES = {276--300},
      ISSN = {0304-4076},
   MRCLASS = {Expansion},
  MRNUMBER = {4561044},
       DOI = {10.1016/j.jeconom.2022.01.003},
       URL = {https://doi.org/10.1016/j.jeconom.2022.01.003},
}

@article {MR3662449,
    AUTHOR = {Zhu, Xuening and Pan, Rui and Li, Guodong and Liu, Yuewen and
              Wang, Hansheng},
     TITLE = {Network vector autoregression},
   JOURNAL = {Ann. Statist.},
  FJOURNAL = {The Annals of Statistics},
    VOLUME = {45},
      YEAR = {2017},
    NUMBER = {3},
     PAGES = {1096--1123},
      ISSN = {0090-5364},
   MRCLASS = {62M10 (62H99 62J05)},
  MRNUMBER = {3662449},
       DOI = {10.1214/16-AOS1476},
       URL = {https://doi.org/10.1214/16-AOS1476},
}

@article {MR2485008,
    AUTHOR = {Bickel, Peter J. and Levina, Elizaveta},
     TITLE = {Covariance regularization by thresholding},
   JOURNAL = {Ann. Statist.},
  FJOURNAL = {The Annals of Statistics},
    VOLUME = {36},
      YEAR = {2008},
    NUMBER = {6},
     PAGES = {2577--2604},
      ISSN = {0090-5364},
   MRCLASS = {62H12 (62F12 62G09)},
  MRNUMBER = {2485008},
MRREVIEWER = {M. Hu\v{s}kov\'{a}},
       DOI = {10.1214/08-AOS600},
       URL = {https://doi.org/10.1214/08-AOS600},
}

@article {MR2893863,
    AUTHOR = {Kreiss, Jens-Peter and Paparoditis, Efstathios and Politis,
              Dimitris N.},
     TITLE = {On the range of validity of the autoregressive sieve
              bootstrap},
   JOURNAL = {Ann. Statist.},
  FJOURNAL = {The Annals of Statistics},
    VOLUME = {39},
      YEAR = {2011},
    NUMBER = {4},
     PAGES = {2103--2130},
      ISSN = {0090-5364},
   MRCLASS = {62M10 (62G09 62M15)},
  MRNUMBER = {2893863},
       DOI = {10.1214/11-AOS900},
       URL = {https://doi.org/10.1214/11-AOS900},
}

@article {MR4325666,
    AUTHOR = {Wang, Jiang and Politis, Dimitris N.},
     TITLE = {Consistent autoregressive spectral estimates: nonlinear time
              series and large autocovariance matrices},
   JOURNAL = {J. Time Series Anal.},
  FJOURNAL = {Journal of Time Series Analysis},
    VOLUME = {42},
      YEAR = {2021},
    NUMBER = {5-6},
     PAGES = {580--596},
      ISSN = {0143-9782},
   MRCLASS = {Expansion},
  MRNUMBER = {4325666},
       DOI = {10.1111/jtsa.12580},
       URL = {https://doi.org/10.1111/jtsa.12580},
}

@article {MR2732601,
    AUTHOR = {McMurry, Timothy L. and Politis, Dimitris N.},
     TITLE = {Banded and tapered estimates for autocovariance matrices and
              the linear process bootstrap},
   JOURNAL = {J. Time Series Anal.},
  FJOURNAL = {Journal of Time Series Analysis},
    VOLUME = {31},
      YEAR = {2010},
    NUMBER = {6},
     PAGES = {471--482},
      ISSN = {0143-9782},
   MRCLASS = {62M10 (60G10 62G09 62G20 62H12)},
  MRNUMBER = {2732601},
MRREVIEWER = {Pedro Galeano},
       DOI = {10.1111/j.1467-9892.2010.00679.x},
       URL = {https://doi.org/10.1111/j.1467-9892.2010.00679.x},
}

@article {MR3346699,
    AUTHOR = {Jentsch, Carsten and Politis, Dimitris N.},
     TITLE = {Covariance matrix estimation and linear process bootstrap for
              multivariate time series of possibly increasing dimension},
   JOURNAL = {Ann. Statist.},
  FJOURNAL = {The Annals of Statistics},
    VOLUME = {43},
      YEAR = {2015},
    NUMBER = {3},
     PAGES = {1117--1140},
      ISSN = {0090-5364},
   MRCLASS = {62G09 (62M10)},
  MRNUMBER = {3346699},
       DOI = {10.1214/14-AOS1301},
       URL = {https://doi.org/10.1214/14-AOS1301},
}

@article {MR3796524,
    AUTHOR = {Fragkeskou, Maria and Paparoditis, Efstathios},
     TITLE = {Extending the range of validity of the autoregressive (sieve)
              bootstrap},
   JOURNAL = {J. Time Series Anal.},
  FJOURNAL = {Journal of Time Series Analysis},
    VOLUME = {39},
      YEAR = {2018},
    NUMBER = {3},
     PAGES = {356--379},
      ISSN = {0143-9782},
   MRCLASS = {62M10 (62G09)},
  MRNUMBER = {3796524},
MRREVIEWER = {Zuzana Pr\'{a}\v{s}kov\'{a}},
       DOI = {10.1111/jtsa.12275},
       URL = {https://doi.org/10.1111/jtsa.12275},
}

@article {MR3331856,
    AUTHOR = {McMurry, Timothy L. and Politis, Dimitris N.},
     TITLE = {High-dimensional autocovariance matrices and optimal linear
              prediction},
   JOURNAL = {Electron. J. Stat.},
  FJOURNAL = {Electronic Journal of Statistics},
    VOLUME = {9},
      YEAR = {2015},
    NUMBER = {1},
     PAGES = {753--788},
   MRCLASS = {62M10 (62G20 62H12 62M15 62M20)},
  MRNUMBER = {3331856},
       DOI = {10.1214/15-EJS1000},
       URL = {https://doi.org/10.1214/15-EJS1000},
}

@article {MR4025739,
    AUTHOR = {Zhu, Ke},
     TITLE = {Statistical inference for autoregressive models under
              heteroscedasticity of unknown form},
   JOURNAL = {Ann. Statist.},
  FJOURNAL = {The Annals of Statistics},
    VOLUME = {47},
      YEAR = {2019},
    NUMBER = {6},
     PAGES = {3185--3215},
      ISSN = {0090-5364},
   MRCLASS = {62F03 (62F12 62F35 62M10)},
  MRNUMBER = {4025739},
MRREVIEWER = {Teo Sharia},
       DOI = {10.1214/18-AOS1775},
       URL = {https://doi.org/10.1214/18-AOS1775},
}

@article {MR4270034,
    AUTHOR = {Das, Srinjoy and Politis, Dimitris N.},
     TITLE = {Predictive inference for locally stationary time series with
              an application to climate data},
   JOURNAL = {J. Amer. Statist. Assoc.},
  FJOURNAL = {Journal of the American Statistical Association},
    VOLUME = {116},
      YEAR = {2021},
    NUMBER = {534},
     PAGES = {919--934},
      ISSN = {0162-1459},
   MRCLASS = {62M10 (62M20 62P12)},
  MRNUMBER = {4270034},
MRREVIEWER = {Wilfredo Palma},
       DOI = {10.1080/01621459.2019.1708368},
       URL = {https://doi.org/10.1080/01621459.2019.1708368},
}

@article {MR3299408,
    AUTHOR = {Kreiss, Jens-Peter and Paparoditis, Efstathios},
     TITLE = {Bootstrapping locally stationary processes},
   JOURNAL = {J. R. Stat. Soc. Ser. B. Stat. Methodol.},
  FJOURNAL = {Journal of the Royal Statistical Society. Series B.
              Statistical Methodology},
    VOLUME = {77},
      YEAR = {2015},
    NUMBER = {1},
     PAGES = {267--290},
      ISSN = {1369-7412},
   MRCLASS = {62D05 (60F05 60G10 62F40 62G07 62P20)},
  MRNUMBER = {3299408},
MRREVIEWER = {Alfredas Ra\v{c}kauskas},
       DOI = {10.1111/rssb.12068},
       URL = {https://doi.org/10.1111/rssb.12068},
}

@article {MR3920364,
    AUTHOR = {Dahlhaus, Rainer and Richter, Stefan and Wu, Wei Biao},
     TITLE = {Towards a general theory for nonlinear locally stationary
              processes},
   JOURNAL = {Bernoulli},
  FJOURNAL = {Bernoulli. Official Journal of the Bernoulli Society for
              Mathematical Statistics and Probability},
    VOLUME = {25},
      YEAR = {2019},
    NUMBER = {2},
     PAGES = {1013--1044},
      ISSN = {1350-7265},
   MRCLASS = {60G10 (60F05 60J25 62M09)},
  MRNUMBER = {3920364},
       DOI = {10.3150/17-bej1011},
       URL = {https://doi.org/10.3150/17-bej1011},
}

@article {MR2190207,
    AUTHOR = {Lai, Tze Leung and Liu, Haiyan and Xing, Haipeng},
     TITLE = {Autoregressive models with piecewise constant volatility and
              regression parameters},
   JOURNAL = {Statist. Sinica},
  FJOURNAL = {Statistica Sinica},
    VOLUME = {15},
      YEAR = {2005},
    NUMBER = {2},
     PAGES = {279--301},
      ISSN = {1017-0405},
   MRCLASS = {62M10 (62F15 62M09)},
  MRNUMBER = {2190207},
}

@article{ding2023autoregressive,
      author={Xiucai Ding and Zhou Zhou},
      title={Auto-Regressive Approximations to Non-stationary Time Series, with Inference and Applications}, 
      journal = {Arxiv id: 2112.00693},
      year={2023},
}

@article {MR4091095,
    AUTHOR = {Mayer, Ulrike and Z\"{a}hle, Henryk and Zhou, Zhou},
     TITLE = {Functional weak limit theorem for a local empirical process of
              non-stationary time series and its application},
   JOURNAL = {Bernoulli},
  FJOURNAL = {Bernoulli. Official Journal of the Bernoulli Society for
              Mathematical Statistics and Probability},
    VOLUME = {26},
      YEAR = {2020},
    NUMBER = {3},
     PAGES = {1891--1911},
      ISSN = {1350-7265},
   MRCLASS = {60F17 (62G20 62G30 62M10)},
  MRNUMBER = {4091095},
       DOI = {10.3150/19-BEJ1174},
       URL = {https://doi.org/10.3150/19-BEJ1174},
}

@article {MR2827528,
    AUTHOR = {Wu, Wei Biao and Zhou, Zhou},
     TITLE = {Gaussian approximations for non-stationary multiple time
              series},
   JOURNAL = {Statist. Sinica},
  FJOURNAL = {Statistica Sinica},
    VOLUME = {21},
      YEAR = {2011},
    NUMBER = {3},
     PAGES = {1397--1413},
      ISSN = {1017-0405},
   MRCLASS = {60F17 (60F05 60G10 60G15 62M10)},
  MRNUMBER = {2827528},
MRREVIEWER = {Mark Podolskij},
       DOI = {10.5705/ss.2008.223},
       URL = {https://doi.org/10.5705/ss.2008.223},
}

@article {MR4134802,
    AUTHOR = {Ding, Xiucai and Zhou, Zhou},
     TITLE = {Estimation and inference for precision matrices of
              nonstationary time series},
   JOURNAL = {Ann. Statist.},
  FJOURNAL = {The Annals of Statistics},
    VOLUME = {48},
      YEAR = {2020},
    NUMBER = {4},
     PAGES = {2455--2477},
      ISSN = {0090-5364},
   MRCLASS = {62M10 (62G05 62G10 62H12)},
  MRNUMBER = {4134802},
       DOI = {10.1214/19-AOS1894},
       URL = {https://doi.org/10.1214/19-AOS1894},
}

@article {MR3779697,
    AUTHOR = {Zhang, Xianyang and Cheng, Guang},
     TITLE = {Gaussian approximation for high dimensional vector under
              physical dependence},
   JOURNAL = {Bernoulli},
  FJOURNAL = {Bernoulli. Official Journal of the Bernoulli Society for
              Mathematical Statistics and Probability},
    VOLUME = {24},
      YEAR = {2018},
    NUMBER = {4A},
     PAGES = {2640--2675},
      ISSN = {1350-7265},
   MRCLASS = {62E17 (60F05 62M10)},
  MRNUMBER = {3779697},
       DOI = {10.3150/17-BEJ939},
       URL = {https://doi.org/10.3150/17-BEJ939},
}

@article{chang2023central,
      title={Central limit theorems for high dimensional dependent data}, 
      author={Jinyuan Chang and Xiaohui Chen and Mingcong Wu},
      year={2023},
      journal={Arxiv id: 2104.12929},
}

@article{mies2022sequential,
      title={Sequential Gaussian approximation for nonstationary time series in high dimensions}, 
      author={Fabian Mies and Ansgar Steland},
      year={2022},
      journal={Arxiv id: 2203.03237},
     
}

@article{doi:10.1080/01621459.2022.2161386,
author = {Zudi Lu and Xiaohang Ren and Rongmao Zhang},
title = {On Semiparametrically Dynamic Functional-Coefficient Autoregressive Spatio-Temporal Models with Irregular Location Wide Nonstationarity},
journal = {J. Amer. Statist. Assoc.},
volume = {0},
number = {0},
pages = {1-12},
year  = {2023},
publisher = {Taylor & Francis},
doi = {10.1080/01621459.2022.2161386},

URL = { 
    
        https://doi.org/10.1080/01621459.2022.2161386
    
    

},
eprint = { 
    
        https://doi.org/10.1080/01621459.2022.2161386
    
    

}

}

@article {MR3151764,
    AUTHOR = {Liu, Hanzhong and Yu, Bin},
     TITLE = {Asymptotic properties of {L}asso+m{LS} and {L}asso+{R}idge in
              sparse high-dimensional linear regression},
   JOURNAL = {Electron. J. Stat.},
  FJOURNAL = {Electronic Journal of Statistics},
    VOLUME = {7},
      YEAR = {2013},
     PAGES = {3124--3169},
   MRCLASS = {62J07 (62F12 62F40)},
  MRNUMBER = {3151764},
MRREVIEWER = {Winston T. Lin},
       DOI = {10.1214/14-EJS875},
       URL = {https://doi.org/10.1214/14-EJS875},
}

@article {MR1391173,
    AUTHOR = {Paparoditis, Efstathios},
     TITLE = {Bootstrapping autoregressive and moving average parameter
              estimates of infinite order vector autoregressive processes},
   JOURNAL = {J. Multivariate Anal.},
  FJOURNAL = {Journal of Multivariate Analysis},
    VOLUME = {57},
      YEAR = {1996},
    NUMBER = {2},
     PAGES = {277--296},
      ISSN = {0047-259X},
   MRCLASS = {62M10 (62E20 62G09 62P20)},
  MRNUMBER = {1391173},
MRREVIEWER = {Arup Bose},
       DOI = {10.1006/jmva.1996.0034},
       URL = {https://doi.org/10.1006/jmva.1996.0034},
}

@article {MR2493017,
    AUTHOR = {Rosenblatt, M.},
     TITLE = {A comment on a conjecture of {N}. {W}iener},
   JOURNAL = {Statist. Probab. Lett.},
  FJOURNAL = {Statistics \& Probability Letters},
    VOLUME = {79},
      YEAR = {2009},
    NUMBER = {3},
     PAGES = {347--348},
      ISSN = {0167-7152},
   MRCLASS = {60G10 (28D05)},
  MRNUMBER = {2493017},
MRREVIEWER = {Vincent De Valk},
       DOI = {10.1016/j.spl.2008.09.001},
       URL = {https://doi.org/10.1016/j.spl.2008.09.001},
}

@article {MR2172215,
    AUTHOR = {Wu, Wei Biao},
     TITLE = {Nonlinear system theory: another look at dependence},
   JOURNAL = {Proc. Natl. Acad. Sci. USA},
  FJOURNAL = {Proceedings of the National Academy of Sciences of the United
              States of America},
    VOLUME = {102},
      YEAR = {2005},
    NUMBER = {40},
     PAGES = {14150--14154},
      ISSN = {0027-8424},
   MRCLASS = {62M10},
  MRNUMBER = {2172215},
       DOI = {10.1073/pnas.0506715102},
       URL = {https://doi.org/10.1073/pnas.0506715102},
}

@article {MR3310530,
    AUTHOR = {Zhou, Zhou},
     TITLE = {Inference for non-stationary time series regression with or
              without inequality constraints},
   JOURNAL = {J. R. Stat. Soc. Ser. B. Stat. Methodol.},
  FJOURNAL = {Journal of the Royal Statistical Society. Series B.
              Statistical Methodology},
    VOLUME = {77},
      YEAR = {2015},
    NUMBER = {2},
     PAGES = {349--371},
      ISSN = {1369-7412},
   MRCLASS = {62J05},
  MRNUMBER = {3310530},
       DOI = {10.1111/rssb.12077},
       URL = {https://doi.org/10.1111/rssb.12077},
}

@article {MR2915091,
    AUTHOR = {Paparoditis, Efstathios and Politis, Dimitris N.},
     TITLE = {Nonlinear spectral density estimation: thresholding the
              correlogram},
   JOURNAL = {J. Time Series Anal.},
  FJOURNAL = {Journal of Time Series Analysis},
    VOLUME = {33},
      YEAR = {2012},
    NUMBER = {3},
     PAGES = {386--397},
      ISSN = {0143-9782},
   MRCLASS = {62G07 (62G20 62M10 62M15)},
  MRNUMBER = {2915091},
MRREVIEWER = {Alireza Nematollahi},
       DOI = {10.1111/j.1467-9892.2011.00771.x},
       URL = {https://doi.org/10.1111/j.1467-9892.2011.00771.x},
}

@article {MR2351105,
    AUTHOR = {Shao, Xiaofeng and Wu, Wei Biao},
     TITLE = {Asymptotic spectral theory for nonlinear time series},
   JOURNAL = {Ann. Statist.},
  FJOURNAL = {The Annals of Statistics},
    VOLUME = {35},
      YEAR = {2007},
    NUMBER = {4},
     PAGES = {1773--1801},
      ISSN = {0090-5364},
   MRCLASS = {62M15 (62E20 62M10)},
  MRNUMBER = {2351105},
MRREVIEWER = {Weidong Liu},
       DOI = {10.1214/009053606000001479},
       URL = {https://doi.org/10.1214/009053606000001479},
}

@article {MR3102549,
    AUTHOR = {B\"{u}hlmann, Peter},
     TITLE = {Statistical significance in high-dimensional linear models},
   JOURNAL = {Bernoulli},
  FJOURNAL = {Bernoulli. Official Journal of the Bernoulli Society for
              Mathematical Statistics and Probability},
    VOLUME = {19},
      YEAR = {2013},
    NUMBER = {4},
     PAGES = {1212--1242},
      ISSN = {1350-7265,1573-9759},
   MRCLASS = {62F03 (62J05 62J07)},
  MRNUMBER = {3102549},
MRREVIEWER = {B.\ M. Golam Kibria},
       DOI = {10.3150/12-BEJSP11},
       URL = {https://doi.org/10.3150/12-BEJSP11},
}

@article {MR1946581,
    AUTHOR = {Fan, Jianqing and Li, Runze},
     TITLE = {Variable selection via nonconcave penalized likelihood and its
              oracle properties},
   JOURNAL = {J. Amer. Statist. Assoc.},
  FJOURNAL = {Journal of the American Statistical Association},
    VOLUME = {96},
      YEAR = {2001},
    NUMBER = {456},
     PAGES = {1348--1360},
      ISSN = {0162-1459,1537-274X},
   MRCLASS = {62H05 (62H12)},
  MRNUMBER = {1946581},
       DOI = {10.1198/016214501753382273},
       URL = {https://doi.org/10.1198/016214501753382273},
}

@article {MR4102694,
    AUTHOR = {Lopes, Miles E. and Lin, Zhenhua and M\"{u}ller, Hans-Georg},
     TITLE = {Bootstrapping max statistics in high dimensions:
              near-parametric rates under weak variance decay and
              application to functional and multinomial data},
   JOURNAL = {Ann. Statist.},
  FJOURNAL = {The Annals of Statistics},
    VOLUME = {48},
      YEAR = {2020},
    NUMBER = {2},
     PAGES = {1214--1229},
      ISSN = {0090-5364,2168-8966},
   MRCLASS = {62G09 (62G05 62G15 62G20 62R10)},
  MRNUMBER = {4102694},
MRREVIEWER = {Shibin\ Zhang},
       DOI = {10.1214/19-AOS1844},
       URL = {https://doi.org/10.1214/19-AOS1844},
}

@book {MR1707286,
    AUTHOR = {Politis, Dimitris N. and Romano, Joseph P. and Wolf, Michael},
     TITLE = {Subsampling},
    SERIES = {Springer Series in Statistics},
 PUBLISHER = {Springer-Verlag, New York},
      YEAR = {1999},
     PAGES = {xvi+347},
      ISBN = {0-387-98854-8},
   MRCLASS = {62G09 (62F12 62G07 62G10 62M10 62M30)},
  MRNUMBER = {1707286},
MRREVIEWER = {Arup\ Bose},
       DOI = {10.1007/978-1-4612-1554-7},
       URL = {https://doi.org/10.1007/978-1-4612-1554-7},
}

@article {MR4209452,
    AUTHOR = {Nicholson, William B. and Wilms, Ines and Bien, Jacob and
              Matteson, David S.},
     TITLE = {High dimensional forecasting via interpretable vector
              autoregression},
   JOURNAL = {J. Mach. Learn. Res.},
  FJOURNAL = {Journal of Machine Learning Research (JMLR)},
    VOLUME = {21},
      YEAR = {2020},
     PAGES = {Paper No. 166, 52},
      ISSN = {1532-4435,1533-7928},
   MRCLASS = {62M10 (62J07 62M20)},
  MRNUMBER = {4209452},
MRREVIEWER = {Anoop\ Chaturvedi},
}

@article {MR3343007,
    AUTHOR = {Meyer, Marco and Kreiss, Jens-Peter},
     TITLE = {On the vector autoregressive sieve bootstrap},
   JOURNAL = {J. Time Series Anal.},
  FJOURNAL = {Journal of Time Series Analysis},
    VOLUME = {36},
      YEAR = {2015},
    NUMBER = {3},
     PAGES = {377--397},
      ISSN = {0143-9782,1467-9892},
   MRCLASS = {62M10 (62F40)},
  MRNUMBER = {3343007},
       DOI = {10.1111/jtsa.12090},
       URL = {https://doi.org/10.1111/jtsa.12090},
}

@book {MR2172368,
	AUTHOR = {L\"{u}tkepohl, Helmut},
	TITLE = {New introduction to multiple time series analysis},
	PUBLISHER = {Springer-Verlag, Berlin},
	YEAR = {2005},
	PAGES = {xxii+764},
	ISBN = {3-540-40172-5},
	MRCLASS = {62-01 (62M10)},
	MRNUMBER = {2172368},
	MRREVIEWER = {Ra\'{u}l\ Pedro\ Mentz},
	DOI = {10.1007/978-3-540-27752-1},
	URL = {https://doi.org/10.1007/978-3-540-27752-1},
}

@inproceedings{seabold2010statsmodels,
  title={statsmodels: Econometric and statistical modeling with python},
  author={Seabold, Skipper and Perktold, Josef},
  booktitle={9th Python in Science Conference},
  year={2010},
}

@article {MR2041534,
    AUTHOR = {Politis, Dimitris N. and White, Halbert},
     TITLE = {Automatic block-length selection for the dependent bootstrap},
   JOURNAL = {Econometric Rev.},
  FJOURNAL = {Econometric Reviews},
    VOLUME = {23},
      YEAR = {2004},
    NUMBER = {1},
     PAGES = {53--70},
      ISSN = {0747-4938,1532-4168},
   MRCLASS = {62M10},
  MRNUMBER = {2041534},
MRREVIEWER = {Francesco\ Battaglia},
       DOI = {10.1081/ETC-120028836},
       URL = {https://doi.org/10.1081/ETC-120028836},
}

@article {MR2656050,
    AUTHOR = {Shao, Xiaofeng},
     TITLE = {The dependent wild bootstrap},
      NOTE = {With supplementary material available online},
   JOURNAL = {J. Amer. Statist. Assoc.},
  FJOURNAL = {Journal of the American Statistical Association},
    VOLUME = {105},
      YEAR = {2010},
    NUMBER = {489},
     PAGES = {218--235},
      ISSN = {0162-1459,1537-274X},
   MRCLASS = {62F40 (62M10)},
  MRNUMBER = {2656050},
MRREVIEWER = {Cees\ G. H. Diks},
       DOI = {10.1198/jasa.2009.tm08744},
       URL = {https://doi.org/10.1198/jasa.2009.tm08744},
}

@article {MR2380557,
    AUTHOR = {Lahiri, S. N. and Furukawa, K. and Lee, Y.-D.},
     TITLE = {A nonparametric plug-in rule for selecting optimal block
              lengths for block bootstrap methods},
   JOURNAL = {Stat. Methodol.},
  FJOURNAL = {Statistical Methodology},
    VOLUME = {4},
      YEAR = {2007},
    NUMBER = {3},
     PAGES = {292--321},
      ISSN = {1572-3127},
   MRCLASS = {99-01},
  MRNUMBER = {2380557},
       DOI = {10.1016/j.stamet.2006.08.002},
       URL = {https://doi.org/10.1016/j.stamet.2006.08.002},
}

@article {MR2374985,
    AUTHOR = {Sun, Yixiao and Phillips, Peter C. B. and Jin, Sainan},
     TITLE = {Optimal bandwidth selection in
              heteroskedasticity-autocorrelation robust testing},
   JOURNAL = {Econometrica},
  FJOURNAL = {Econometrica. Journal of the Econometric Society},
    VOLUME = {76},
      YEAR = {2008},
    NUMBER = {1},
     PAGES = {175--194},
      ISSN = {0012-9682,1468-0262},
   MRCLASS = {62M07 (62G10 62M10)},
  MRNUMBER = {2374985},
       DOI = {10.1111/j.0012-9682.2008.00822.x},
       URL = {https://doi.org/10.1111/j.0012-9682.2008.00822.x},
}

@article {MR2847974,
    AUTHOR = {Chatterjee, A. and Lahiri, S. N.},
     TITLE = {Bootstrapping lasso estimators},
   JOURNAL = {J. Amer. Statist. Assoc.},
  FJOURNAL = {Journal of the American Statistical Association},
    VOLUME = {106},
      YEAR = {2011},
    NUMBER = {494},
     PAGES = {608--625},
      ISSN = {0162-1459,1537-274X},
   MRCLASS = {62J07 (60B10 62F12 62F25 62F35 62F40 62J10 62P10)},
  MRNUMBER = {2847974},
       DOI = {10.1198/jasa.2011.tm10159},
       URL = {https://doi.org/10.1198/jasa.2011.tm10159},
}

@article{10.1162/neco.1997.9.8.1735,
    author = {Hochreiter, Sepp and Schmidhuber, Jürgen},
    title = "{Long Short-Term Memory}",
    journal = {Neural Computation},
    volume = {9},
    number = {8},
    pages = {1735-1780},
    year = {1997},
    month = {11},
    issn = {0899-7667},
    doi = {10.1162/neco.1997.9.8.1735},
    url = {https://doi.org/10.1162/neco.1997.9.8.1735},
    eprint = {https://direct.mit.edu/neco/article-pdf/9/8/1735/813796/neco.1997.9.8.1735.pdf},
}

@article{liu2021highdimensional,
      title={High-dimensional Simultaneous Inference on Non-Gaussian VAR Model via De-biased Estimator}, 
      author={Linbo Liu and Danna Zhang},
      year={2021},
      journal = {arXiv id: 2111.01382}
}

@article{zhangPolitis,
    author = {Zhang, Yunyi and Politis, Dimitris N},
    title = "{Debiased and thresholded ridge regression for linear models with heteroskedastic and correlated errors}",
    journal = {J. R. Stat. Soc. Ser. B. Stat. Methodol.},
    volume = {85},
    number = {2},
    pages = {327-355},
    year = {2023},
    issn = {1369-7412},
    doi = {10.1093/jrsssb/qkad006},
    url = {https://doi.org/10.1093/jrsssb/qkad006},
    eprint = {https://academic.oup.com/jrsssb/article-pdf/85/2/327/50204568/qkad006.pdf},
}

@article {MR0515681,
    AUTHOR = {Efron, B.},
     TITLE = {Bootstrap methods: another look at the jackknife},
   JOURNAL = {Ann. Statist.},
  FJOURNAL = {The Annals of Statistics},
    VOLUME = {7},
      YEAR = {1979},
    NUMBER = {1},
     PAGES = {1--26},
      ISSN = {0090-5364,2168-8966},
   MRCLASS = {62E15 (62G05 62H30 62J05)},
  MRNUMBER = {515681},
MRREVIEWER = {B.\ Ya.\ Levit},
       URL =
              {http://links.jstor.org/sici?sici=0090-5364(197901)7:1<1:BMALAT>2.0.CO;2-6&origin=MSN},
}

@article {MR1310224,
    AUTHOR = {Politis, Dimitris N. and Romano, Joseph P.},
     TITLE = {The stationary bootstrap},
   JOURNAL = {J. Amer. Statist. Assoc.},
  FJOURNAL = {Journal of the American Statistical Association},
    VOLUME = {89},
      YEAR = {1994},
    NUMBER = {428},
     PAGES = {1303--1313},
      ISSN = {0162-1459,1537-274X},
   MRCLASS = {62G09 (62M10)},
  MRNUMBER = {1310224},
MRREVIEWER = {Michael\ Falk},
       URL =
              {http://links.jstor.org/sici?sici=0162-1459(199412)89:428<1303:TSB>2.0.CO;2-3&origin=MSN},
}

@incollection {MR1197789,
    AUTHOR = {Politis, Dimitris N. and Romano, Joseph P.},
     TITLE = {A circular block-resampling procedure for stationary data},
 BOOKTITLE = {Exploring the limits of bootstrap ({E}ast {L}ansing, {MI},
              1990)},
    SERIES = {Wiley Ser. Probab. Math. Statist. Probab. Math. Statist.},
     PAGES = {263--270},
 PUBLISHER = {Wiley, New York},
      YEAR = {1992},
      ISBN = {0-471-53631-8},
   MRCLASS = {62G09 (62M10)},
  MRNUMBER = {1197789},
}

@article {MR1466304,
    AUTHOR = {B\"{u}hlmann, Peter},
     TITLE = {Sieve bootstrap for time series},
   JOURNAL = {Bernoulli},
  FJOURNAL = {Bernoulli. Official Journal of the Bernoulli Society for
              Mathematical Statistics and Probability},
    VOLUME = {3},
      YEAR = {1997},
    NUMBER = {2},
     PAGES = {123--148},
      ISSN = {1350-7265},
   MRCLASS = {62M10 (62G09)},
  MRNUMBER = {1466304},
MRREVIEWER = {Somnath\ Datta},
       DOI = {10.2307/3318584},
       URL = {https://doi.org/10.2307/3318584},
}

@article {MR4134800,
    AUTHOR = {Meyer, Marco and Paparoditis, Efstathios and Kreiss,
              Jens-Peter},
     TITLE = {Extending the validity of frequency domain bootstrap methods
              to general stationary processes},
   JOURNAL = {Ann. Statist.},
  FJOURNAL = {The Annals of Statistics},
    VOLUME = {48},
      YEAR = {2020},
    NUMBER = {4},
     PAGES = {2404--2427},
      ISSN = {0090-5364,2168-8966},
   MRCLASS = {62M10 (62G09 62M15)},
  MRNUMBER = {4134800},
       DOI = {10.1214/19-AOS1892},
       URL = {https://doi.org/10.1214/19-AOS1892},
}

@article {MR3779709,
    AUTHOR = {Cao, Hongyuan and Liu, Weidong and Zhou, Zhou},
     TITLE = {Simultaneous nonparametric regression analysis of sparse
              longitudinal data},
   JOURNAL = {Bernoulli},
  FJOURNAL = {Bernoulli. Official Journal of the Bernoulli Society for
              Mathematical Statistics and Probability},
    VOLUME = {24},
      YEAR = {2018},
    NUMBER = {4A},
     PAGES = {3013--3038},
      ISSN = {1350-7265,1573-9759},
   MRCLASS = {62G08 (62G15)},
  MRNUMBER = {3779709},
MRREVIEWER = {Daoji\ Li},
       DOI = {10.3150/17-BEJ952},
       URL = {https://doi.org/10.3150/17-BEJ952},
}

@article {MR1294896,
    AUTHOR = {Burman, Prabir and Chow, Edmond and Nolan, Deborah},
     TITLE = {A cross-validatory method for dependent data},
   JOURNAL = {Biometrika},
  FJOURNAL = {Biometrika},
    VOLUME = {81},
      YEAR = {1994},
    NUMBER = {2},
     PAGES = {351--358},
      ISSN = {0006-3444,1464-3510},
   MRCLASS = {62G05},
  MRNUMBER = {1294896},
       DOI = {10.1093/biomet/81.2.351},
       URL = {https://doi.org/10.1093/biomet/81.2.351},
}

@book {DPbook,
    AUTHOR = {Politis, Dimitris N. and McElroy, Tucker S.},
     TITLE = {Time Series: A First Course with Bootstrap Starter},
 PUBLISHER = {Chapman and Hall/CRC Press, Boca Raton},
      YEAR = {2020},
     PAGES = {300},
      ISBN = {9780429109553},
}

@article {MR4500619,
    AUTHOR = {Chernozhuokov, Victor and Chetverikov, Denis and Kato, Kengo
              and Koike, Yuta},
     TITLE = {Improved central limit theorem and bootstrap approximations in
              high dimensions},
   JOURNAL = {Ann. Statist.},
  FJOURNAL = {The Annals of Statistics},
    VOLUME = {50},
      YEAR = {2022},
    NUMBER = {5},
     PAGES = {2562--2586},
      ISSN = {0090-5364,2168-8966},
   MRCLASS = {60F05 (62E17)},
  MRNUMBER = {4500619},
MRREVIEWER = {Fraser\ Alexander\ Daly},
       DOI = {10.1214/22-aos2193},
       URL = {https://doi.org/10.1214/22-aos2193},
}
\clearpage
\appendix
\numberwithin{equation}{section}
\numberwithin{lemma}{section}
\numberwithin{figure}{section}
\section{Proofs of theorems in section \ref{section.m_alpha_dependent}
}
\begin{proof}[proof of lemma \ref{lemma.operation}]
1. Notice that $\boldsymbol \eta^{(t)}_k = \boldsymbol\epsilon^{(t - j)}_{l}$, here $j = \lfloor\frac{k - 1}{d}\rfloor$ and $l = k - d\times j$, $\lfloor x\rfloor$ represents the largest integer that is smaller than or equal to $x$. Therefore,  $\boldsymbol \eta^{(t)}_k$ is a function of $\cdots, e_{t - 1}, e_t$ and $\mathbf{E}\boldsymbol \eta^{(t)}_k = 0$. Besides,
\begin{align*}
\sup_{t\in\mathbf{Z}, k = 1,\cdots, pd}\Vert\boldsymbol\eta^{(t)}_k \Vert_m\leq \sup_{t\in\mathbf{Z}, i = 1,\cdots, d}\Vert\boldsymbol\epsilon^{(t)}_{i}\Vert_m
 = O(1).
\end{align*}
Similar to eq.\eqref{eq.replace_epsilon}, define
$\boldsymbol \eta^{(t)}_k = h_k^{(t, T)}(\cdots, e_{t - 1}, e_t)$, $\boldsymbol\eta^{(t, s)}_k = h_k^{(t,T)}(\cdots, e_{t - s - 1}, e^\dagger_{t - s}, e_{t - s + 1},\cdots, e_t)$, and $\lambda_k^{(t, s)} = \Vert\boldsymbol \eta^{(t)}_k - \boldsymbol\eta^{(t, s)}_k\Vert_m$, $\lambda^{(s)} = \sup_{k = 1,\cdots, pd, t\in\mathbf{Z}}\lambda_k^{(t, s)}$. Then
\begin{align*}
\lambda_k^{(t, s)} = \Vert\boldsymbol\epsilon^{(t - j)}_{l} - \boldsymbol\epsilon^{(t - j, s - j)}_{l}\Vert_m\leq \delta^{(s - j)}\leq \sum_{q = 0}^{p - 1}\delta^{(s - q)}\\
\Rightarrow \sum_{s = k}^\infty \lambda^{(s)}\leq \sum_{j = 0}^{p - 1}\sum_{s = k}^\infty\delta^{(s - j)}\leq C\sum_{j = 0}^{p - 1}\frac{1}{(1\vee (1 + k - j))^\alpha}
\leq \frac{C_0}{(1 + k)^\alpha},
\end{align*}
and we prove the first result.

2. Notice that $\mathbf{E}\boldsymbol \gamma^{(t)}_{i} = \mathbf{E}\boldsymbol \zeta^{(t)}_{i} = 0$. Besides,
\begin{align*}
\Vert\boldsymbol\epsilon^{(t)}_i\boldsymbol\epsilon^{(t)}_j - \mathbf{E}\boldsymbol\epsilon^{(t)}_i\boldsymbol\epsilon^{(t)}_j\Vert_{m / 2}
\leq 2\Vert\boldsymbol\epsilon^{(t)}_i\Vert_m\times \Vert\boldsymbol\epsilon^{(t)}_j\Vert_{m}\leq C\\
\text{and } \Vert\boldsymbol\epsilon^{(t - 1)}_i \boldsymbol\epsilon^{(t)}_j - \mathbf{E}\boldsymbol\epsilon^{(t - 1)}_i \boldsymbol\epsilon^{(t)}_j\Vert_{m / 2}
\leq 2\Vert\boldsymbol\epsilon^{(t - 1)}_i\Vert_m\times \Vert\boldsymbol\epsilon^{(t)}_j\Vert_m\leq C.
\end{align*}
For $\boldsymbol \gamma^{(t)}_{i}$ and $\boldsymbol \zeta^{(t)}_{i}$ are functions of $\cdots, e_{t - 1}, e_t$, define
\begin{align*}
\boldsymbol\gamma^{(t)}_{i} = H_i^{(t, T)}(\cdots, e_{t - 1}, e_t)\ \text{and }  \boldsymbol\zeta^{(t)}_{i} = G_i^{(t, T)}(\cdots, e_{t - 1}, e_t),\\
\boldsymbol\gamma^{(t, s)}_{i} = H_i^{(t, T)}(\cdots, e_{t - s - 1}, e^\dagger_{t - s}, e_{t - s + 1}, \cdots, e_{t - 1}, e_t),\\
\boldsymbol\zeta^{(t, s)}_{i} = G_i^{(t, T)}(\cdots, e_{t - s - 1}, e^\dagger_{t - s}, e_{t - s + 1},\cdots, e_{t - 1}, e_t),\\
\psi_i^{(t, s)} = \Vert \boldsymbol\gamma^{(t)}_{i} - \boldsymbol\gamma^{(t, s)}_{i}\Vert_{m/2}\ \text{and }
\psi^{(s)} = \sup_{t\in\mathbf{Z}, i = 1,\cdots, d^2}\psi_i^{(t, s)},\\
\phi_i^{(t, s)} = \Vert \boldsymbol\zeta^{(t)}_{i} - \boldsymbol\zeta^{(t, s)}_{i}\Vert_{m/2}\ \text{and  }
\phi^{(s)} = \sup_{t\in\mathbf{Z}, i = 1,\cdots, d^2}\phi_i^{(t, s)}.
\end{align*}
Then we have
\begin{align*}
\boldsymbol \gamma^{(t,s)}_{(i - 1)\times d + j} = \boldsymbol\epsilon^{(t,s)}_i\boldsymbol\epsilon^{(t,s)}_j - \mathbf{E}\boldsymbol\epsilon^{(t)}_i\boldsymbol\epsilon^{(t)}_j\ \text{and } \boldsymbol \zeta^{(t,s)}_{(i - 1)\times d + j} = \boldsymbol\epsilon^{(t - 1, s  - 1)}_i \boldsymbol\epsilon^{(t,s)}_j - \mathbf{E}\boldsymbol\epsilon^{(t - 1)}_i \boldsymbol\epsilon^{(t)}_j\\
\Rightarrow \psi_{(i - 1)\times d + j}^{(t, s)}\leq \Vert\boldsymbol\epsilon^{(t)}_i\boldsymbol\epsilon^{(t)}_j - \boldsymbol\epsilon^{(t, s)}_i\boldsymbol\epsilon^{(t)}_j\Vert_{m/2} + \Vert\boldsymbol\epsilon^{(t,s)}_i\boldsymbol\epsilon^{(t)}_j - \boldsymbol\epsilon^{(t,s)}_i\boldsymbol\epsilon^{(t,s)}_j\Vert_{m/2}\\
\leq \Vert\boldsymbol\epsilon^{(t,s)}_i - \boldsymbol\epsilon^{(t)}_i\Vert_m\times \Vert\boldsymbol\epsilon^{(t)}_j\Vert_m
+ \Vert\boldsymbol\epsilon^{(t,s)}_i\Vert_m\times \Vert\boldsymbol\epsilon^{(t,s)}_j - \boldsymbol\epsilon^{(t)}_j\Vert_m\leq C\delta^{(s)}\Rightarrow
\sum_{s = k}^\infty\psi^{(s)}\leq \frac{C_0}{(1 + k)^\alpha}\\
\text{and } \phi_{(i - 1)\times d + j}^{(t, s)}\leq \Vert\boldsymbol\epsilon^{(t - 1)}_i \boldsymbol\epsilon^{(t)}_j - \boldsymbol\epsilon^{(t - 1, s - 1)}_i \boldsymbol\epsilon^{(t)}_j\Vert_{m / 2} + \Vert\boldsymbol\epsilon^{(t - 1, s - 1)}_i \boldsymbol\epsilon^{(t)}_j - \boldsymbol\epsilon^{(t - 1, s - 1)}_i \boldsymbol\epsilon^{(t,s)}_j\Vert_{m / 2}\\
\leq \Vert\boldsymbol\epsilon^{(t - 1)}_i - \boldsymbol\epsilon^{(t - 1, s - 1)}_i \Vert_m\times \Vert\boldsymbol\epsilon^{(t)}_j\Vert_m
+ \Vert\boldsymbol\epsilon^{(t - 1, s - 1)}_i\Vert_m\times \Vert\boldsymbol\epsilon^{(t)}_j - \boldsymbol\epsilon^{(t,s)}_j\Vert_m\leq C(\delta^{(s - 1)} + \delta^{(s)})\\
\Rightarrow \sum_{s = k}^\infty\phi^{(s)}\leq \frac{C_1}{(1 + k)^\alpha},
\end{align*}
and we prove 2.

3. Notice that $\mathbf{E}\sum_{j = 1}^{d} \mathbf{L}_{ij}\boldsymbol\epsilon^{(t)}_j = 0$,
\begin{align*}
\Vert\sum_{j = 1}^{d} \mathbf{L}_{ij}\boldsymbol\epsilon^{(t)}_j\Vert_m\leq \sum_{j = 1}^d \vert\mathbf{L}_{ij}\vert\times\Vert\boldsymbol\epsilon^{(t)}_j\Vert_m\leq
\vert\mathbf{L}\vert_{L_1}\times \sup_{t\in\mathbf{Z}, j = 1,\cdots, d}\Vert\boldsymbol\epsilon^{(t)}_j\Vert_m = O(1),\\
\text{and } \Vert\sum_{j = 1}^{d} \mathbf{L}_{ij}\boldsymbol\epsilon^{(t)}_j - \sum_{j = 1}^{d} \mathbf{L}_{ij}\boldsymbol\epsilon^{(t,s)}_j\Vert_m
\leq \sum_{j = 1}^d \vert\mathbf{L}_{ij}\vert\times \Vert\boldsymbol\epsilon^{(t)}_j  -\boldsymbol\epsilon^{(t,s)}_j\Vert_m\leq
\vert\mathbf{L}\vert_{L_1}\times \delta^{(s)}.
\end{align*}
Therefore
$$
\sum_{s = k}^\infty\sup_{t\in\mathbf{Z}, i = 1,\cdots, c}\Vert\sum_{j = 1}^{d} \mathbf{L}_{ij}\boldsymbol\epsilon^{(t)}_j - \sum_{j = 1}^{d} \mathbf{L}_{ij}
\boldsymbol\epsilon^{(t,s)}_j\Vert_m\leq \frac{C_2}{(1 + k)^\alpha},
$$
and we prove 3.

4. From section 2.2 in \cite{MR1238940}, $\boldsymbol\beta^{(t)}  = \boldsymbol\epsilon^{(t)} + \sum_{k = 1}^\infty\mathbf{L}^k\boldsymbol\epsilon^{(t - k)}$, so $\mathbf{E}\boldsymbol\beta^{(t)} = 0$. For a random vector $\mathbf{y}\in\mathbf{R}^d$ such that $\Vert\mathbf{y}_i\Vert_m\leq C$,
$$
\Vert\sum_{j = 1}^d \mathbf{L}_{ij}\mathbf{y}_j\Vert_m\leq \sum_{\mathbf{L}_{ij}\neq 0}\vert\mathbf{L}_{ij}\vert\times \Vert\mathbf{y}_j\Vert_m\leq \max_{j = 1,\cdots, d}\Vert\mathbf{y}_j\Vert_m\times \max_{i = 1,\cdots,c}\sum_{\mathbf{L}_{ij}\neq 0}\vert\mathbf{L}_{ij}\vert
$$
For $\sum_{\mathbf{L}_{ij}\neq 0}\vert\mathbf{L}_{ij}\vert = \mathbf{e}_i\mathbf{L}\mathbf{f}_i$ with $\mathbf{e}_i = (\underbrace{0,0,\cdots, 0}_{i-  1}, 1, 0,\cdots,0)^T$, i.e., the coordinate-wise projection; and
$$
\mathbf{f}_i =
\begin{cases}
1\ \text{if } \mathbf{L}_{ij} > 0\\
0\ \text{if } \mathbf{L}_{ij} = 0\\
-1\ \text{if } \mathbf{L}_{ij} < 0
\end{cases}.
$$
This implies $\sum_{\mathbf{L}_{ij}\neq 0}\vert\mathbf{L}_{ij}\vert\leq \rho\vert\mathbf{e}_i\vert_2\times \vert\mathbf{f}_i\vert_2\leq \rho\sqrt{\vert\mathbf{L}_{i\cdot}\vert_0}$,
and correspondingly
\begin{equation}
\Vert\sum_{j = 1}^d \mathbf{L}_{ij}\mathbf{y}_j\Vert_m\leq C\rho\sqrt{\max_{i = 1,\cdots, d}\vert\mathbf{L}_{i\cdot}\vert_0}.
\label{eq.iterA}
\end{equation}
Apply induction, we have
\begin{align*}
\Vert(\mathbf{L}^k\boldsymbol\epsilon^{(t - k)})_i\Vert_m\leq \rho^k\times \max_{i = 1,\cdots, d}\vert\mathbf{L}_{i\cdot}\vert_0^{k/2}\times \max_{j = 1,\cdots, d}\Vert\boldsymbol\epsilon^{(t - k)}_j\Vert_m\\
\Rightarrow \Vert\boldsymbol\beta^{(t)}_i\Vert_m\leq \Vert\boldsymbol\epsilon^{(t)}_i\Vert_m + \sum_{k = 1}^\infty\Vert(\mathbf{L}^k\boldsymbol\epsilon^{(t - k)})_i\Vert_m\leq C
\end{align*}
for a constant $C$. Similarly,
\begin{align*}
\boldsymbol\beta^{(t,s)}_i  = \boldsymbol\epsilon^{(t,s)}_i + \sum_{k = 1}^s(\mathbf{L}^k\boldsymbol\epsilon^{(t - k, s - k)})_i,
\end{align*}
so from eq.\eqref{eq.iterA},
\begin{align*}
\Vert\boldsymbol\beta^{(t)}_i - \boldsymbol\beta^{(t,s)}_i\Vert_m\leq \Vert\boldsymbol\epsilon^{(t)}_i - \boldsymbol\epsilon^{(t,s)}_i\Vert_m + \sum_{l = 1}^s\Vert(\mathbf{L}^l(\boldsymbol\epsilon^{(t - l)} - \boldsymbol\epsilon^{(t - l, s - l)}))_i\Vert_m\\
\leq \delta^{(s)} + \sum_{l = 1}^s\rho^l\times \max_{i = 1,\cdots, d}\vert\mathbf{L}_{i\cdot}\vert_0^{l/2}\times\delta^{(s - l)}\\
\Rightarrow \sum_{s = k}^\infty \sup_{t\in\mathbf{Z}, i = 1,\cdots, d}\Vert\boldsymbol\beta^{(t)}_i - \boldsymbol\beta^{(t,s)}_i\Vert_m
\leq \sum_{s = k}^\infty\delta^{(s)} + \sum_{l = 1}^\infty \rho^l\max_{i = 1,\cdots, d}\vert\mathbf{L}_{i\cdot}\vert_0^{l/2}\sum_{s = l\vee k}^\infty\delta^{(s - l)}\\
\leq \frac{C}{(1 + k)^\alpha} + \sum_{l = 1}^\infty\rho^l\vert\mathbf{L}_{i\cdot}\vert_0^{l/2}\times \frac{C}{(1 + (k - l)\vee 0)^\alpha}.
\end{align*}
Since
\begin{align*}
\sum_{l = 1}^\infty\rho^l\vert\mathbf{L}_{i\cdot}\vert_0^{l/2}\times \frac{1}{(1 + (k - l)\vee 0)^\alpha}\leq \frac{1}{(1 + \frac{k}{2})^\alpha}\sum_{l = 1}^{\lfloor\frac{k}{2}\rfloor}\rho^l\times\vert\mathbf{L}_{i\cdot}\vert_0^{l/2} + \sum_{l = \lfloor\frac{k}{2}\rfloor + 1}^\infty\rho^l\vert\mathbf{L}_{i\cdot}\vert_0^{l/2}
\leq \frac{C}{(1 + k)^\alpha},
\end{align*}
we prove 4.
\end{proof}

\begin{proof}[proof of theorem \ref{theorem.MA}]
1. From lemma 1 and (B.6), (B.7) in \cite{zhangPolitis} , we have
\begin{align*}
\Vert \boldsymbol \Pi_{i\cdot} \mathbf{b}\Vert_m = \Vert\sum_{t = 1}^T \boldsymbol \epsilon^{(t)}_i\mathbf{b}_t\Vert_m\leq C\vert\mathbf{b}\vert_2\times
\left(\sup_{t\in\mathbf{Z}, i = 1,\cdots, d}\Vert \boldsymbol \epsilon^{(t)}_i\Vert_m + \sum_{s = 0}^\infty\delta^{(s)}\right)\leq C_1\vert\mathbf{b}\vert_2\\
\text{and } \Vert\boldsymbol \Pi_{i\cdot}^{(s)} \mathbf{b}\Vert_m = \Vert\sum_{t = 1}^T(\boldsymbol \epsilon^{(t)}_i - \mathbf{E}\boldsymbol \epsilon^{(t)}_i|\mathcal{F}^{(t, s)})\mathbf{b}_t\Vert_m\leq C\vert\mathbf{b}\vert_2\times\sum_{k = s}^\infty\delta^{(k)}\leq C_1\vert\mathbf{b}\vert_2\times (1 + s)^{-\alpha}
\end{align*}
Therefore,
\begin{align*}
\Vert\ \vert\boldsymbol \Pi \mathbf{b}\vert_\infty\ \Vert_m
= \Vert\ \max_{i = 1,\cdots, d}\vert\Pi_{i\cdot} \mathbf{b}\vert\ \Vert_m\leq d^{1/m}\max_{i = 1,\cdots, d}\Vert\Pi_{i\cdot} \mathbf{b}\Vert_m\leq Cd^{1/m}
\times \vert\mathbf{b}\vert_2,\\
\Vert\ \vert\boldsymbol \Pi^{(s)} \mathbf{b}\vert_\infty\ \Vert_m = \Vert\ \max_{i = 1,\cdots, d}
\vert\Pi_{i\cdot}^{(s)} \mathbf{b}\vert\ \Vert_m\leq Cd^{1/m}\vert\mathbf{b}\vert_2\times (1 + s)^{-\alpha}.
\end{align*}

2. From lemma  B.1 and eq.(B.18) in \cite{zhangPolitis},
\begin{equation}
\begin{aligned}
\vert
Prob\left(\vert\sqrt{T}\overline{\boldsymbol\epsilon}\vert_\infty\leq x\right) - Prob\left(\vert\boldsymbol \xi\vert_\infty \leq x\right)
\vert\\
\leq \sup_{x\in\mathbf{R}}\vert\mathbf{E}h_{\psi,\psi,x}(\sqrt{T}\overline{\boldsymbol\epsilon}_1,\cdots, \sqrt{T}\overline{\boldsymbol\epsilon}_d) - \mathbf{E}h_{\psi,\psi,x}(\boldsymbol \xi_1,\cdots, \boldsymbol \xi_d)\vert\\
+ Ct(1 + \sqrt{\log(d)} + \sqrt{\vert\log(t)\vert}),
\end{aligned}
\label{eq.fir_part}
\end{equation}
here $t = \frac{1 + \log(2d)}{\psi}$ and $h_{\psi,\psi,x}(\cdot)$ coincides with \cite{zhangPolitis}. For
\begin{align*}
\Vert\sqrt{T}\overline{\boldsymbol\epsilon}_i - \frac{1}{\sqrt{T}}\sum_{t = 1}^T
\mathbf{E}\boldsymbol\epsilon^{(t)}_i|\mathcal{F}^{(t,s)}\Vert_m\leq C(1 + s)^{-\alpha},
\end{align*}
we have
\begin{equation}
\begin{aligned}
\vert\mathbf{E}h_{\psi,\psi,x}(\sqrt{T}\overline{\boldsymbol\epsilon}_1,\cdots, \sqrt{T}\overline{\boldsymbol\epsilon}_d) -
\mathbf{E}h_{\psi,\psi,x}\left(\frac{1}{\sqrt{T}}\sum_{t = 1}^T\mathbf{E}\boldsymbol\epsilon^{(t)}_1|\mathcal{F}^{(t,s)}, \cdots, \frac{1}{\sqrt{T}}\sum_{t = 1}^T\mathbf{E}\boldsymbol\epsilon^{(t)}_d|\mathcal{F}^{(t,s)}\right)\vert\\
\leq C\psi\Vert\
\max_{i = 1,\cdots, d}\vert\sqrt{T}\overline{\boldsymbol\epsilon}_i - \frac{1}{\sqrt{T}}\sum_{t = 1}^T\mathbf{E}\boldsymbol\epsilon^{(t)}_i|\mathcal{F}^{(t,s)}\vert\
\Vert_m\leq C_1\psi d^{1/m}\times (1 + s)^{-\alpha}.
\label{eq.second_part}
\end{aligned}
\end{equation}
For a given integer $q > s$, define the `big-block' and the `small-block'
$$
\boldsymbol \gamma^{(k)} = \frac{1}{\sqrt{T}}\sum_{t = (q + s)\times (k - 1) + 1}^{((q+s)\times (k - 1) + q)\wedge T}\mathbf{E}\boldsymbol\epsilon^{(t)}|\mathcal{F}^{(t,s)}\ \text{and }
\boldsymbol \eta^{(k)} = \frac{1}{\sqrt{T}}\sum_{t = (q + s)\times (k - 1) + q + 1}^{((q + s)\times k)\wedge T}
\mathbf{E}\boldsymbol\epsilon^{(t)}|\mathcal{F}^{(t,s)}.
$$
Define $v = \lceil\frac{T}{q + s}\rceil$, here $\lceil x\rceil$ is the smallest integer that is larger than or equal to $x$. Then
$$
\frac{1}{\sqrt{T}}\sum_{t = 1}^T\mathbf{E}\boldsymbol\epsilon^{(t)}_i|\mathcal{F}^{(t,s)} = \sum_{k = 1}^v \boldsymbol\gamma^{(k)}_i + \sum_{k = 1}^v \boldsymbol \eta^{(k)}_i.
$$
Moreover, $\boldsymbol\gamma^{(k)}$ are mutually independent; and $\boldsymbol \eta^{(k)}$ are mutually independent. Since $\mathbf{E}\boldsymbol\epsilon^{(t)}|\mathcal{F}^{(t,s)} = \mathbf{E}\boldsymbol\epsilon^{(t)}|\mathcal{F}^{(t,0)} + \sum_{l = 1}^s\left(\mathbf{E}\boldsymbol\epsilon^{(t)}|\mathcal{F}^{(t,l)} - \mathbf{E}\boldsymbol\epsilon^{(t)}|\mathcal{F}^{(t, l - 1)}\right)$, from (B.7) in \cite{zhangPolitis},
\begin{equation}
\begin{aligned}
\Vert
\frac{1}{\sqrt{T}}\sum_{t = 1}^T\mathbf{E}\boldsymbol\epsilon^{(t)}_i|\mathcal{F}^{(t,s)} - \sum_{k = 1}^v \boldsymbol\gamma^{(k)}_i
\Vert_m = \Vert\sum_{k = 1}^v \boldsymbol\eta^{(k)}_i\Vert_m\leq \frac{C\sqrt{v\times s}}{\sqrt{T}}\\
\Rightarrow \sup_{x\in\mathbf{R}}
\vert
\mathbf{E}h_{\psi,\psi,x}\left(\frac{1}{\sqrt{T}}\sum_{t = 1}^T\mathbf{E}\boldsymbol\epsilon^{(t)}_1|\mathcal{F}^{(t,s)}, \cdots, \frac{1}{\sqrt{T}}\sum_{t = 1}^T\mathbf{E}\boldsymbol\epsilon^{(t)}_d|\mathcal{F}^{(t,s)}\right)\\
- \mathbf{E}h_{\psi,\psi,x}\left(\sum_{k = 1}^v \boldsymbol\gamma^{(k)}_1,\cdots, \sum_{k = 1}^v \boldsymbol\gamma^{(k)}_d\right)
\vert\\
\leq C\psi\times \Vert\ \max_{k = 1,\cdots, d}\vert\sum_{k = 1}^v \boldsymbol \eta^{(k)}_i\vert\ \Vert_m\leq \frac{C_1\psi\times d^{1/m}\times \sqrt{vs}}{\sqrt{T}}.
\end{aligned}
\label{eq.third_part}
\end{equation}
Define $\boldsymbol\zeta^{(t)*}_i, t = 1,\cdots, T, i = 1,\cdots, d$ as joint normal random variables such that $\mathbf{E}\boldsymbol\zeta^{(t)*}_i = 0$,
$Cov(\boldsymbol\zeta^{(t_1)*}_{i_1}, \boldsymbol\zeta^{(t_2)*}_{i_2}) = Cov(\mathbf{E}\boldsymbol\epsilon^{(t_1)}_{i_1}|\mathcal{F}^{(t,s)}, \mathbf{E}\boldsymbol\epsilon^{(t_2)}_{i_2}|\mathcal{F}^{(t,s)})$, and $\boldsymbol\zeta^{(t)*}_i$ is independent of $e_t, t\in\mathbf{Z}$. Define
\begin{align*}
\boldsymbol \gamma^{(k)*}_i = \frac{1}{\sqrt{T}}\sum_{t = (q + s)\times (k - 1) + 1}^{((q+s)\times (k - 1) + q)\wedge T}\boldsymbol\zeta^{(t)*}_i,
\end{align*}
we have $\mathbf{E}\boldsymbol\gamma^{(k)*} = 0$,
$Cov(\boldsymbol\gamma^{(k)*}_i, \boldsymbol\gamma^{(k)*}_j) = Cov(\boldsymbol\gamma^{(k)}_i, \boldsymbol\gamma^{(k)}_j)$, and $\boldsymbol\gamma^{(k)*}\in\mathbf{R}^d$ are mutually independent. Define $\boldsymbol\lambda^{(k)} = \sum_{l = 1}^{k - 1}\boldsymbol\gamma^{(l)} + \sum_{l = k + 1}^v\boldsymbol\gamma^{(l)*}$, then
$\boldsymbol\lambda^{(k)} + \boldsymbol\gamma^{(k)} = \boldsymbol\lambda^{(k + 1)} + \boldsymbol\gamma^{(k + 1)*}$. Moreover,
\begin{equation}
\begin{aligned}
\vert\mathbf{E}h_{\psi,\psi,x}\left(\sum_{k = 1}^v \boldsymbol\gamma^{(k)}_1,\cdots, \sum_{k = 1}^v \boldsymbol\gamma^{(k)}_d\right)
- \mathbf{E}h_{\psi,\psi,x}\left(\sum_{k = 1}^v \boldsymbol\gamma^{(k)*}_1,\cdots, \sum_{k = 1}^v \boldsymbol\gamma^{(k)*}_d\right)\vert\\
= \vert
\mathbf{E}h_{\psi,\psi,x}\left(\boldsymbol\lambda^{(v)}_1 + \boldsymbol\gamma^{(v)}_1, \cdots, \boldsymbol\lambda^{(v)}_d + \boldsymbol\gamma^{(v)}_d\right)
- \mathbf{E}h_{\psi,\psi,x}\left(\boldsymbol\lambda^{(1)}_1 + \boldsymbol\gamma^{(1)*}_1,\cdots, \boldsymbol\lambda^{(1)}_d + \boldsymbol\gamma^{(1)*}_d\right)
\vert\\
\leq \sum_{l = 1}^v\vert
\mathbf{E}h_{\psi,\psi,x}\left(\boldsymbol\lambda^{(l)}_1 + \boldsymbol\gamma^{(l)}_1, \cdots, \boldsymbol\lambda^{(l)}_d + \boldsymbol\gamma^{(l)}_d\right)
- \mathbf{E}h_{\psi,\psi,x}\left(\boldsymbol\lambda^{(l)}_1 + \boldsymbol\gamma^{(l)*}_1,\cdots, \boldsymbol\lambda^{(l)}_d + \boldsymbol\gamma^{(l)*}_d\right)
\vert.
\end{aligned}
\end{equation}
Similar to eq.(B.26) in \cite{zhangPolitis}, we have
\begin{equation}
\begin{aligned}
\vert
\mathbf{E}h_{\psi,\psi,x}\left(\boldsymbol\lambda^{(l)}_1 + \boldsymbol\gamma^{(l)}_1, \cdots, \boldsymbol\lambda^{(l)}_d + \boldsymbol\gamma^{(l)}_d\right)
- \mathbf{E}h_{\psi,\psi,x}\left(\boldsymbol\lambda^{(l)}_1 + \boldsymbol\gamma^{(l)*}_1,\cdots, \boldsymbol\lambda^{(l)}_d + \boldsymbol\gamma^{(l)*}_d\right)
\vert\\
\leq C\psi^3\times d^{3/m}\times \max_{s = 1,\cdots, d}\Vert\boldsymbol\gamma^{(l)}_s\Vert_m^3\leq C_1\psi^3d^{3/m}\times\left(\frac{\sqrt{q}}{\sqrt{T}}\right)^3\\
\Rightarrow \vert\mathbf{E}h_{\psi,\psi,x}\left(\sum_{k = 1}^v \boldsymbol\gamma^{(k)}_1,\cdots, \sum_{k = 1}^v \boldsymbol\gamma^{(k)}_d\right)
- \mathbf{E}h_{\psi,\psi,x}\left(\sum_{k = 1}^v \boldsymbol\gamma^{(k)*}_1,\cdots, \sum_{k = 1}^v \boldsymbol\gamma^{(k)*}_d\right)\vert\\
\leq Cv\times\psi^3d^{3/m}\left(\frac{\sqrt{q}}{\sqrt{T}}\right)^3\leq C_1\psi^3\times d^{3/m}\times \left(\frac{q}{T}\right)^{1/2}.
\end{aligned}
\label{eq.fif}
\end{equation}
Similar to (B.29) and (B.30) in \cite{zhangPolitis},
\begin{equation}
\begin{aligned}
\vert
\mathbf{E}h_{\psi,\psi,x}\left(\sum_{k = 1}^v \boldsymbol\gamma^{(k)*}_1,\cdots, \sum_{k = 1}^v \boldsymbol\gamma^{(k)*}_d\right)
- \mathbf{E}h_{\psi,\psi,x}\left(\boldsymbol \xi_1,\cdots, \boldsymbol \xi_d\right)
\vert
\leq C\psi\times d^{1/m}\times \max_{i = 1,\cdots, d}\Vert\sum_{k = 1}^v \boldsymbol \eta^{(k)}_i\Vert_2\\
 + \frac{C\psi^2}{T}\max_{i_1, i_2 = 1,\cdots, d}\vert\sum_{t_1 = 1}^T\sum_{t_2 = 1}^T Cov(\mathbf{E}\boldsymbol\epsilon^{(t_1)}_{i_1}|\mathcal{F}^{(t_1,s)}, \mathbf{E}\boldsymbol\epsilon^{(t_2)}_{i_2}|\mathcal{F}^{(t_2,s)}) - Cov(\boldsymbol \epsilon_{i_1}^{(t_1)}, \boldsymbol \epsilon_{i_2}^{(t_2)})\vert\\
 \leq \frac{C_1\psi\times d^{1/m}\sqrt{s}}{\sqrt{q}} + \frac{C_2\psi^2}{T}\sum_{\vert t_1 - t_2\vert\leq s}\Vert \boldsymbol \epsilon_{i_1}^{(t_1)} - \mathbf{E}\boldsymbol\epsilon^{(t_1)}_{i_1}|\mathcal{F}^{(t_1,s)}\Vert_m\\
  + \frac{C_2\psi^2}{T}\sum_{\vert t_1 - t_2\vert\leq s}\Vert \boldsymbol \epsilon_{i_2}^{(t_2)} - \mathbf{E}\boldsymbol\epsilon^{(t_2)}_{i_2}|\mathcal{F}^{(t_2,s)}\Vert_m
 + \frac{C_2\psi^2}{T}\sum_{\vert t_1 - t_2\vert > s}\vert Cov(\boldsymbol \epsilon_{i_1}^{(t_1)}, \boldsymbol \epsilon_{i_2}^{(t_2)})\vert\\
 \leq \frac{C_1\psi\times d^{1/m}\sqrt{s}}{\sqrt{q}} + C_3\psi^2(1 + s)^{1-\alpha} + C_3\psi^2\sum_{l = s + 1}^\infty\frac{1}{(1 + l)^\alpha}.
\end{aligned}
\label{eq.fourth}
\end{equation}
According to the assumption, $\alpha_d = \frac{m\alpha}{1 + 12\alpha} - \delta$ with $\delta>0$. Choose $\psi = \log^3(T)$ and $s = \lfloor T^{\alpha_s}\rfloor$, $q = \lfloor T^{\alpha_q}\rfloor$; here \begin{align*}
\alpha_s = 1 - \frac{12\alpha_d}{m} - \frac{10\delta}{m} = \frac{1}{1 + 12\alpha} + \frac{2\delta}{m} > 0\ \text{and } \alpha_q  =
1 - \frac{6\alpha_d}{m} - \frac{4\delta}{m} = \frac{1 + 6\alpha}{1 + 12\alpha} + \frac{2\delta}{m} > 0.
\end{align*}
Since
\begin{align*}
\frac{\alpha_d}{m} - \alpha\alpha_s = -\frac{\delta}{m}(1 +2\alpha) < 0,\ \frac{\alpha_d}{m} + \frac{\alpha_s}{2} - \frac{\alpha_q}{2} =  -\frac{2\alpha}{1 + 12\alpha} - \frac{\delta}{m}<0\\
\frac{3\alpha_d}{m} + \frac{\alpha_q}{2} - \frac{1}{2} = -\frac{2\delta}{m} < 0.
\end{align*}
From eq.\eqref{eq.fir_part}, eq.\eqref{eq.second_part}, eq.\eqref{eq.third_part}, eq.\eqref{eq.fif}, and eq.\eqref{eq.fourth}, we   prove 2.

3. Notice that
\begin{align*}
\vert
\frac{1}{T}\sum_{t_1 = 1}^T\sum_{t_2 = 1}^T K\left(\frac{t_1 - t_2}{k_T}\right)\boldsymbol\epsilon^{(t_1)}_i\boldsymbol\epsilon^{(t_2)}_j - T\times Cov(\overline{\boldsymbol\epsilon}_i, \overline{\boldsymbol\epsilon}_j)
\vert\\
\leq \vert\frac{1}{T}\sum_{t_1 = 1}^T\sum_{t_2 = 1}^T \left(1 - K\left(\frac{t_1 - t_2}{k_T}\right)\right)\mathbf{E}\boldsymbol\epsilon^{(t_1)}_i\boldsymbol\epsilon^{(t_2)}_j\vert\\
+ \vert\frac{1}{T}\sum_{t_1 = 1}^T\sum_{t_2 = 1}^TK\left(\frac{t_1 - t_2}{k_T}\right)\left(\boldsymbol\epsilon^{(t_1)}_i\boldsymbol\epsilon^{(t_2)}_j - \mathbf{E}\boldsymbol\epsilon^{(t_1)}_i\boldsymbol\epsilon^{(t_2)}_j\right)\vert.
\end{align*}
From section 0.9.7 in \cite{MR2978290},  (B.33) and (B.34) in \cite{zhangPolitis}, and eq.\eqref{eq.covS},
\begin{align*}
\vert\frac{1}{T}\sum_{t_1 = 1}^T\sum_{t_2 = 1}^T \left(1 - K\left(\frac{t_1 - t_2}{k_T}\right)\right)\mathbf{E}\boldsymbol\epsilon^{(t_1)}_i\boldsymbol\epsilon^{(t_2)}_j\vert
\leq \frac{C}{T}\sum_{t_1 = 1}^T\sum_{t_2 = 1}^T \left(1 - K\left(\frac{t_1 - t_2}{k_T}\right)\right)(1 + \vert t_1 - t_2\vert)^{-\alpha}\\
\leq C_1\sum_{s = 0}^\infty \left(1 - K\left(\frac{s}{k_T}\right)\right)(1 + s)^{-\alpha} = O(v_T).
\end{align*}
On the other hand,
\begin{align*}
\Vert\
\frac{1}{T}\sum_{t_1 = 1}^T\sum_{t_2 = 1}^TK\left(\frac{t_1 - t_2}{k_T}\right)\left(\boldsymbol\epsilon^{(t_1)}_i\boldsymbol\epsilon^{(t_2)}_j - \mathbf{E}\boldsymbol\epsilon^{(t_1)}_i\boldsymbol\epsilon^{(t_2)}_j\right)\
\Vert_{m/2}\\
\leq \frac{1}{T}\sum_{s = 0}^{T - 1}K\left(\frac{s}{k_T}\right)\Vert\sum_{t_2 = 1}^{T - s}\boldsymbol\epsilon^{(t_2 + s)}_i\boldsymbol\epsilon^{(t_2)}_j - \mathbf{E}\boldsymbol\epsilon^{(t_2 + s)}_i\boldsymbol\epsilon^{(t_2)}_j\Vert_{m / 2}\\
+ \frac{1}{T}\sum_{s = 0}^{T - 1}K\left(\frac{s}{k_T}\right)\Vert\sum_{t_1 = 1}^{T - s}\boldsymbol\epsilon^{(t_1)}_i\boldsymbol\epsilon^{(t_1 + s)}_j - \mathbf{E}\boldsymbol\epsilon^{(t_1)}_i\boldsymbol\epsilon^{(t_1 + s)}_j\Vert_{m / 2}.
\end{align*}
Similar to eq.(B.37) in \cite{MR4441125},
\begin{align*}
\Vert\sum_{t_2 = 1}^{T - s}\boldsymbol\epsilon^{(t_2 + s)}_i\boldsymbol\epsilon^{(t_2)}_j - \mathbf{E}\boldsymbol\epsilon^{(t_2 + s)}_i\boldsymbol\epsilon^{(t_2)}_j\Vert_{m / 2}
\leq \Vert\sum_{t_2 = 1}^{T - s}\mathbf{E}\boldsymbol\epsilon^{(t_2 + s)}_i\boldsymbol\epsilon^{(t_2)}_j|\mathcal{F}^{(t_2 + s, 0)} - \mathbf{E}\boldsymbol\epsilon^{(t_2 + s)}_i\boldsymbol\epsilon^{(t_2)}_j\Vert_{m/2}\\
+ \sum_{l = 1}^\infty\Vert\sum_{t_2 = 1}^{T - s}\left(\mathbf{E}\boldsymbol\epsilon^{(t_2 + s)}_i\boldsymbol\epsilon^{(t_2)}_j|\mathcal{F}^{(t_2 + s, l)} - \mathbf{E}\boldsymbol\epsilon^{(t_2 + s)}_i\boldsymbol\epsilon^{(t_2)}_j|\mathcal{F}^{(t_2 + s, l - 1)}\right)\Vert_{m/2}\\
\leq C\sqrt{T} + C\sqrt{T}\left(\sum_{l =1}^\infty\delta^{(l)} + \sum_{l = 1}^\infty \delta^{(l - s)}\right)\leq C_1\sqrt{T}.
\end{align*}
Therefore
\begin{equation}
\begin{aligned}
\Vert\
\frac{1}{T}\sum_{t_1 = 1}^T\sum_{t_2 = 1}^TK\left(\frac{t_1 - t_2}{k_T}\right)\left(\boldsymbol\epsilon^{(t_1)}_i\boldsymbol\epsilon^{(t_2)}_j - \mathbf{E}\boldsymbol\epsilon^{(t_1)}_i\boldsymbol\epsilon^{(t_2)}_j\right)\
\Vert_{m/2}
\leq \frac{C}{\sqrt{T}}\sum_{s = 0}^\infty K\left(\frac{s}{k_T}\right)\leq \frac{C_1 k_T}{\sqrt{T}}\\
\text{and } \Vert\ \max_{i,j = 1,\cdots, d}\vert\frac{1}{T}\sum_{t_1 = 1}^T\sum_{t_2 = 1}^TK\left(\frac{t_1 - t_2}{k_T}\right)\boldsymbol\epsilon^{(t_1)}_i\boldsymbol\epsilon^{(t_2)}_j - T\times Cov(\overline{\boldsymbol\epsilon}_i, \overline{\boldsymbol\epsilon}_j)\vert\ \Vert_{m/2}\\
\leq \max_{i,j = 1,\cdots, d}\vert\frac{1}{T}\sum_{t_1 = 1}^T\sum_{t_2 = 1}^T \left(1 - K\left(\frac{t_1 - t_2}{k_T}\right)\right)\mathbf{E}\boldsymbol\epsilon^{(t_1)}_i\boldsymbol\epsilon^{(t_2)}_j\vert\\
+ \Vert\ \max_{i,j = 1,\cdots, d}\vert\frac{1}{T}\sum_{t_1 = 1}^T\sum_{t_2 = 1}^TK\left(\frac{t_1 - t_2}{k_T}\right)\left(\boldsymbol\epsilon^{(t_1)}_i\boldsymbol\epsilon^{(t_2)}_j - \mathbf{E}\boldsymbol\epsilon^{(t_1)}_i\boldsymbol\epsilon^{(t_2)}_j\right)\vert\ \Vert_{m / 2}\\
= O\left(v_T + d^{4/m}\times \frac{k_T}{\sqrt{T}}\right),
\end{aligned}
\end{equation}
and we prove 3.
\end{proof}
\section{Proofs of theorems in section \ref{section.Theoretical}}
Concatenate $\mathbf{x}^{(t)}$ into $\mathbf{z}^{(t)} = (\mathbf{x}^{(t)\top}, \mathbf{x}^{(t -  1)\top}, \cdots, \mathbf{x}^{(t - p + 1)\top})^\top$
for $t\in\mathbf{Z}$. From lemma \ref{lemma.operation}, we know that $\mathbf{z}^{(t)}$ are $(m,\alpha)-$short range dependent random variables.
Moreover, define $\boldsymbol\Sigma^{(0)} = \mathbf{E}\mathbf{z}^{(p)}\mathbf{z}^{(p)\top}$ and
$\boldsymbol\Sigma^{(1)} = \mathbf{E}\mathbf{z}^{(p)}\mathbf{z}^{(p + 1)\top}$, then
$\mathbf{z}^{(t)}\mathbf{z}^{(t)\top} - \boldsymbol\Sigma^{(0)}$ and $\mathbf{z}^{(t)}\mathbf{z}^{(t + 1)\top} - \boldsymbol\Sigma^{(1)}$ should be $(\frac{m}{2},\alpha)-$short range dependent random variables. With the help of these observations, we are able to derive an important lemma, which serves as a foundation for the follow-up analysis.
\begin{lemma}
  Suppose assumption 1 to 3 in section \ref{section.Theoretical}. Then we have
  \begin{equation}\label{eq.covariances}
  \begin{aligned}
    \Vert\ \vert\frac{1}{T}\sum_{t = p}^{T - 1}\mathbf{z}^{(t)}\mathbf{z}^{(t)\top} - \boldsymbol\Sigma^{(0)}\vert_\infty\ \Vert_{m/2}
    = O\left(T^{\frac{4\alpha_d}{m} - \frac{1}{2}}\right)\\
    \text{and } \Vert\ \vert\frac{1}{T}\sum_{t = p}^{T - 1}\mathbf{z}^{(t)}\mathbf{z}^{(t + 1)\top} - \boldsymbol\Sigma^{(1)} \vert_\infty\ \Vert_{m/2} = O\left(T^{\frac{4\alpha_d}{m} - \frac{1}{2}}\right).
  \end{aligned}
  \end{equation}
  \label{lemma.covariance_matrix}
\end{lemma}
\begin{proof}[proof of lemma \ref{lemma.covariance_matrix}]
For
\begin{align*}
\vert\boldsymbol\Sigma^{(0)}_{ij}\vert = \vert\mathbf{E}\mathbf{z}^{(p)}_i\mathbf{z}^{(p)}_j\vert\leq \Vert\mathbf{z}^{(p)}_i\Vert_m\times \Vert\mathbf{z}^{(p)}_j\Vert_m\leq C\\
\text{and }\vert\boldsymbol\Sigma^{(1)}_{ij}\vert= \vert\mathbf{E}\mathbf{z}^{(p)}_i\mathbf{z}^{(p + 1)}_j\vert\leq \Vert\mathbf{z}^{(p)}_i\Vert_m\times \Vert\mathbf{z}^{(p + 1)}_j\Vert_m\leq C.
\end{align*}

From theorem \ref{theorem.MA},
\begin{align*}
\Vert\ \vert\frac{1}{T}\sum_{t = p}^{T - 1}\mathbf{z}^{(t)}\mathbf{z}^{(t)\top} - \boldsymbol\Sigma^{(0)}\vert_\infty\ \Vert_{m/2}
\leq \Vert\ \vert\frac{1}{T}\sum_{t = p}^{T - 1}(\mathbf{z}^{(t)}\mathbf{z}^{(t)\top} - \boldsymbol\Sigma^{(0)})\vert_\infty\ \Vert_{m/2}
+ \frac{p}{T}\vert\boldsymbol\Sigma^{(0)}\vert_\infty\\
\leq \frac{Cd^{\frac{4}{m}}}{\sqrt{T}}\leq C_1\times T^{\frac{4\alpha_d}{m} - \frac{1}{2}}\\
\text{and } \Vert\ \vert\frac{1}{T}\sum_{t = p}^{T - 1}\mathbf{z}^{(t)}\mathbf{z}^{(t + 1)\top} - \boldsymbol\Sigma^{(1)}\vert_\infty\ \Vert_{m/2}
\leq \Vert\ \vert\frac{1}{T}\sum_{t = p}^{T - 1}(\mathbf{z}^{(t)}\mathbf{z}^{(t + 1)\top} - \boldsymbol\Sigma^{(1)})\vert_\infty\ \Vert_{m/2} + \frac{p}{T}\vert\boldsymbol\Sigma^{(1)}\vert_\infty\\
\leq C_2\times T^{\frac{4\alpha_d}{m } - \frac{1}{2}},
\end{align*}
and we prove lemma \ref{lemma.covariance_matrix}.
\end{proof}

\begin{proof}[proof of theorem \ref{theorem.Lasso_consistency}]
Define $\widehat{\boldsymbol\Sigma}^{(0)} = \frac{1}{T}\sum_{t = p}^{T - 1}\mathbf{z}^{(t)}\mathbf{z}^{(t)\top}$ and
$\widehat{\boldsymbol\Sigma}^{(1)} = \frac{1}{T}\sum_{t = p}^{T - 1}\mathbf{z}^{(t)}\mathbf{z}^{(t + 1)\top}$. By definition
\begin{align*}
\frac{1}{2T}\vert \mathbf{y}^{(l)} - \mathbf{W}\widetilde{\mathbf{S}}_{\cdot l}\vert_2^2 + \lambda \vert\widetilde{\mathbf{S}}_{\cdot l}\vert_1
\leq \frac{1}{2T}\vert\mathbf{y}^{(l)} - \mathbf{W}\mathbf{S}_{\cdot l}\vert_2^2 + \lambda\vert\mathbf{S}_{\cdot l}\vert_1\\
\Rightarrow \frac{1}{2}(\widetilde{\mathbf{S}}_{\cdot l} - \mathbf{S}_{\cdot l})^\top\widehat{\boldsymbol\Sigma}^{(0)}(\widetilde{\mathbf{S}}_{\cdot l} - \mathbf{S}_{\cdot l})
\leq \lambda\vert\mathbf{S}_{\cdot l}\vert_1 - \lambda\vert\widetilde{\mathbf{S}}_{\cdot l}\vert_1 + \frac{1}{T}\boldsymbol\eta^{(l)\top}\mathbf{W}
(\widetilde{\mathbf{S}}_{\cdot l} - \mathbf{S}_{\cdot l})\\
\leq \lambda\sum_{j\in\mathcal{B}_l}\vert\widetilde{\mathbf{S}}_{j l} - \mathbf{S}_{jl}\vert - \lambda\sum_{j\not\in\mathcal{B}_l}\vert\widetilde{\mathbf{S}}_{j l}\vert
+ \vert\frac{1}{T}\mathbf{W}^\top\boldsymbol\eta^{(l)}\vert_\infty\times
\vert\widetilde{\mathbf{S}}_{\cdot l} - \mathbf{S}_{\cdot l}\vert_1
\end{align*}
From the Yule-Walker equation (see \cite{MR1238940})
\begin{equation}
\begin{aligned}
\mathbf{z}^{(t)}_j\boldsymbol\epsilon^{(t+1)}_l = (\mathbf{z}^{(t)}_j\mathbf{z}_l^{(t + 1)} - \boldsymbol\Sigma^{(1)}_{jl})- \sum_{k =  1}^{pd}\mathbf{S}_{kl}(\mathbf{z}_k^{(t)}\mathbf{z}^{(t)}_j - \boldsymbol\Sigma^{(0)}_{kj}),
\end{aligned}
\label{eq.Yule_walker}
\end{equation}
also notice that $\sum_{k = 1}^{pd}\vert\mathbf{S}_{kl}\vert\leq C\times \max_{l = 1,\cdots, d}\vert\mathcal{B}_l\vert\leq C_0$ for all $l$.
Therefore, from lemma \ref{lemma.operation}, $\mathbf{z}^{(t)}\boldsymbol\epsilon^{(t + 1)\top}$ are also $(m/2, \alpha)-$short range dependent random variables.
For $(\mathbf{W}^\top\boldsymbol\eta^{(l)})_j = \sum_{t = p}^{T - 1}\mathbf{z}^{(t)}_j\boldsymbol\epsilon^{(t+1)}_l$, theorem \ref{theorem.MA} implies
\begin{align*}
\max_{l = 1,\cdots, d}\vert\frac{1}{T}\mathbf{W}^\top\boldsymbol\eta^{(l)}\vert_\infty = \max_{j = 1,\cdots, pd, l = 1,\cdots, d}\vert\frac{1}{T}\sum_{t = p}^{T - 1}\mathbf{z}^{(t)}_j\boldsymbol\epsilon^{(t+1)}_l\vert = O_p\left(T^{\frac{4\alpha_d}{m} - \frac{1}{2}}\right).
\end{align*}
From assumption 2, with probability tending to 1
\begin{align*}
0\leq \frac{3\lambda}{2}\sum_{j\in\mathcal{B}_l}\vert\widetilde{\mathbf{S}}_{j l} - \mathbf{S}_{jl}\vert - \frac{\lambda}{2}\sum_{j\not\in\mathcal{B}_l}\vert\widetilde{\mathbf{S}}_{j l}\vert\\
\Rightarrow \sum_{j\not\in\mathcal{B}_l}\vert\widetilde{\mathbf{S}}_{j l}\vert\leq 3\sum_{j\in\mathcal{B}_l}\vert\widetilde{\mathbf{S}}_{j l} - \mathbf{S}_{jl}\vert\ \text{for all }l = 1,\cdots, d.
\end{align*}
Therefore, with probability tending to $1$
\begin{equation}
\begin{aligned}
\vert
\widetilde{\mathbf{S}}_{\cdot l} - \mathbf{S}_{\cdot l}
\vert_1\leq 4\sum_{j\in\mathcal{B}_l}\vert\widetilde{\mathbf{S}}_{j l} - \mathbf{S}_{jl}\vert\leq 4\sqrt{\vert\mathcal{B}\vert_l}\times \vert\widetilde{\mathbf{S}}_{\cdot l} - \mathbf{S}_{\cdot l}\vert_2\leq C\vert \widetilde{\mathbf{S}}_{\cdot l} - \mathbf{S}_{\cdot l}\vert_2,
\end{aligned}
\label{eq.Lasso_L1}
\end{equation}
and from assumption 2, for sufficiently large $T$
\begin{align*}
(\widetilde{\mathbf{S}}_{\cdot l} - \mathbf{S}_{\cdot l})^\top\widehat{\boldsymbol\Sigma}^{(0)}(\widetilde{\mathbf{S}}_{\cdot l} - \mathbf{S}_{\cdot l})
\geq (\widetilde{\mathbf{S}}_{\cdot l} - \mathbf{S}_{\cdot l})^\top\boldsymbol\Sigma^{(0)}(\widetilde{\mathbf{S}}_{\cdot l} - \mathbf{S}_{\cdot l})
- \vert \widehat{\boldsymbol\Sigma}^{(0)} - \boldsymbol\Sigma^{(0)}\vert_\infty\times \vert\widetilde{\mathbf{S}}_{\cdot l} - \mathbf{S}_{\cdot l}\vert_1^2\\
\geq \frac{c}{2}\vert\widetilde{\mathbf{S}}_{\cdot l} - \mathbf{S}_{\cdot l}\vert_2^2.
\end{align*}
Correspondingly, we have
\begin{align*}
\frac{c}{2}\vert\widetilde{\mathbf{S}}_{\cdot l} - \mathbf{S}_{\cdot l}\vert_2^2\leq \frac{3\lambda}{2}\sum_{j\in\mathcal{B}_l}\vert\widetilde{\mathbf{S}}_{j l} - \mathbf{S}_{jl}\vert\leq C\lambda\vert\widetilde{\mathbf{S}}_{\cdot l} - \mathbf{S}_{\cdot l}\vert_2\Rightarrow \vert\widetilde{\mathbf{S}}_{\cdot l} - \mathbf{S}_{\cdot l}\vert_2\leq C\lambda
\end{align*}
for any $l$ with probability tending to $1$. From eq.\eqref{eq.Lasso_L1}, with probability tending to $1$
\begin{align*}
\vert\widetilde{\mathbf{S}}_{\cdot l} - \mathbf{S}_{\cdot l}\vert_1\leq C\vert\widetilde{\mathbf{S}}_{\cdot l} - \mathbf{S}_{\cdot l}\vert_2\leq C_1\lambda
\end{align*}
and we prove eq.\eqref{eq.Lasso_consistency}.

For
\begin{align*}
Prob\left(\cup_{l = 1,\cdots, d}(\widetilde{\mathcal{B}}_l \neq \mathcal{B}_l)\right)\leq
Prob\left(\exists l = 1,\cdots, d, j\in\mathcal{B}_l: \vert\widetilde{\mathbf{S}}_{jl}\vert\leq b_T\right)\\
+ Prob\left(\exists l = 1,\cdots, d, j\not\in\mathcal{B}_l: \vert\widetilde{\mathbf{S}}_{jl}\vert > b_T\right).
\end{align*}
Notice that
\begin{align*}
\vert\widetilde{\mathbf{S}}_{jl}\vert\geq \vert\mathbf{S}_{jl}\vert - \vert\widetilde{\mathbf{S}}_{jl} - \mathbf{S}_{jl}\vert,\\
\text{so } Prob\left(\exists l = 1,\cdots, d, j\in\mathcal{B}_l: \vert\widetilde{\mathbf{S}}_{jl}\vert\leq b_T\right)
\leq Prob\left(\max_{l = 1,\cdots, d}\vert\widetilde{\mathbf{S}}_{\cdot l} - \mathbf{S}_{\cdot l}\vert_\infty > b_T\right) = o(1).
\end{align*}
Similarly,
\begin{align*}
Prob\left(\exists l = 1,\cdots, d, j\not\in\mathcal{B}_l: \vert\widetilde{\mathbf{S}}_{jl}\vert > b_T\right)\leq
Prob\left(\max_{l = 1,\cdots, d}\vert\widetilde{\mathbf{S}}_{jl} - \mathbf{S}_{jl}\vert > b_T\right) = o(1),
\end{align*}
and we prove eq.\eqref{eq.model_selection}.
\end{proof}

Before proving the asymptotic distribution of the post-selection estimator $\widehat{\mathbf{S}}$, we derive the algebraic properties of the partial inverse operator in the following lemma.

\begin{lemma}
The partial inverse operator $\mathcal{F}_{\cdot}(\cdot)$ in definition \ref{definition.partial_inverse} has the following properties:

i. For a given matrix $\mathbf{A}\in\mathbf{R}^{m\times m}$ and a set $\mathcal{B} = \{1\leq k_1 < k_2 < \cdots < k_s\leq m\}$, define the submatrix $\mathbf{A}_\mathcal{B}$ as in the introduction. If $\mathbf{A}$ is symmetric and all of its eigenvalues are greater than a constant $c>0$, then
\begin{equation}
\mathcal{F}_{\mathcal{B}}(\mathbf{A})\ \text{is symmetric and } \vert\mathbf{F}_\mathcal{B}(\mathbf{A})\vert_{L_1}\leq \frac{\vert\mathcal{B}\vert}{c}.
\label{eq.first_beta}
\end{equation}

ii. Suppose $\mathbf{A}_\mathcal{B}$ is invertible. In addition suppose $\boldsymbol\beta\in\mathcal{R}^m$ satisfies $\boldsymbol\beta_i = 0$ if $i\not\in\mathcal{B}$. Then we have
\begin{equation}
\mathcal{F}_\mathcal{B}(\mathbf{A})\mathbf{A}\boldsymbol\beta = \boldsymbol\beta.
\label{eq.second_beta}
\end{equation}
\label{lemma.post-selection}
\end{lemma}
Therefore, even if the matrix $\mathbf{A}$ is not invertible, which is a common practice when the number of parameters to estimate exceeds the sample size in linear regression, the partial inverse matrix $\mathcal{F}_\mathcal{B}(\mathbf{A})$ still behaves like the inverse matrix when acting on a sparse vector $\boldsymbol\beta$. This property motivates the development of the post-selection estimator $\widehat{\mathbf{S}}$, which can be recognized as a generalization of the least square estimator to high-dimensional setting.
\begin{proof}[proof of lemma \ref{lemma.post-selection}]
i. For $\mathbf{A}$ is symmetric and positive definite, $\mathbf{A}_\mathcal{B}$ is symmetric and positive definite, and its inverse $\mathbf{A}_\mathcal{B}^{-1}$ is symmetric. Therefore, eq.\eqref{eq.matrix_C} implies that $\mathcal{F}_\mathcal{B}(\mathbf{A})$ is symmetric. Moreover, for
$\vert\mathbf{A}_\mathcal{B}^{-1}\vert_2\leq \frac{1}{c}$,
\begin{align*}
\vert\mathcal{F}_\mathcal{B}(\mathbf{A})\vert_{L_1} = \vert\mathbf{A}_\mathcal{B}^{-1}\vert_{L_1}\leq \frac{1}{c}\times \vert\mathcal{B}\vert,
\end{align*}
and we prove eq.\eqref{eq.first_beta}.

ii. For any $j = 1,\cdots, m$, if $j\not\in\mathcal{B}$, then $\mathcal{F}_\mathcal{B}(\mathbf{A})_{j\cdot} = 0$, so the $j$th element of $\mathcal{F}_\mathcal{B}(\mathbf{A})\mathbf{A}\boldsymbol\beta$ is $0$, which equals $\boldsymbol\beta_j$. On the other hand, if $j = k_v$ for some $v = 1,\cdots, s$, then
\begin{align*}
(\mathcal{F}_\mathcal{B}(\mathbf{A})\mathbf{A}\boldsymbol\beta)_j = \sum_{q = 1}^s \mathcal{F}_\mathcal{B}(\mathbf{A})_{jk_q}\sum_{r = 1}^s \mathbf{A}_{k_qk_r}\boldsymbol\beta_{k_r} = \sum_{r = 1}^s \boldsymbol\beta_{k_r}\sum_{q = 1}^s \mathcal{F}_\mathcal{B}(\mathbf{A})_{jk_q}\mathbf{A}_{k_qk_r} = \boldsymbol\beta_j,
\end{align*}
and we prove eq.\eqref{eq.second_beta}.
\end{proof}
A key difference between the partial inverse operator $\mathcal{F}_{\cdot}(\cdot)$ and the Moore -- Penrose pusedo--inverse is that $\mathcal{F}_{\cdot}(\cdot)$ requires the sparsity structure of $\boldsymbol\beta$, and is able to maintain the sparsity pattern of $\boldsymbol\beta$; while the Moore - Penrose pusedo-inverse normally cannot recover the sparsity structure when $\mathbf{A}$ is not invertible.

\begin{proof}[proof of theorem \ref{theorem.Gaussian_approx_}]
From theorem \ref{theorem.Lasso_consistency}, $Prob\left(\cap_{l = 1,\cdots, d}(\widetilde{\mathcal{B}}_l  = \mathcal{B}_l)\right)\to 1$ as the sample size $T\to\infty$.
For $\boldsymbol\Sigma^{(0)}$'s eigenvalues are assumed to be greater than a constant $c>0$, lemma \ref{lemma.post-selection} implies
$$
\vert\mathcal{F}_{\mathcal{B}_l}(\boldsymbol\Sigma^{(0)})\vert_{L_1}\leq \frac{1}{c}\times \max_{l = 1,\cdots, d}\vert\mathcal{B}_l\vert\leq C_1.
$$
Therefore, define the vector
\begin{align*}
\boldsymbol{\kappa}^{(t + 1)} =
\left(
\begin{matrix}
(\mathcal{F}_{\mathcal{B}_1}(\boldsymbol\Sigma^{(0)})\mathbf{z}^{(t)}\boldsymbol\epsilon_1^{(t + 1)})^\top &
(\mathcal{F}_{\mathcal{B}_2}(\boldsymbol\Sigma^{(0)})\mathbf{z}^{(t)}\boldsymbol\epsilon_2^{(t + 1)})^\top &
\cdots &
(\mathcal{F}_{\mathcal{B}_d}(\boldsymbol\Sigma^{(0)})\mathbf{z}^{(t)}\boldsymbol\epsilon_d^{(t + 1)})^\top
\end{matrix}
\right)^\top\in\mathbf{R}^{pd^2},
\end{align*}
here $\mathbf{z}^{(t)} = (\mathbf{x}^{(t)\top},\cdots, \mathbf{x}^{(t - p + 1)\top})^\top$ and $t\in\mathbf{Z}$. Then
\begin{align*}
\mathbf{E}\mathcal{F}_{\mathcal{B}_l}(\boldsymbol\Sigma^{(0)})\mathbf{z}^{(t)}\boldsymbol\epsilon_l^{(t + 1)} = \mathcal{F}_{\mathcal{B}_l}(\boldsymbol\Sigma^{(0)})\mathbf{E}\mathbf{z}^{(t)}\boldsymbol\epsilon_l^{(t + 1)} = 0.
\end{align*}
Besides, for any $j = 1,\cdots, pd$,
\begin{align*}
\Vert\ (\mathcal{F}_{\mathcal{B}_l}(\boldsymbol\Sigma^{(0)})\mathbf{z}^{(t)}\boldsymbol\epsilon_l^{(t + 1)})_j\ \Vert_{m/2}
= \Vert\sum_{q = 1}^{pd}(\mathcal{F}_{\mathcal{B}_l}(\boldsymbol\Sigma^{(0)}))_{jq}\mathbf{z}^{(t)}_q\boldsymbol\epsilon_l^{(t + 1)}\Vert_{m / 2}\\
\leq \vert\mathcal{F}_{\mathcal{B}_l}(\boldsymbol\Sigma^{(0)})\vert_{L_1}\times \max_{q = 1,\cdots, pd}\Vert\mathbf{z}^{(t)}_q\Vert_m\times \Vert\epsilon_l^{(t + 1)}\Vert_m\leq C_0
\end{align*}
for a constant $C_0$.
Moreover,
\begin{align*}
\Vert
\sum_{q = 1}^{pd}(\mathcal{F}_{\mathcal{B}_l}(\boldsymbol\Sigma^{(0)}))_{jq}\mathbf{z}^{(t)}_q\boldsymbol\epsilon_l^{(t + 1)}
- \sum_{q = 1}^{pd}(\mathcal{F}_{\mathcal{B}_l}(\boldsymbol\Sigma^{(0)}))_{jq}\mathbf{z}^{(t, j - 1)}_q\boldsymbol\epsilon_l^{(t + 1, j)}
\Vert_{m/2}\\
\leq \vert\mathcal{F}_{\mathcal{B}_l}(\boldsymbol\Sigma^{(0)})\vert_{L_1}\times \left(\max_{q = 1,\cdots, pd}\Vert \mathbf{z}_q^{(t)}\Vert_m\times \Vert\epsilon_l^{(t + 1)} - \epsilon_l^{(t + 1, j)}\Vert_{m} + \Vert\boldsymbol\epsilon_l^{(t+1, j)}\Vert_m\times \max_{q = 1,\cdots, pd}\Vert \mathbf{z}_q^{(t)} -  \mathbf{z}_q^{(t, j - 1)}\Vert_m\right).
\end{align*}
Therefore, by definition $\boldsymbol\kappa^{(t)}$ are $(\frac{m}{2},\alpha)-$short range dependent random variables. For
\begin{align*}
\max_{l = 1,\cdots, d, j = 1,\cdots, pd}\frac{1}{\sqrt{T}}\vert\sum_{t = p}^{T - 1}\boldsymbol\kappa_{(l - 1)\times pd + j}^{(t + 1)}\vert = \max_{l = 1,\cdots, d, j\in\mathcal{B}_l}\frac{1}{\sqrt{T}}\vert\sum_{t = p}^{T - 1}\boldsymbol\kappa_{(l - 1)\times pd + j}^{(t + 1)}\vert,
\end{align*}
from assumption 3 and according to theorem \ref{theorem.MA},
\begin{align*}
\sup_{x\in\mathbf{R}}\vert
Prob\left(\max_{l = 1,\cdots, d, j = 1,\cdots, pd}\frac{1}{\sqrt{T - p}}\vert\sum_{t = p}^{T - 1}\boldsymbol\kappa_{(l - 1)\times pd + j}^{(t + 1)}\vert\leq x\right)\\
 - Prob\left(\max_{l = 1,\cdots, d, j\in\mathcal{B}_l}\vert\boldsymbol\Gamma_{jl}^\circ\vert\leq x\right)
\vert = o(1),
\end{align*}
here $\boldsymbol\Gamma_{jl}^\circ, l = 1,2,\cdots, d, j\in\mathcal{B}_l$ are joint normal random variables with mean $0$ and
$$
Cov(\boldsymbol\Gamma_{j_1l_1}^\circ, \boldsymbol\Gamma_{j_2l_2}^\circ) = \frac{T}{T - p}\boldsymbol\Lambda^{(l_1, l_2)}_{j_1 j_2} = \boldsymbol\Lambda^{(l_1, l_2)}_{j_1 j_2} + O(1 / T),
$$
and $\boldsymbol\Lambda^{(l_1, l_2)}_{j_1 j_2}$ is defined in eq.\eqref{eq.cov_matrixX}. For any $x$,
\begin{align*}
Prob\left(\max_{l = 1,\cdots, d, j = 1,\cdots, pd}\frac{1}{\sqrt{T}}\vert\sum_{t = p}^{T - 1}\boldsymbol\kappa_{(l - 1)\times pd + j}^{(t + 1)}\vert\leq x\right)\\
= Prob\left(\max_{l = 1,\cdots, d, j = 1,\cdots, pd}\frac{1}{\sqrt{T - p}}\vert\sum_{t = p}^{T - 1}\boldsymbol\kappa_{(l - 1)\times pd + j}^{(t + 1)}\vert\leq x + \frac{\sqrt{T} - \sqrt{T - p}}{\sqrt{T - p}}x\right).
\end{align*}
From lemma B.1. in \cite{zhangPolitis}, we prove
\begin{equation}
\begin{aligned}
\sup_{x\in\mathbf{R}}\vert
Prob\left(\max_{l = 1,\cdots, d, j = 1,\cdots, pd}\frac{1}{\sqrt{T}}\vert\sum_{t = p}^{T - 1}\boldsymbol\kappa_{(l - 1)\times pd + j}^{(t + 1)}\vert\leq x\right)\\
 - Prob\left(\max_{l = 1,\cdots, d, j\in\mathcal{B}_l}\vert\boldsymbol\Gamma_{jl}\vert\leq x\right)
\vert = o(1).
\end{aligned}
\end{equation}
See eq.\eqref{eq.prob_result} for the definition of $\boldsymbol\Gamma_{jl}$.

On the other hand, form lemma \ref{lemma.covariance_matrix}, for any vector $\mathbf{c}\in\mathbf{R}^{\vert\mathcal{B}_l\vert}$,
\begin{align*}
\mathbf{c}^\top\widehat{\boldsymbol\Sigma}^{(0)}_{\mathcal{B}_l}\mathbf{c}\geq
\mathbf{c}^\top\boldsymbol\Sigma^{(0)}_{\mathcal{B}_l}\mathbf{c} - \vert\mathbf{c}^\top(\boldsymbol\Sigma^{(0)}_{\mathcal{B}_l} - \widehat{\boldsymbol\Sigma}_{\mathcal{B}_l}^{(0)})\mathbf{c}\vert
\geq c\vert\mathbf{c}\vert_2^2 - \vert\boldsymbol\Sigma^{(0)} - \widehat{\boldsymbol\Sigma}^{(0)}\vert_\infty\times \vert\mathbf{c}\vert_1^2\\
\geq c\vert\mathbf{c}\vert_2^2 - \vert\boldsymbol\Sigma^{(0)} - \widehat{\boldsymbol\Sigma}^{(0)}\vert_\infty\times \max_{l = 1,\cdots, d}\vert\mathcal{B}_l\vert\times \vert\mathbf{c}\vert_2^2
\geq \frac{c}{2}\vert\mathbf{c}\vert_2^2\\
\text{and } \vert\boldsymbol\Sigma^{(0)} - \widehat{\boldsymbol\Sigma}^{(0)}\vert_2 \leq
\vert\boldsymbol\Sigma^{(0)} - \widehat{\boldsymbol\Sigma}^{(0)}\vert_\infty\times \max_{l  = 1,\cdots,d}\vert\mathcal{B}_l\vert
\end{align*}
for sufficiently large $T$ with probability tending to $1$. Therefore, all $\widehat{\boldsymbol\Sigma}^{(0)}_{\mathcal{B}_l}$ are positive definite, and from eq.(5.8.2) in \cite{MR2978290},
\begin{equation}
\begin{aligned}
\vert\mathcal{F}_{\mathcal{B}_l}(\widehat{\boldsymbol\Sigma}^{(0)}) - \mathcal{F}_{\mathcal{B}_l}(\boldsymbol\Sigma^{(0)})\vert_{L_1}
= \vert(\widehat{\boldsymbol\Sigma}^{(0)}_{\mathcal{B}_l})^{-1} - (\boldsymbol\Sigma^{(0)}_{\mathcal{B}_l})^{-1}\vert_{L_1}\leq C\vert(\widehat{\boldsymbol\Sigma}^{(0)}_{\mathcal{B}_l})^{-1} - (\boldsymbol\Sigma^{(0)}_{\mathcal{B}_l})^{-1}\vert_2\\
\leq C_1\times \vert\boldsymbol\Sigma^{(0)} - \widehat{\boldsymbol\Sigma}^{(0)}\vert_\infty\ \text{for sufficiently large } T\\
\Rightarrow \max_{l = 1,\cdots, d}\vert\mathcal{F}_{\mathcal{B}_l}(\widehat{\boldsymbol\Sigma}^{(0)}) - \mathcal{F}_{\mathcal{B}_l}(\boldsymbol\Sigma^{(0)})\vert_{L_1}
= O_p\left(T^{\frac{4\alpha_d}{m} - \frac{1}{2}}\right).
\label{eq.delta_Sigma}
\end{aligned}
\end{equation}
Therefore, if $\widetilde{\mathcal{B}}_l = \mathcal{B}_l$ for $l  =1,\cdots, d$, then
\begin{equation}
\begin{aligned}
\max_{j = 1,\cdots, pd, l = 1,\cdots, d}
\vert
\sqrt{T}(\widehat{\mathbf{S}}_{j l} - \mathbf{S}_{j l}) - \frac{1}{\sqrt{T}}\sum_{t = p}^{T - 1}\boldsymbol\kappa_{(l - 1)\times pd + j}^{(t + 1)}\vert\\
\leq \frac{1}{\sqrt{c}}\max_{l = 1,\cdots, d}\vert(\mathcal{F}_{\mathcal{B}_l}(\widehat{\boldsymbol\Sigma}^{(0)}) - \mathcal{F}_{\mathcal{B}_l}(\boldsymbol\Sigma^{(0)}))\frac{1}{\sqrt{T}}\sum_{t = p}^{T - 1}\mathbf{z}^{(t)}\boldsymbol\epsilon^{(t + 1)}_l\vert_\infty\\
\leq \frac{1}{\sqrt{c}}\max_{l = 1,\cdots, d}\vert\mathcal{F}_{\mathcal{B}_l}(\widehat{\boldsymbol\Sigma}^{(0)}) - \mathcal{F}_{\mathcal{B}_l}(\boldsymbol\Sigma^{(0)})\vert_{L_1}
\times \max_{l = 1,\cdots, d}\vert\frac{1}{\sqrt{T}}\sum_{t = p}^{T - 1}\mathbf{z}^{(t)}\boldsymbol\epsilon^{(t + 1)}_l\vert_\infty
\end{aligned}
\label{eq.diff}
\end{equation}
From eq.\eqref{eq.Yule_walker}, $\mathbf{z}^{(t)}\boldsymbol\epsilon^{(t + 1)\top}$ are $(m/2,\alpha)-$short range dependent random variables. Therefore, theorem \ref{theorem.MA} implies
\begin{align*}
\max_{l = 1,\cdots, d}\vert\frac{1}{\sqrt{T}}\sum_{t = p}^{T - 1}\mathbf{z}^{(t)}\boldsymbol\epsilon^{(t + 1)}_l\vert_\infty = O_p\left(T^{\frac{4\alpha_d}{m}}\right),
\end{align*}
and from eq.\eqref{eq.diff}, we have
\begin{equation}
\begin{aligned}
\max_{j = 1,\cdots, pd, l = 1,\cdots, d}
\vert
\sqrt{T}(\widehat{\mathbf{S}}_{j l} - \mathbf{S}_{j l}) - \frac{1}{\sqrt{T}}\sum_{t = p}^{T - 1}\boldsymbol\kappa_{(l - 1)\times pd + j}^{(t + 1)}\vert
= O_p\left(T^{\frac{8\alpha_d}{m} - \frac{1}{2}}\right).
\end{aligned}
\end{equation}
From lemma B.1 in \cite{zhangPolitis}, for any $\omega>0$ and sufficiently large $T$,
\begin{align*}
Prob\left(
\max_{j = 1,\cdots, pd, l = 1,\cdots, d}\sqrt{T}\vert\widehat{\mathbf{S}}_{jl} - \mathbf{S}_{jl}\vert\leq x
\right)\leq \omega\\
+ Prob\left(\max_{j = 1,\cdots, pd, l = 1,\cdots, d}\vert\frac{1}{\sqrt{T}}\sum_{t = p}^{T - 1}\boldsymbol\kappa_{(l - 1)\times pd + j}^{(t + 1)}\vert\leq x + C_\omega T^{\frac{8\alpha_d}{m} - \frac{1}{2}}\right)\\
\leq Prob\left(\max_{l = 1,\cdots, d, j\in\mathcal{B}_l}\vert\boldsymbol\Gamma_{jl}\vert\leq x\right) + \omega + C_1\times  T^{\frac{8\alpha_d}{m} - \frac{1}{2}}\times\log(T)\\
\text{and } Prob\left(\max_{j = 1,\cdots, pd, l = 1,\cdots, d}\sqrt{T}\vert\widehat{\mathbf{S}}_{jl} - \mathbf{S}_{jl}\vert\leq x
\right)\geq -\omega\\
+ Prob\left(\max_{j = 1,\cdots, pd, l = 1,\cdots, d}\vert\frac{1}{\sqrt{T}}\sum_{t = p}^{T - 1}\boldsymbol\kappa_{(l - 1)\times pd + j}^{(t + 1)}\vert\leq x - C_\omega T^{\frac{8\alpha_d}{m} - \frac{1}{2}}\right)\\
\geq Prob\left(\max_{l = 1,\cdots, d, j\in\mathcal{B}_l}\vert\boldsymbol\Gamma_{jl}\vert\leq x\right) - \omega - C_1\times  T^{\frac{8\alpha_d}{m} - \frac{1}{2}}\times\log(T),
\end{align*}
which implies eq.\eqref{eq.prob_result}.
\end{proof}

\begin{proof}[proof of theorem \ref{theorem.bootstrap}]
First from theorem \ref{theorem.Lasso_consistency}, we have $\widetilde{\mathcal{B}}_l = \mathcal{B}_l$ for $l = 1,\cdots, d$ with probability tending to $1$. If $\widetilde{\mathcal{B}}_l = \mathcal{B}_l$, then
\begin{align*}
\boldsymbol\Delta^*_{jl} = \sqrt{T}(\widehat{\mathbf{S}}^*_{j l} - \widehat{\mathbf{S}}_{j l})
= \frac{1}{\sqrt{T}}\sum_{t = p}^{T - 1}\sum_{q = 1}^{pd}\mathcal{F}_{\mathcal{B}_l}(\widehat{\boldsymbol\Sigma}^{(0)})_{jq}\widehat{\boldsymbol\Theta}^{(t)}_{q l}e^{(t)},
\end{align*}
so $\boldsymbol\Delta^*_{jl} = 0$ if $j\not\in\mathcal{B}_l$. If $j_1\in\mathcal{B}_{l_1}\ \text{and } j_2\in\mathcal{B}_{l_2}$, then for sufficiently large $T$ we have
\begin{align*}
\mathbf{E}^*\boldsymbol\Delta^*_{j_1l_1}\boldsymbol\Delta^*_{j_2l_2} = \frac{1}{T}
\sum_{t_1 = p}^{T - 1}\sum_{t_2 = p}^{T - 1}\sum_{q_1 = 1}^{pd}\sum_{q_2 = 1}^{pd}K\left(\frac{t_1 - t_2}{k_T}\right)\mathcal{F}_{\mathcal{B}_{l_1}}(\widehat{\boldsymbol\Sigma}^{(0)})_{j_1q_1}\\
\times\widehat{\boldsymbol\Theta}^{(t_1)}_{q_1 l_1}
\mathcal{F}_{\mathcal{B}_{l_2}}(\widehat{\boldsymbol\Sigma}^{(0)})_{j_2q_2}\widehat{\boldsymbol\Theta}^{(t_2)}_{q_2 l_2}.
\end{align*}
Notice that for $l = 1,\cdots, d$,
\begin{align*}
\widehat{\boldsymbol\Theta}^{(t)}_{ql}
= \mathbf{z}^{(t)}_q\boldsymbol\epsilon^{(t + 1)}_l - \sum_{s = 1}^{pd}\mathbf{z}^{(t)}_q\mathbf{z}^{(t)}_s(\widehat{\mathbf{S}}_{sl} - \mathbf{S}_{sl}),
\end{align*}
and define the terms
\begin{align*}
\boldsymbol\omega^{(t,l)}_{js} = \left(\sum_{q = 1}^{pd}\mathcal{F}_{\mathcal{B}_l}(\boldsymbol\Sigma^{(0)})_{jq}\mathbf{z}_q^{(t)}\right)\times \mathbf{z}_s^{(t)},\\
\text{then }\sum_{q = 1}^{pd}\mathcal{F}_{\mathcal{B}_l}(\boldsymbol\Sigma^{(0)})_{jq} \sum_{s = 1}^{pd}\mathbf{z}^{(t)}_q\mathbf{z}^{(t)}_s(\widehat{\mathbf{S}}_{sl} - \mathbf{S}_{sl}) = \sum_{s = 1}^{pd}\boldsymbol\omega^{(t,l)}_{js}(\widehat{\mathbf{S}}_{sl} - \mathbf{S}_{sl}).
\end{align*}
Moreover, we have
\begin{align*}
\Vert
\boldsymbol\omega_{js}^{(t,l)}
\Vert_{m/2}\leq \Vert\mathbf{z}_s^{(t)}\Vert_m\times \left(\sum_{q = 1}^{pd}\vert\mathcal{F}_{\mathcal{B}_l}(\boldsymbol\Sigma^{(0)})_{jq}\vert\times \Vert\mathbf{z}^{(t)}_q\Vert_m\right)\leq C
\end{align*}
for a fixed constant $C$. Define the $R^{pd^2}$ vector
\begin{align*}
\boldsymbol\psi^{(t + 1)}=
\left[
\begin{matrix}
\boldsymbol\epsilon^{(t+1)}_1\mathcal{F}_{\mathcal{B}_1}(\boldsymbol\Sigma^{(0)})\mathbf{z}^{(t)}\\
\boldsymbol\epsilon^{(t+1)}_2\mathcal{F}_{\mathcal{B}_2}(\boldsymbol\Sigma^{(0)})\mathbf{z}^{(t)}\\
\vdots\\
\boldsymbol\epsilon^{(t + 1)}_d\mathbf{F}_{\mathcal{B}_d}(\boldsymbol\Sigma^{(0)})\mathbf{z}^{(t)}
\end{matrix}
\right],\ \text{then }\mathbf{E} \boldsymbol\psi^{(t + 1)}  = 0.
\end{align*}
Besides, $\boldsymbol\psi^{(t + 1)}$ is a function of $\cdots, e_{t}, e_{t + 1}$ (see section \ref{section.m_alpha_dependent}). Denote the function
$\boldsymbol\psi^{(t + 1)}_i = Q^{(t + 1)}_i(\cdots, e_t, e_{t + 1})$.
For each $v = 0,1,\cdots,$ define
$$
\boldsymbol\psi^{(t + 1, v)}_i = Q^{(t + 1)}_i(\cdots, e_{t + 1 - v - 1}, e_{t + 1 - v}^\dagger, e_{t + 1 - v + 1},\cdots, e_{t + 1}).
$$
Then we have for each $l = 1,\cdots, d, j = 1,\cdots, pd$,
\begin{align*}
\Vert
\sum_{q = 1}^{pd}\mathcal{F}_{\mathcal{B}_l}(\boldsymbol\Sigma^{(0)})_{jq}\boldsymbol\epsilon^{(t+1)}_l\mathbf{z}^{(t)}_q
\Vert_{m/2}\leq \sum_{q = 1}^{pd}\vert\mathcal{F}_{\mathcal{B}_l}(\boldsymbol\Sigma^{(0)})_{jq}\vert\times\Vert\boldsymbol\epsilon^{(t+1)}_l\Vert_{m}\Vert\mathbf{z}^{(t)}_q\Vert_m\\
\leq C\max_{l = 1,\cdots, d}\vert\mathcal{F}_{\mathcal{B}_l}(\boldsymbol\Sigma^{(0)})\vert_{L_1}\leq C_1
\end{align*}
for a constant $C_1$. Moreover,
\begin{align*}
\Vert\boldsymbol\psi^{(t + 1)}_{(l - 1)pd + j} - \boldsymbol\psi^{(t + 1, v)}_{(l - 1)pd + j}\Vert_{m/2}
=\Vert
\sum_{q = 1}^{pd}\mathcal{F}_{\mathcal{B}_l}(\boldsymbol\Sigma^{(0)})_{jq}(\boldsymbol\epsilon^{(t+1)}_l\mathbf{z}^{(t)}_q - \boldsymbol\epsilon^{(t+1, v)}_l\mathbf{z}^{(t, v - 1)}_q)
\Vert_{m/2}\\
\leq C\sup_{q = 1,\cdots, pd, t\in\mathbf{Z}}\Vert\mathbf{z}_q^{(t)} - \mathbf{z}_q^{(t, v - 1)}\Vert_m + C\sup_{l = 1,\cdots, d, t\in\mathbf{Z}}
\Vert\boldsymbol\epsilon^{(t+1)}_l - \boldsymbol\epsilon^{(t+1, v)}_l\Vert_m.
\end{align*}
For $\boldsymbol\epsilon^{(t+1)}$, $\mathbf{z}^{(t)}$ are all $(m,\alpha)-$short range dependent random variables, these two formulas imply that
$\boldsymbol\psi^{(t + 1)}$ are $(m/2,\alpha)-$short range dependent random variables. Notice that
\begin{align*}
Cov\left(\sum_{t = p}^{T - 1}\boldsymbol\psi^{(t + 1)}_{(l_1 - 1)\times pd + j_1}, \sum_{t = p}^{T - 1}\boldsymbol\psi^{(t + 1)}_{(l_2 - 1)\times pd + j_2}\right)\\
= \sum_{t_1 = p}^{T - 1}\sum_{t_2 = p}^{T - 1}\sum_{q_1 = 1}^{pd}\sum_{q_2 = 1}^{pd}
\mathcal{F}_{\mathcal{B}_{l_1}}(\boldsymbol\Sigma^{(0)})_{j_1q_1}\mathcal{F}_{\mathcal{B}_{l_2}}(\boldsymbol\Sigma^{(0)})_{j_2q_2}\mathbf{E}\boldsymbol\epsilon_{l_1}^{(t_1 + 1)}\boldsymbol\epsilon_{l_2}^{(t_2 + 1)}
\mathbf{z}^{(t_1)}_{q_1}\mathbf{z}^{(t_2)}_{q_2} = T\boldsymbol\Lambda^{(l_1,l_2)}_{j_1j_2},
\end{align*}
see eq.\eqref{eq.cov_matrixX} for the definition of $\boldsymbol\Lambda^{(l_1, l_2)}$. Moreover, $\boldsymbol\psi^{(t)}_{(l - 1)\times pd + j} = 0$ if $j\not\in\mathcal{B}_l$, so
$\boldsymbol\psi^{(t)}$ only has $\sum_{l = 1}^d \vert\mathcal{B}_l\vert = O(d)$ non--zero elements, and
\begin{align*}
\max_{l_1,l_2 = 1,\cdots, d, j_1,j_2 = 1,\cdots, pd}\vert\frac{1}{T - p}\sum_{t_1 = p}^{T - 1}\sum_{t_2 = p}^{T - 1}\boldsymbol\psi^{(t_1 + 1)}_{(l_1 - 1)\times pd + j_1}\boldsymbol\psi^{(t_2 + 1)}_{(l_2 - 1)\times pd + j_2}K\left(\frac{t_1 - t_2}{k_T}\right)
- \frac{T}{T - p}\boldsymbol\Lambda^{(l_1, l_2)}_{j_1j_2}
\vert\\
= \max_{l_1,l_2 = 1,\cdots, j_1\in\mathcal{B}_{l_1}, j_2\in\mathcal{B}_{l_2}}\vert\frac{1}{T - p}\sum_{t_1 = p}^{T - 1}\sum_{t_2 = p}^{T - 1}\boldsymbol\psi^{(t_1 + 1)}_{(l_1 - 1)\times pd + j_1}\boldsymbol\psi^{(t_2 + 1)}_{(l_2 - 1)\times pd + j_2}K\left(\frac{t_1 - t_2}{k_T}\right)
- \frac{T}{T - p}\boldsymbol\Lambda^{(l_1, l_2)}_{j_1j_2}
\vert\\
= O_p\left(v_T + k_T\times T^{\frac{8\alpha_d}{m} - \frac{1}{2}}\right).
\end{align*}
Therefore,
\begin{equation}
\begin{aligned}
\max_{l_1,l_2 = 1,\cdots, d, j_1,j_2 = 1,\cdots, pd}\vert\frac{1}{T}\sum_{t_1 = p}^{T - 1}\sum_{t_2 = p}^{T - 1}\boldsymbol\psi^{(t_1 + 1)}_{(l_1 - 1)\times pd + j_1}\boldsymbol\psi^{(t_2 + 1)}_{(l_2 - 1)\times pd + j_2}K\left(\frac{t_1 - t_2}{k_T}\right)
- \boldsymbol\Lambda^{(l_1, l_2)}_{j_1j_2}
\vert\\
\leq \frac{p}{T}\max_{l_1, l_2 = 1,\cdots, d, j_1\in\mathcal{B}_{l_1}, j_2\in\mathcal{B}_{l_2}}\vert
\frac{1}{T - p}\sum_{t_1 = p}^{T - 1}\sum_{t_2 = p}^{T - 1}\boldsymbol\psi^{(t_1 + 1)}_{(l_1 - 1)\times pd + j_1}\boldsymbol\psi^{(t_2 + 1)}_{(l_2 - 1)\times pd+j_2}K\left(\frac{t_1 - t_2}{k_T}\right)
\vert\\
+ \max_{l_1,l_2 = 1,\cdots, d, j_1\in\mathcal{B}_{l_1}, j_2\in\mathcal{B}_{l_2}}\vert\frac{1}{T - p}\sum_{t_1 = p}^{T - 1}\sum_{t_2 = p}^{T - 1}\boldsymbol\psi^{(t_1 + 1)}_{(l_1 - 1)\times pd + j_1}\boldsymbol\psi^{(t_2 + 1)}_{(l_2 - 1)\times pd + j_2}K\left(\frac{t_1 - t_2}{k_T}\right)
- \frac{T}{T - p}\boldsymbol\Lambda^{(l_1, l_2)}_{j_1j_2}
\vert\\
+ \frac{p}{T - p}\max_{l_1,l_2 = 1,\cdots, d, j_1\in\mathcal{B}_{l_1}, j_2\in\mathcal{B}_{l_2}}\vert\boldsymbol\Lambda^{(l_1, l_2)}_{j_1j_2}\vert = O_p\left(v_T + k_T\times T^{\frac{8\alpha_d}{m} - \frac{1}{2}}\right).
\end{aligned}
\end{equation}
Since
\begin{align*}
\sum_{q = 1}^{pd}\mathcal{F}_{\mathcal{B}_l}(\boldsymbol\Sigma^{(0)})_{jq}\widehat{\boldsymbol\Theta}_{ql}^{(t)}
= \sum_{q = 1}^{pd}\mathcal{F}_{\mathcal{B}_l}(\boldsymbol\Sigma^{(0)})_{jq}\mathbf{z}_q^{(t)}\boldsymbol\epsilon_l^{(t + 1)}
- \sum_{s\in\mathcal{B}_l}\boldsymbol\omega_{js}^{(t,l)}(\widehat{\mathbf{S}}_{sl} - \mathbf{S}_{sl})\\
= \boldsymbol\psi^{(t + 1)}_{(l - 1)\times pd + j} - \sum_{s\in\mathcal{B}_l}\boldsymbol\omega_{js}^{(t,l)}(\widehat{\mathbf{S}}_{sl} - \mathbf{S}_{sl}),
\end{align*}
we have
\begin{align*}
\vert\frac{1}{T}\sum_{t_1 = p}^{T - 1}\sum_{t_2 = p}^{T - 1}K\left(\frac{t_1 - t_2}{k_T}\right)\times
\left(
\sum_{q = 1}^{pd}\mathcal{F}_{\mathcal{B}_{l_1}}(\boldsymbol\Sigma^{(0)})_{j_1q}\widehat{\boldsymbol\Theta}_{ql_1}^{(t)}
\right)
\times \left(
\sum_{q = 1}^{pd}\mathcal{F}_{\mathcal{B}_{l_2}}(\boldsymbol\Sigma^{(0)})_{j_2q}\widehat{\boldsymbol\Theta}_{ql_2}^{(t)}
\right)\\
- \frac{1}{T}\sum_{t_1 = p}^{T - 1}\sum_{t_2 = p}^{T - 1}K\left(\frac{t_1 - t_2}{k_T}\right)\boldsymbol\psi_{(l_1 - 1)\times pd +j_1}^{(t_1 + 1)}\boldsymbol\psi_{(l_2 - 1)\times pd + j_2}^{(t_2 + 1)}\vert\\
\leq
\vert\sum_{s\in\mathcal{B}_{l_1}}(\widehat{\mathbf{S}}_{sl_1} - \mathbf{S}_{sl_1})\frac{1}{T}\sum_{t_1 = p}^{T - 1}\sum_{t_2 = p}^{T - 1} K\left(\frac{t_1 - t_2}{k_T}\right)\boldsymbol\psi_{(l_2 - 1)\times pd + j_2}^{(t_2 + 1)}\boldsymbol\omega_{j_1s}^{(t_1, l_1)}\vert\\
+ \vert\sum_{s\in\mathcal{B}_{l_2}}(\widehat{\mathbf{S}}_{sl_2} - \mathbf{S}_{sl_2})\frac{1}{T}\sum_{t_1 = p}^{T - 1}\sum_{t_2 = p}^{T - 1}K\left(\frac{t_1 - t_2}{k_T}\right)\boldsymbol\psi_{(l_1 - 1)\times pd + j_1}^{(t_1 + 1)}\boldsymbol\omega_{j_2s}^{(t_2, l_2)}\vert\\
+ \vert
\sum_{s_1\in\mathcal{B}_{l_1}}\sum_{s_2\in\mathcal{B}_{l_2}}(\widehat{\mathbf{S}}_{s_1l_1} - \mathbf{S}_{s_1l_1})(\widehat{\mathbf{S}}_{s_2l_2} - \mathbf{S}_{s_2l_2})
\times \frac{1}{T}\sum_{t_1 = p}^{T - 1}\sum_{t_2 = p}^{T - 1}K\left(\frac{t_1 - t_2}{k_T}\right)\boldsymbol\omega_{j_1s_1}^{(t_1, l_1)}\boldsymbol\omega_{j_2s_2}^{(t_2, l_2)}
\vert\\
\leq \vert\mathcal{B}_{l_1}\vert\times \max_{s\in\mathcal{B}_{l_1}}\vert\widehat{\mathbf{S}}_{sl_1} - \mathbf{S}_{sl_1}\vert\times \max_{s\in\mathcal{B}_{l_1}}\vert\frac{1}{T}\sum_{t_1 = p}^{T - 1}\sum_{t_2 = p}^{T - 1} K\left(\frac{t_1 - t_2}{k_T}\right)\boldsymbol\psi_{(l_2 - 1)\times pd + j_2}^{(t_2 + 1)}\boldsymbol\omega_{j_1s}^{(t_1, l_1)}\vert\\
+ \vert\mathcal{B}_{l_2}\vert\times \max_{s\in\mathcal{B}_{l_2}}\vert\widehat{\mathbf{S}}_{sl_2} - \mathbf{S}_{sl_2}\vert\times \max_{s\in\mathcal{B}_{l_2}}\vert\frac{1}{T}\sum_{t_1 = p}^{T - 1}\sum_{t_2 = p}^{T - 1}K\left(\frac{t_1 - t_2}{k_T}\right)\boldsymbol\psi_{(l_1 - 1)\times pd + j_1}^{(t_1 + 1)}\boldsymbol\omega_{j_2s}^{(t_2, l_2)}
\vert\\
+ \vert\mathcal{B}_{l_1}\vert\times\vert\mathcal{B}_{l_2}\vert\times \max_{s_1\in\mathcal{B}_{l_1}}\vert\widehat{\mathbf{S}}_{s_1l_1} - \mathbf{S}_{s_1l_1}\vert
\times \max_{s_2\in\mathcal{B}_{l_2}}\vert\widehat{\mathbf{S}}_{s_2l_2} - \mathbf{S}_{s_2l_2}\vert\\
\times
\max_{s_1\in\mathcal{B}_{l_1}}\max_{s_2\in\mathcal{B}_{l_2}}\vert\frac{1}{T}\sum_{t_1 = p}^{T - 1}\sum_{t_2 = p}^{T - 1}K\left(\frac{t_1 - t_2}{k_T}\right)\boldsymbol\omega_{j_1s_1}^{(t_1, l_1)}\boldsymbol\omega_{j_2s_2}^{(t_2, l_2)}
\vert.
\end{align*}
Therefore,
\begin{align*}
\max_{l_1, l_2 = 1,\cdots, d, j_1, j_2 = 1,\cdots, pd}\vert\frac{1}{T}\sum_{t_1 = p}^{T - 1}\sum_{t_2 = p}^{T - 1}K\left(\frac{t_1 - t_2}{k_T}\right)\times
\left(
\sum_{q = 1}^{pd}\mathcal{F}_{\mathcal{B}_{l_1}}(\boldsymbol\Sigma^{(0)})_{j_1q}\widehat{\boldsymbol\Theta}_{ql_1}^{(t)}
\right)
\times \left(
\sum_{q = 1}^{pd}\mathcal{F}_{\mathcal{B}_{l_2}}(\boldsymbol\Sigma^{(0)})_{j_2q}\widehat{\boldsymbol\Theta}_{ql_2}^{(t)}
\right)\\
- \frac{1}{T}\sum_{t_1 = p}^{T - 1}\sum_{t_2 = p}^{T - 1}K\left(\frac{t_1 - t_2}{k_T}\right)\boldsymbol\psi_{(l_1 - 1)\times pd +j_1}^{(t_1 + 1)}\boldsymbol\psi_{(l_2 - 1)\times pd + j_2}^{(t_2 + 1)}\vert\\
= \max_{l_1, l_2 = 1,\cdots, d, j_1\in\mathcal{B}_{l_1}, j_2\in\mathcal{B}_{l_2}}\vert\frac{1}{T}\sum_{t_1 = p}^{T - 1}\sum_{t_2 = p}^{T - 1}K\left(\frac{t_1 - t_2}{k_T}\right)\times
\left(
\sum_{q = 1}^{pd}\mathcal{F}_{\mathcal{B}_{l_1}}(\boldsymbol\Sigma^{(0)})_{j_1q}\widehat{\boldsymbol\Theta}_{ql_1}^{(t)}
\right)
\times \left(
\sum_{q = 1}^{pd}\mathcal{F}_{\mathcal{B}_{l_2}}(\boldsymbol\Sigma^{(0)})_{j_2q}\widehat{\boldsymbol\Theta}_{ql_2}^{(t)}
\right)\\
- \frac{1}{T}\sum_{t_1 = p}^{T - 1}\sum_{t_2 = p}^{T - 1}K\left(\frac{t_1 - t_2}{k_T}\right)\boldsymbol\psi_{(l_1 - 1)\times pd +j_1}^{(t_1 + 1)}\boldsymbol\psi_{(l_2 - 1)\times pd + j_2}^{(t_2 + 1)}\vert\\
\leq C\max_{j = 1,\cdots, pd, l = 1,\cdots, d}\vert\widehat{\mathbf{S}}_{jl} - \mathbf{S}_{jl}\vert\times \max_{l_1, l_2 = 1,\cdots, d, j_1,s\in\mathcal{B}_{l_1}, j_2\in\mathcal{B}_{l_2}}\vert\frac{1}{T}\sum_{t_1 = p}^{T - 1}\sum_{t_2 = p}^{T - 1} K\left(\frac{t_1 - t_2}{k_T}\right)\boldsymbol\psi_{(l_2 - 1)\times pd + j_2}^{(t_2 + 1)}\boldsymbol\omega_{j_1s}^{(t_1, l_1)}\vert\\
+ C\max_{j = 1,\cdots, pd, l = 1,\cdots, d}\vert\widehat{\mathbf{S}}_{jl} - \mathbf{S}_{jl}\vert\times \max_{l_1,l_2 = 1,\cdots, d, j_1\in\mathcal{B}_{l_1}, j_2,s\in\mathcal{B}_2}\vert\frac{1}{T}\sum_{t_1 = p}^{T - 1}\sum_{t_2 = p}^{T - 1}K\left(\frac{t_1 - t_2}{k_T}\right)\boldsymbol\psi_{(l_1 - 1)\times pd + j_1}^{(t_1 + 1)}\boldsymbol\omega_{j_2s}^{(t_2, l_2)}\vert\\
+ C\max_{j = 1,\cdots, pd, l = 1,\cdots, d}\vert\widehat{\mathbf{S}}_{jl} - \mathbf{S}_{jl}\vert^2\times
\max_{l_1, l_2 = 1,\cdots, d, s_1,j_1\in\mathcal{B}_{l_1}, s_2,j_2\in\mathcal{B}_{l_2}}\vert\frac{1}{T}\sum_{t_1 = p}^{T - 1}\sum_{t_2 = p}^{T - 1}K\left(\frac{t_1 - t_2}{k_T}\right)\boldsymbol\omega_{j_1s_1}^{(t_1, l_1)}\boldsymbol\omega_{j_2s_2}^{(t_2, l_2)}
\vert.
\end{align*}

Notice that the matrix $\left(K\left(\frac{t_1 - t_2}{k_T}\right)\right)_{t_1, t_2 = p,\cdots, T - 1}$ is a Toeplitz matrix. Therefore, from section 0.9.7 in \cite{MR2978290}, for any given $l_1, l_2, j_1, j_2, s$,
\begin{align*}
\Vert\ \vert\frac{1}{T}\sum_{t_1 = p}^{T - 1}\sum_{t_2 = p}^{T - 1} K\left(\frac{t_1 - t_2}{k_T}\right)\boldsymbol\psi_{(l_2 - 1)\times pd + j_2}^{(t_2 + 1)}\boldsymbol\omega_{j_1s}^{(t_1, l_1)}\vert\ \Vert_{m/4}\\
\leq \frac{Ck_T}{T}\times \Vert\ \sqrt{\sum_{t = p}^{T -  1}\boldsymbol\psi_{(l_2 - 1)\times pd + j_2}^{(t + 1)2}}
\times \sqrt{\sum_{t = p}^{T -  1}\boldsymbol\omega_{j_1s}^{(t, l_1)2}}\ \Vert_{m/4}\\
\leq \frac{Ck_T}{T}\times \sqrt{\Vert\sum_{t = p}^{T -  1}\boldsymbol\psi_{(l_2 - 1)\times pd + j_2}^{(t + 1)2}\Vert_{m/4}}\times \sqrt{\Vert\sum_{t = p}^{T - 1}\boldsymbol\omega_{j_1s}^{(t, l_1)2}\Vert_{m/4}}\\
\leq \frac{Ck_T}{T}\times \sqrt{\sum_{t = p}^{T -  1}\Vert\boldsymbol\psi_{(l_2 - 1)\times pd + j_2}^{(t + 1)}\Vert_{m/2}^2}\times \sqrt{\sum_{t = p}^{T - 1}\Vert\boldsymbol\omega_{j_1s}^{(t, l_1)}\Vert_{m/2}^2}\leq C_1k_T\\
\text{and } \Vert\ \vert\frac{1}{T}\sum_{t_1 = p}^{T - 1}\sum_{t_2 = p}^{T - 1}K\left(\frac{t_1 - t_2}{k_T}\right)\boldsymbol\omega_{j_1s_1}^{(t_1, l_1)}\boldsymbol\omega_{j_2s_2}^{(t_2, l_2)}
\vert\ \Vert_{m/4}
\leq \frac{Ck_T}{T}\Vert\sum_{t = p}^{T - 1}\boldsymbol\omega_{j_1s_1}^{(t, l_1)2}\Vert_{m/4}^{1/2}\times \Vert\sum_{t = p}^{T - 1}\boldsymbol\omega_{j_2s_2}^{(t, l_2)2}\Vert_{m/4}^{1/2}\\
\leq \frac{Ck_T}{T}\sqrt{\sum_{t = p}^{T - 1}\Vert\boldsymbol\omega_{j_1s_1}^{(t, l_1)}\Vert_{m/2}^2}\times \sqrt{\sum_{t = p}^{T - 1}\Vert\boldsymbol\omega_{j_2s_2}^{(t, l_2)}\Vert_{m/2}^2}\leq C_1k_T
\end{align*}
for a constant $C_1$, which implies
\begin{align*}
\max_{l_1, l_2 = 1,\cdots, d, j_1, s\in\mathcal{B}_{l_1}, j_2\in\mathcal{B}_{l_2}}\vert\frac{1}{T}\sum_{t_1 = p}^{T - 1}\sum_{t_2 = p}^{T - 1} K\left(\frac{t_1 - t_2}{k_T}\right)\boldsymbol\psi_{(l_2 - 1)\times pd + j_2}^{(t_2 + 1)}\boldsymbol\omega_{j_1s}^{(t_1, l_1)}\vert = O_p\left(T^{\frac{8\alpha_d}{m}}\times k_T\right),\\
\max_{l_1,l_2 = 1,\cdots, d, j_1\in\mathcal{B}_{l_1}, j_2,s\in\mathcal{B}_2}\vert\frac{1}{T}\sum_{t_1 = p}^{T - 1}\sum_{t_2 = p}^{T - 1}K\left(\frac{t_1 - t_2}{k_T}\right)\boldsymbol\psi_{(l_1 - 1)\times pd + j_1}^{(t_1 + 1)}\boldsymbol\omega_{j_2s}^{(t_2, l_2)}\vert= O_p\left(T^{\frac{8\alpha_d}{m}}\times k_T\right),\\
\text{and }\max_{l_1, l_2 = 1,\cdots, d, s_1,j_1\in\mathcal{B}_{l_1}, s_2,j_2\in\mathcal{B}_{l_2}}\vert\frac{1}{T}\sum_{t_1 = p}^{T - 1}\sum_{t_2 = p}^{T - 1}K\left(\frac{t_1 - t_2}{k_T}\right)\boldsymbol\omega_{j_1s_1}^{(t_1, l_1)}\boldsymbol\omega_{j_2s_2}^{(t_2, l_2)}
\vert  = O_p\left(T^{\frac{8\alpha_d}{m}}\times k_T\right).
\end{align*}
As a result, we have
\begin{equation}
\begin{aligned}
\max_{l_1, l_2 = 1,\cdots, d, j_1, j_2 = 1,\cdots, pd}\vert\frac{1}{T}\sum_{t_1 = p}^{T - 1}\sum_{t_2 = p}^{T - 1}K\left(\frac{t_1 - t_2}{k_T}\right)\times
\left(
\sum_{q = 1}^{pd}\mathcal{F}_{\mathcal{B}_{l_1}}(\boldsymbol\Sigma^{(0)})_{j_1q}\widehat{\boldsymbol\Theta}_{ql_1}^{(t)}
\right)
\times \left(
\sum_{q = 1}^{pd}\mathcal{F}_{\mathcal{B}_{l_2}}(\boldsymbol\Sigma^{(0)})_{j_2q}\widehat{\boldsymbol\Theta}_{ql_2}^{(t)}
\right)\\
- \frac{1}{T}\sum_{t_1 = p}^{T - 1}\sum_{t_2 = p}^{T - 1}K\left(\frac{t_1 - t_2}{k_T}\right)\boldsymbol\psi_{(l_1 - 1)\times pd +j_1}^{(t_1 + 1)}\boldsymbol\psi_{(l_2 - 1)\times pd + j_2}^{(t_2 + 1)}\vert\\
= O_p\left(T^{\frac{8\alpha_d}{m} - \frac{1}{2}}\times k_T\times \sqrt{\log(T)}\right)
\end{aligned}
\end{equation}
Form eq.\eqref{eq.delta_Sigma}, $\max_{l = 1,\cdots, d}\vert\mathcal{F}_{\mathcal{B}_l}(\widehat{\boldsymbol\Sigma}^{(0)}) - \mathcal{F}_{\mathcal{B}_l}(\boldsymbol\Sigma^{(0)})\vert_{L_1}
= O_p\left(T^{\frac{4\alpha_d}{m} - \frac{1}{2}}\right)$. Therefore, define the terms $\boldsymbol\xi_{jq}^{(l)} = \mathcal{F}_{\mathcal{B}_l}(\widehat{\boldsymbol\Sigma}^{(0)})_{jq} - \mathcal{F}_{\mathcal{B}_l}(\boldsymbol\Sigma^{(0)})_{jq}$, we have
\begin{align*}
\sum_{q = 1}^{pd}\mathcal{F}_{\mathcal{B}_l}(\widehat{\boldsymbol\Sigma}^{(0)})_{jq}\widehat{\boldsymbol\Theta}_{ql}^{(t)}
= \sum_{q\in\mathcal{B}_l}\mathcal{F}_{\mathcal{B}_l}(\boldsymbol\Sigma^{(0)})_{jq}\widehat{\boldsymbol\Theta}_{ql}^{(t)} + \sum_{q\in\mathcal{B}_l}\boldsymbol\xi_{jq}^{(l)}\widehat{\boldsymbol\Theta}_{ql}^{(t)}
\end{align*}
and
\begin{align*}
\vert\frac{1}{T}\sum_{t_1 = p}^{T - 1}\sum_{t_2 = p}^{T - 1}K\left(\frac{t_1 - t_2}{k_T}\right)\times
\left(
\sum_{q = 1}^{pd}\mathcal{F}_{\mathcal{B}_{l_1}}(\widehat{\boldsymbol\Sigma}^{(0)})_{j_1q}\widehat{\boldsymbol\Theta}_{ql_1}^{(t)}
\right)
\times \left(
\sum_{q = 1}^{pd}\mathcal{F}_{\mathcal{B}_{l_2}}(\widehat{\boldsymbol\Sigma}^{(0)})_{j_2q}\widehat{\boldsymbol\Theta}_{ql_2}^{(t)}
\right)\\
- \frac{1}{T}\sum_{t_1 = p}^{T - 1}\sum_{t_2 = p}^{T - 1}K\left(\frac{t_1 - t_2}{k_T}\right)\times
\left(
\sum_{q = 1}^{pd}\mathcal{F}_{\mathcal{B}_{l_1}}(\boldsymbol\Sigma^{(0)})_{j_1q}\widehat{\boldsymbol\Theta}_{ql_1}^{(t)}
\right)
\times \left(
\sum_{q = 1}^{pd}\mathcal{F}_{\mathcal{B}_{l_2}}(\boldsymbol\Sigma^{(0)})_{j_2q}\widehat{\boldsymbol\Theta}_{ql_2}^{(t)}
\right)\vert\\
\leq \vert
\sum_{q_1\in\mathcal{B}_{l_1}}\sum_{q_2\in\mathcal{B}_{l_2}}\boldsymbol\xi_{j_1q_1}^{(l_1)}\mathcal{F}_{\mathcal{B}_{l_2}}(\boldsymbol\Sigma^{(0)})_{j_2q_2}
\times \frac{1}{T}\sum_{t_1 = p}^{T - 1}\sum_{t_2 = p}^{T - 1} K\left(\frac{t_1 - t_2}{k_T}\right)\widehat{\boldsymbol\Theta}_{q_1l_1}^{(t_1)}\widehat{\boldsymbol\Theta}_{q_2l_2}^{(t_2)}
\vert\\
+ \vert
\sum_{q_1\in\mathcal{B}_{l_1}}\sum_{q_2\in\mathcal{B}_{l_2}}\mathcal{F}_{\mathcal{B}_{l_2}}(\boldsymbol\Sigma^{(0)})_{j_1q_1}\boldsymbol\xi_{j_2q_2}^{(l_2)}
\times \frac{1}{T}\sum_{t_1 = p}^{T - 1}\sum_{t_2 = p}^{T - 1} K\left(\frac{t_1 - t_2}{k_T}\right)\widehat{\boldsymbol\Theta}_{q_1l_1}^{(t_1)}\widehat{\boldsymbol\Theta}_{q_2l_2}^{(t_2)}
\vert\\
+\vert
\sum_{q_1\in\mathcal{B}_{l_1}}\sum_{q_2\in\mathcal{B}_{l_2}}\boldsymbol\xi_{j_1q_1}^{(l_1)}\boldsymbol\xi_{j_2q_2}^{(l_2)}\times \frac{1}{T}\sum_{t_1 = p}^{T - 1}\sum_{t_2 = p}^{T - 1} K\left(\frac{t_1 - t_2}{k_T}\right)\widehat{\boldsymbol\Theta}_{q_1l_1}^{(t_1)}\widehat{\boldsymbol\Theta}_{q_2l_2}^{(t_2)}
\vert
\end{align*}
From Cauchy - Schwarz inequality
\begin{align*}
\sum_{t = p}^{T - 1}\widehat{\boldsymbol\Theta}_{ql}^{(t)2} = \sum_{t = p}^{T - 1}\mathbf{z}^{(t)2}_q\times\left(\boldsymbol\epsilon^{(t + 1)}_l - \sum_{s = 1}^{pd}\mathbf{z}^{(t)}_s(\widehat{\mathbf{S}}_{sl} - \mathbf{S}_{sl})\right)^2\\
\leq 2\sum_{t = p}^{T - 1}\mathbf{z}^{(t)2}_q\boldsymbol\epsilon^{(t + 1)2}_l + 2\sum_{t = p}^{T - 1}\mathbf{z}^{(t)2}_q\sum_{s\in\mathcal{B}_l}\mathbf{z}_s^{(t)2}\times \sum_{s\in\mathcal{B}_l}(\widehat{\mathbf{S}}_{sl} - \mathbf{S}_{sl})^2.
\end{align*}
For
\begin{align*}
\Vert\sum_{t = p}^{T - 1}\mathbf{z}^{(t)2}_q\boldsymbol\epsilon^{(t + 1)2}_l\Vert_{m/4}\leq \sum_{t = p}^{T - 1}\Vert\mathbf{z}_q\Vert_m^2\times\Vert\boldsymbol\epsilon_l^{(t + 1)}\Vert_m^2\leq CT\\
\Rightarrow \max_{l = 1,\cdots, d, q\in\mathcal{B}_l}\sum_{t = p}^{T - 1}\mathbf{z}^{(t)2}_q\boldsymbol\epsilon^{(t + 1)2}_l = O_p\left(T^{\frac{4\alpha_d}{m} + 1}\right)
\end{align*}
and
\begin{align*}
\Vert
\sum_{t = p}^{T - 1}\mathbf{z}^{(t)2}_q\sum_{s\in\mathcal{B}_l}\mathbf{z}_s^{(t)2}
\Vert_{m/4}\leq \sum_{t = p}^{T - 1}\sum_{s\in\mathcal{B}_l}\Vert\mathbf{z}^{(t)}_q\Vert_{m}^2\Vert\mathbf{z}_s^{(t)}\Vert_m^2\leq CT\\
\Rightarrow \max_{l = 1,\cdots, d, q\in\mathcal{B}_l}\sum_{t = p}^{T - 1}\mathbf{z}^{(t)2}_q\sum_{s\in\mathcal{B}_l}\mathbf{z}_s^{(t)2} = O_p\left(T^{\frac{4\alpha_d}{m} + 1} \right).
\end{align*}
From eq. \eqref{eq.S_max_bound},we have
\begin{align*}
\max_{l = 1,\cdots, d, q\in\mathcal{B}_l}\sum_{t = p}^{T - 1}\widehat{\boldsymbol\Theta}_{ql}^{(t)2}
\leq 2\max_{l = 1,\cdots, d, q\in\mathcal{B}_l}\sum_{t = p}^{T - 1}\mathbf{z}^{(t)2}_q\boldsymbol\epsilon^{(t + 1)2}_l\\
 + 2\max_{l = 1,\cdots, d, q\in\mathcal{B}_l}\sum_{t = p}^{T - 1}\mathbf{z}^{(t)2}_q\sum_{s\in\mathcal{B}_l}\mathbf{z}_s^{(t)2}\times \max_{l = 1,\cdots, d, q\in\mathcal{B}_l}\sum_{s\in\mathcal{B}_l}(\widehat{\mathbf{S}}_{sl} - \mathbf{S}_{sl})^2\\
 = O_p\left(T^{\frac{4\alpha_d}{m} + 1}\right).
\end{align*}
Therefore, from section 0.9.7 in \cite{MR2978290}
\begin{align*}
\max_{l_1,l_2 = 1,\cdots, d, q_1\in\mathcal{B}_{l_1}, q_2\in\mathcal{B}_{l_2}}\vert\frac{1}{T}\sum_{t_1 = p}^{T - 1}\sum_{t_2 = p}^{T - 1} K\left(\frac{t_1 - t_2}{k_T}\right)\widehat{\boldsymbol\Theta}_{q_1l_1}^{(t_1)}\widehat{\boldsymbol\Theta}_{q_2l_2}^{(t_2)}\vert\\
\leq \frac{Ck_T}{T}\times \sqrt{\sum_{t = p}^{T - 1}\widehat{\boldsymbol\Theta}_{q_1l_1}^{(t)2}}\times \sqrt{\sum_{t = p}^{T - 1}\widehat{\boldsymbol\Theta}_{q_2l_2}^{(t_2)2}}
= O_p\left(k_T\times T^{\frac{4\alpha_d}{m}}\right),
\end{align*}
and
\begin{align*}
\max_{l_1,l_2 = 1,\cdots, d, j_1,j_2 = 1,\cdots, pd}\vert
\sum_{q_1\in\mathcal{B}_{l_1}}\sum_{q_2\in\mathcal{B}_{l_2}}\mathcal{F}_{\mathcal{B}_{l_2}}(\boldsymbol\Sigma^{(0)})_{j_1q_1}\boldsymbol\xi_{j_2q_2}^{(l_2)}
\times \frac{1}{T}\sum_{t_1 = p}^{T - 1}\sum_{t_2 = p}^{T - 1} K\left(\frac{t_1 - t_2}{k_T}\right)\widehat{\boldsymbol\Theta}_{q_1l_1}^{(t_1)}\widehat{\boldsymbol\Theta}_{q_2l_2}^{(t_2)}
\vert\\
\leq C\max_{l = 1,\cdots, d, q\in\mathcal{B}_l, j  =1,\cdots, pd}\vert\boldsymbol\xi_{j_2q_2}^{(l_2)}\vert
\times \max_{l_1, l_2 = 1,\cdots, d, q_1\in\mathcal{B}_{l_1}, q_2\in\mathcal{B}_{l_2}}\vert\frac{1}{T}\sum_{t_1 = p}^{T - 1}\sum_{t_2 = p}^{T - 1} K\left(\frac{t_1 - t_2}{k_T}\right)\widehat{\boldsymbol\Theta}_{q_1l_1}^{(t_1)}\widehat{\boldsymbol\Theta}_{q_2l_2}^{(t_2)}\vert\\
= O_p\left(k_T\times T^{\frac{8\alpha_d}{m} - \frac{1}{2}}\right).
\end{align*}
Moreover,
\begin{align*}
\max_{j_1,j_2 = 1,\cdots, pd, l_1, l_2 = 1,\cdots, d}\vert
\sum_{q_1\in\mathcal{B}_{l_1}}\sum_{q_2\in\mathcal{B}_{l_2}}\boldsymbol\xi_{j_1q_1}^{(l_1)}\boldsymbol\xi_{j_2q_2}^{(l_2)}\times \frac{1}{T}\sum_{t_1 = p}^{T - 1}\sum_{t_2 = p}^{T - 1} K\left(\frac{t_1 - t_2}{k_T}\right)\widehat{\boldsymbol\Theta}_{q_1l_1}^{(t_1)}\widehat{\boldsymbol\Theta}_{q_2l_2}^{(t_2)}
\vert\\
\leq C\max_{j = 1,\cdots, pd, l = 1,\cdots, d, q\in\mathcal{B}_l}\vert\boldsymbol\xi_{jq}^{(l)}\vert^2\times\max_{l_1, l_2 = 1,\cdots, d, q_1\in\mathcal{B}_{l_1}, q_2\in\mathcal{B}_{l_2}}\vert
\frac{1}{T}\sum_{t_1 = p}^{T - 1}\sum_{t_2 = p}^{T - 1} K\left(\frac{t_1 - t_2}{k_T}\right)\widehat{\boldsymbol\Theta}_{q_1l_1}^{(t_1)}\widehat{\boldsymbol\Theta}_{q_2l_2}^{(t_2)}
\vert\\
= O_p\left(k_T\times T^{\frac{12\alpha_d}{m} - 1}\right).
\end{align*}
Therefore,
\begin{equation}
\begin{aligned}
\max_{j_1,j_2 = 1,\cdots, pd, l_1, l_2 = 1,\cdots, d}\vert\frac{1}{T}\sum_{t_1 = p}^{T - 1}\sum_{t_2 = p}^{T - 1}K\left(\frac{t_1 - t_2}{k_T}\right)\times
\left(
\sum_{q = 1}^{pd}\mathcal{F}_{\mathcal{B}_{l_1}}(\widehat{\boldsymbol\Sigma}^{(0)})_{j_1q}\widehat{\boldsymbol\Theta}_{ql_1}^{(t)}
\right)
\times \left(
\sum_{q = 1}^{pd}\mathcal{F}_{\mathcal{B}_{l_2}}(\widehat{\boldsymbol\Sigma}^{(0)})_{j_2q}\widehat{\boldsymbol\Theta}_{ql_2}^{(t)}
\right)\\
- \frac{1}{T}\sum_{t_1 = p}^{T - 1}\sum_{t_2 = p}^{T - 1}K\left(\frac{t_1 - t_2}{k_T}\right)\times
\left(
\sum_{q = 1}^{pd}\mathcal{F}_{\mathcal{B}_{l_1}}(\boldsymbol\Sigma^{(0)})_{j_1q}\widehat{\boldsymbol\Theta}_{ql_1}^{(t)}
\right)
\times \left(
\sum_{q = 1}^{pd}\mathcal{F}_{\mathcal{B}_{l_2}}(\boldsymbol\Sigma^{(0)})_{j_2q}\widehat{\boldsymbol\Theta}_{ql_2}^{(t)}
\right)\vert\\
= O_p\left(k_T\times T^{\frac{8\alpha_d}{m} - \frac{1}{2}}\right).
\end{aligned}
\end{equation}
and we prove eq.\eqref{eq.convergence_covariance}.

According to algorithm \ref{algorithm.bootstrap}, $\boldsymbol\Delta^*_{jl}$ have joint normal distribution conditional
on the observations $\mathbf{x}^{(t)}$. From lemma B.1 in \cite{zhangPolitis}, we have for sufficiently large $T$
\begin{align*}
Prob^*\left(\max_{j = 1,\cdots, pd, l = 1,\cdots, d}\vert\boldsymbol\Delta^*_{jl}\vert\leq x\right) =
Prob^*\left(\max_{l = 1,\cdots, d, j \in \mathcal{B}_l}\vert\boldsymbol\Delta^*_{jl}\vert\leq x\right)\\
\Rightarrow
\sup_{x\in\mathbf{R}}\vert Prob^*\left(\max_{j = 1,\cdots, pd, l = 1,\cdots, d}\vert\boldsymbol\Delta^*_{jl}\vert\leq x\right)
- Prob\left(\max_{l = 1,\cdots, d, j\in\mathcal{B}_l}\vert\boldsymbol\Gamma_{jl}\vert\leq x\right)
\vert\\
\leq C(1 + \log^3(T))\eta^{1/3} + \eta^{1/6},
\end{align*}
here $\eta = \max_{j = 1,\cdots, pd, l = 1,\cdots, d}\vert\mathbf{E}^*\boldsymbol\Delta_{j_1l_1}^*\boldsymbol\Delta_{j_2l_2}^* -
\boldsymbol\Lambda^{(l_1, l_2)}_{j_1 j_2}\vert $, which has order
$=
O_p\left(v_T + k_T\sqrt{\log(T)}T^{\frac{8\alpha_d}{m} - \frac{1}{2}}\right)$. This observation proves eq.\eqref{eq.boot_consistency}.

\end{proof}

\section{Extra numerical simulation results}
This section demonstrates additional simulation results not included in the manuscript. Similar to Figure \ref{figure.performance_sec},
the performance of the post--selection estimator is comparable to that of the threshold Lasso. It remains effective even when the model
selection algorithms cannot select the optimal hyper-parameters, demonstrating the effectiveness of the post -- selection estimator
in dealing with practical data.
\begin{figure}[htbp]
  \centering
  \subfigure[Independent]{
    \includegraphics[width = 7cm]{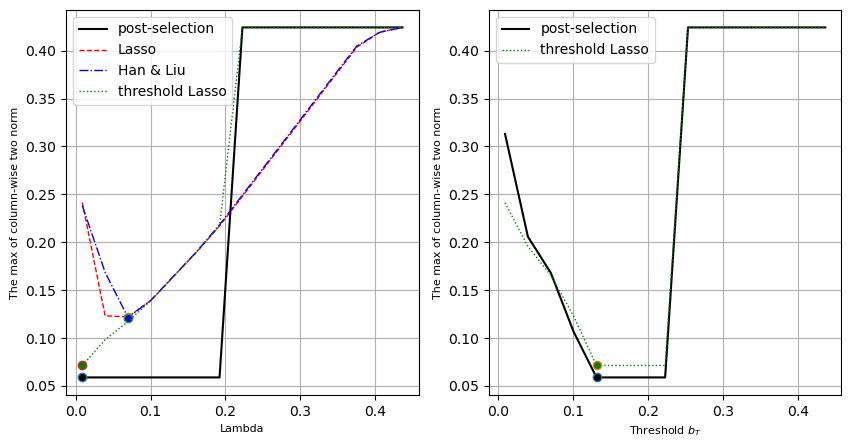}
    \label{figure.a1Ind}
  }
  \subfigure[Product normal] {
    \includegraphics[width=7cm]{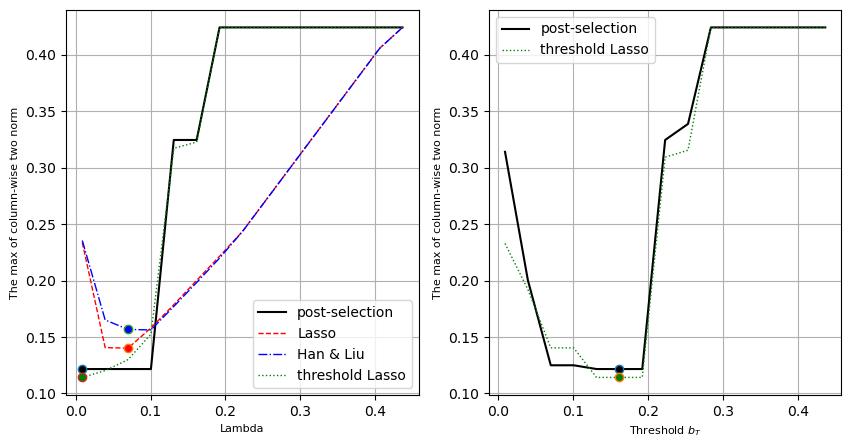}
    \label{figure.a1Prd}
  }
  \subfigure[Non-stationary]{
    \includegraphics[width=7cm]{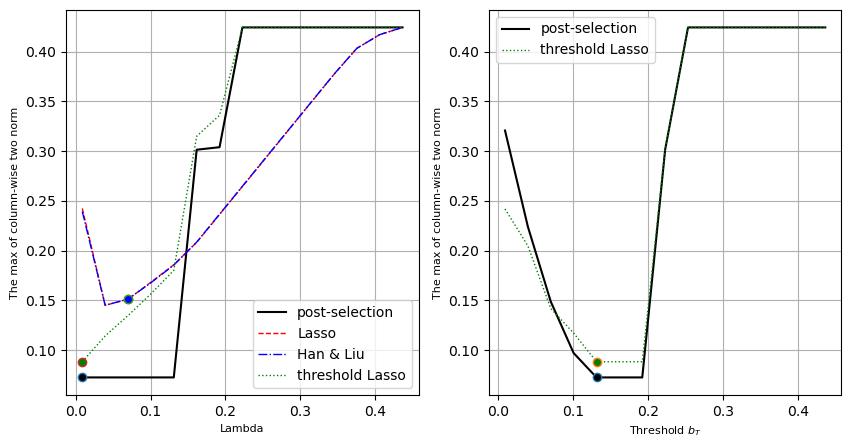}
    \label{figure.a1Non}
  }
  \caption{The performance of various algorithms concerning different choices of hyperparameters.
  Figure \ref{figure.a1Ind}, Figure \ref{figure.a1Prd}, and Figure \ref{figure.a1Non}
  respectively demonstrate the performance of algorithms when the innovations are independent, product normal,
  and non-stationary. The observations are generated using the AR(1) model.}
  \label{figure.performance_a_1}
\end{figure}

\begin{figure}[htbp]
  \centering
  \subfigure[Independent]{
    \includegraphics[width = 7cm]{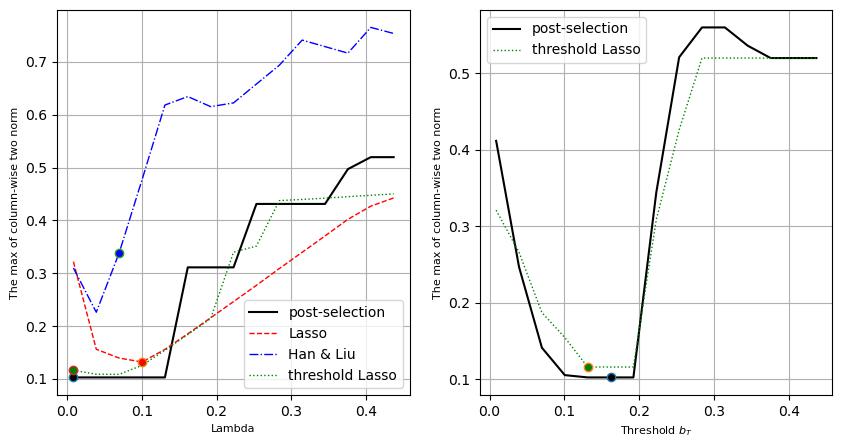}
    \label{figure.a2Ind}
  }
  \subfigure[Product normal] {
    \includegraphics[width=7cm]{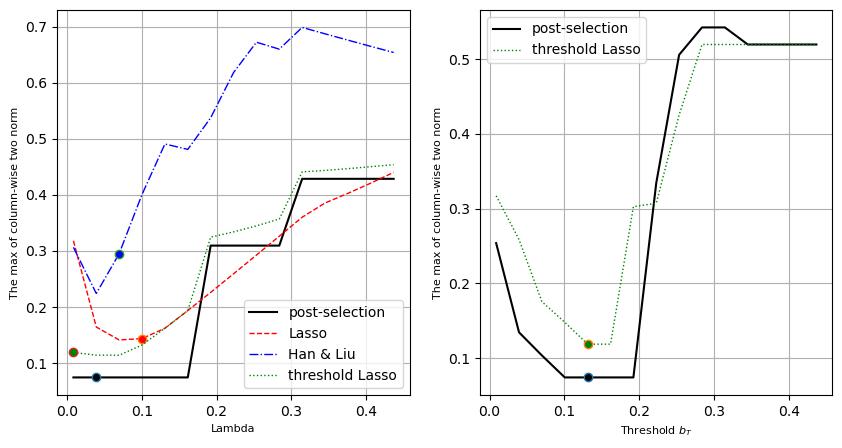}
    \label{figure.a2Prd}
  }
  \subfigure[Non-stationary]{
    \includegraphics[width=7cm]{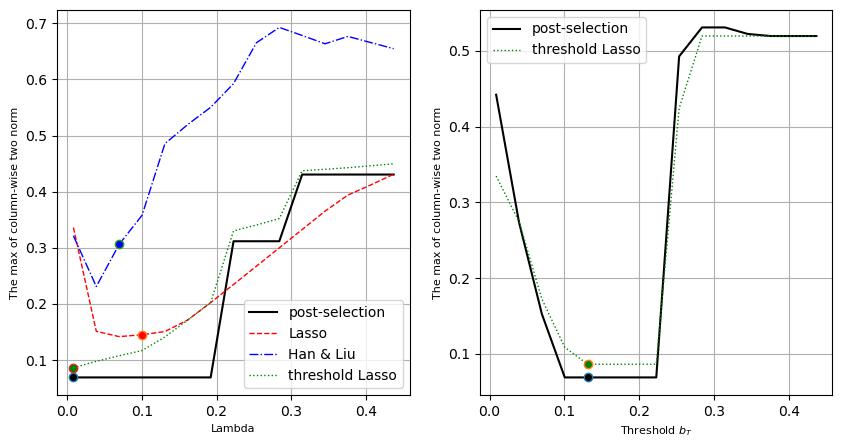}
    \label{figure.a2Non}
  }
  \caption{The performance of various algorithms concerning different choices of hyperparameters.
  Figure \ref{figure.a2Ind}, Figure \ref{figure.a2Prd}, and Figure \ref{figure.a2Non}
  respectively demonstrate the performance of algorithms when the innovations are independent, product normal,
  and non-stationary. The observations are generated using the AR(2) model.}
  \label{figure.performance_a_2}
\end{figure}

\end{document}